\documentclass[a4paper,oneside,times,draft,preprint]{elsarticle}
\usepackage[british]{babel}

\makeatletter
\def\ps@pprintTitle{%
   \let\@oddhead\@empty
   \let\@evenhead\@empty
   \def\@oddfoot{\reset@font\itshape\footnotesize\hfill\today}%
   \let\@evenfoot\@oddfoot}
\makeatother

\usepackage{xcolor}

\usepackage{amsmath, amsthm, amssymb}
\usepackage[a4paper]{geometry}

\usepackage{mathtools}
\mathtoolsset{showonlyrefs}

\usepackage[textsize=tiny]{todonotes}

\usepackage[shortlabels]{enumitem}

\usepackage[ruled,vlined]{algorithm2e}
\SetKwInput{KwInput}{Input}       

\usepackage{hyperref}
\hypersetup{final}

\newtheorem{theorem}{Theorem}[section]
\newtheorem{lemma}[theorem]{Lemma}
\newtheorem{proposition}[theorem]{Proposition}

\newtheorem{definition}[theorem]{Definition}
\theoremstyle{definition}
\newtheorem{assumption}[theorem]{Assumption}

\newtheorem{remark}[theorem]{Remark}

\newcommand{\E}{\ensuremath{\mathbb{E}}}  
  
\newcommand{\R}{\ensuremath{\mathbb{R}}}

\newcommand{\N}{\ensuremath{\mathbb{N}}}

\newcommand{\Prob}{\ensuremath{\mathbb{P}}}

\DeclareMathOperator*{\argmin}{arg\,min}
\newcommand{\eps}{\varepsilon}

\begin{document}

\begin{frontmatter}

\title{Mirror descent for stochastic control problems with measure-valued controls}
\ifpdf
\hypersetup{
pdftitle={Mirror descent for stochastic control problems},
pdfauthor={B. Kerimkulov, D. \v{S}i\v{s}ka, and L. Szpruch}
}
\fi

\date{\today}
\author[label1]{Bekzhan Kerimkulov}
\ead{B.Kerimkulov@ed.ac.uk}
\author[label1]{David \v{S}i\v{s}ka}
\ead{D.Siska@ed.ac.uk}
\author[label1,label2]{{\L}ukasz  Szpruch}
\ead{L.Szpruch@ed.ac.uk}
\author[label3]{Yufei Zhang}
\ead{yufei.zhang@imperial.ac.uk}

\affiliation[label1]{organization={School of Mathematics, University of Edinburgh}}
\affiliation[label2]{organization={Alan Turing Institute}}
\affiliation[label3]{organization={Imperial College London}}

\begin{abstract}
This paper studies the convergence of the mirror descent algorithm for finite horizon stochastic control problems with measure-valued control processes. 
The control objective involves a convex regularisation function, denoted as $h$, with regularisation strength determined by the weight $\tau\ge 0$.
The setting covers regularised relaxed control problems. 
Under suitable conditions, we establish the relative smoothness and convexity of the control objective with respect to the Bregman divergence of $h$, and  prove linear convergence of the algorithm for $\tau=0$ and exponential convergence for $\tau>0$. 
The results apply to common regularisers including relative entropy, $\chi^2$-divergence, and entropic Wasserstein costs.
This validates recent reinforcement learning heuristics that adding regularisation accelerates the convergence of gradient methods. 
The proof exploits careful regularity estimates of backward stochastic differential equations in the bounded mean oscillation norm.  

\end{abstract}



\end{frontmatter}

\section{Introduction}
\label{sec:intro}

This paper  designs a gradient descent algorithm for a finite horizon stochastic control problems where the controlled system is  a  
$\R^d$-valued diffusion process, 
and both the drift and diffusion coefficients of the state process are controlled by measure-valued control processes.
These measure-valued controls   include  as a special   case   relaxed controls, which are important for several reasons. 
Employing relaxed controls provides a convenient approach to guarantee the existence of a minimiser for the control problem, even when an optimal classical control may not exist \cite{nicole1987compactification}.
The  relaxed   control perspective also offers a principled framework for designing efficient  algorithms. 
It has been observed that 
the use of relaxed controls, along with  additional regularization terms,
improves algorithm stability and efficiency \cite{haarnoja2017reinforcement,wang2020reinforcement,wang2018exploration, wang2020continuous}.
The regularised relaxed control formulation further ensures the    optimal controls   are   stable with respect to model perturbations \cite{reisinger2021regularity}. This stability of optimal controls proves to be essential for optimising the sample efficiency of associated learning algorithms~\cite{basei2022logarithmic, szpruch2021exploration,
szpruch2022optimal, tang2022exploratory, guo2023reinforcement}.

In the sequel, we will give a heuristic
overview of the problem, the algorithm and our results. All notation and assumptions will be
stated precisely in Section \ref{sec main}.

\subsection*{Problem formulation}
Let 
$T\in (0,\infty)$ be a given terminal time, 
let 
$(\Omega, \mathcal{F}, \mathbb{F},  \Prob)$ be  a filtered
probability space satisfying the usual conditions, 
on which a $d'$-dimensional standard Brownian motion 
$W=(W_t)_{t\in [0,T]}$ is defined. 
Let $\mathcal P(A)$ be the space of probability measures   on a  separable metric space $A$, 
and let $\mathfrak C$ be a  convex subset of $\mathcal P(A)$.
We consider the admissible control space $\mathcal{A_{\mathfrak C}}$
containing $\mathbb F$-progressively measurable processes  
$\pi:\Omega\times [0,T]\to \mathfrak C$ satisfying suitable integrability conditions,  which will be precisely formulated later in~\eqref{eq:admissible_AC}. 
For each $\pi\in \mathcal{A_{\mathfrak C}}$, consider the following controlled dynamics: 
\begin{equation}\label{sde}
dX_s(\pi)=b_s(X_s(\pi),\pi_s)\,ds+\sigma_s(X_s(\pi),\pi_s)\,dW_s\,,\,\, s\in [0,T]\,,\,\,\, X_0 (\pi) = x_0\,,
\end{equation}
where 
$x_0\in \R^d$ is a given initial state,
$b: [0,T]\times \mathbb R^{d}\times \mathcal{P}(A)\to  \mathbb R^{d}$
and 
$\sigma:[0,T]\times\R^d\times \mathcal{P}(A)
\rightarrow\R^{d\times d'}$
are sufficiently regular coefficients such that 
\eqref{sde} admits a unique strong solution $X(\pi)$.
For a given  regularisation parameter   $\tau\ge 0$, 
the agent's objective is to minimise the following cost functional
\begin{equation}\label{problem}
J^{\tau}(\pi)\coloneqq \E\left[\int_{0}^{T}\left( f_s(X_s(\pi),\pi_s)+\tau\,h(\pi_s)\right)ds+g(X_T(\pi))\right]
\end{equation}
over all  
admissible controls $\pi\in \mathcal{A_{\mathfrak C}}$,
where 
$f:[0,T]\times \mathbb R^{d}\times \mathcal{P}(A)\to  \mathbb R$ and 
$g: \mathbb R^{d} \to  \mathbb R$ are sufficiently regular cost functions, 
and $h:\mathfrak C\to \R$ is a convex regularisation function.
Common choices  of  $h$ include the relative entropy used  in reinforcement learning 
\cite{wang2020reinforcement,wang2020continuous,giegrich2022convergence, zhao2023policy, kerimkulov2023fisher},
the $f$-divergence used in information theory and statistical learning \cite{csiszar1964informationstheoretische}, 
and the (regularised) Wasserstein cost used in optimal transport 
\cite{feydy2019interpolating,chizat2023doubly}. 
Choquet regularization is proposed in~\cite{han2023choquet} for the special case when $A=\mathbb R$.
Typically,  the regulariser $h$  does not admit higher-order dertivatives over its domain and in this sense it is less regular than the cost functions $f$ and $g$.

\subsection*{Difficulties in gradient-based algorithms.}

This work aims to design a convergent gradient-based algorithm for  a (nearly) optimal control of \eqref{problem}. 
To this end, for each $\tau\ge 0$, we 
equivalently write   \eqref{problem} as 
\begin{equation}
\label{eq:F_H}
J^\tau(\pi)=J^0(\pi) + \tau \mathcal{H}(\pi),
\quad \textnormal{with $
\mathcal{H}(\pi)\coloneqq 
\E\left[\int_{0}^{T} h(\pi_s) ds \right]$} 
\end{equation}
and the objective is to minimise this over $\pi\in \mathcal A_{\mathfrak C}$.
Such an infinite dimensional optimisation perspective facilitates extending  gradient-based algorithms designed for static minimisation problems to the current dynamic setting.
The decomposition of $J^\tau$ into $J^0$ and $\mathcal H$ further enables treating   the non-smooth regulariser $h$ separately within the   algorithms.

A well-established gradient-based algorithm for   minimising   a  functional $f+ g$ 
over a  Hilbert space  $(\mathcal X,\|\cdot\|_{\mathcal X})$
is the proximal gradient algorithm  \cite{beck2009fast}. 
It is shown in \cite{beck2009fast} that, 
if  $f$ is Fr\'{e}chet  differentiable, strongly convex 
and smooth 
(i.e., the derivative of $f$   is  Lipschitz continuous)
with respect to the norm $\|\cdot\|_{\mathcal X}$,
then the non-smooth term 
$g$ can be handled by the  proximal operator,
and a  proximal gradient method converges exponentially to the minimiser.

However, incorporating measure-valued controls in \eqref{problem} poses significant challenges when adapting the proximal gradient algorithm to the present setting. Firstly, the space $\mathcal A_{\mathfrak C}$ of admissible controls is not a Hilbert space, rendering the convergence results in \cite{beck2009fast} inapplicable. Secondly, computing the gradient of $J^0$ necessitates introducing a suitable notation of derivatives over $\mathcal {P}(A)$.
Moreover, 
to establish strong convexity and smoothness of $J^0$, it is essential to assess its regularity with respect to distinct controls $\pi$ and $\pi'$.
 Since $\mathcal {P}(A)$ can be equipped with multiple metrics (such as the total variation metric or Wasserstein metrics), it remains unclear which metric is the most  natural choice to quantify the regularity of $J^0$.
Lastly, the convergence analysis in \cite{beck2009fast} overlooks the potential regularisation effect of $h$. While \cite{beck2009fast} requires $J^0:\mathcal{A}_{\mathfrak C}\to \R $ to be strongly convex for  exponential convergence  despite the existence of the regulariser $h$, it has been observed in \cite{mei2020global, vsivska2020gradient, kerimkulov2023fisher} that selecting $h$ as the relative entropy   accelerates the convergence of gradient descent algorithms.

\subsection*{Mirror descent algorithm}

In this work, 
we adopt an alternative approach and introduce a mirror descent algorithm for \eqref{problem}.
Our algorithm can be viewed as a dynamic  extension of  \cite{aubin2022mirror},  which 
 builds on optimisation results in $\mathbb{R}^d$~\cite{bauschke2017descent,lu2018relatively}  and studies  mirror descent algorithms for minimising a cost $f+g$ over spaces of measures. 
The key observations therein are  
(i) the Bregman divergence associated with $g$ provides a natural {notion to express regularity of involved functions}; and 
(ii) the  relative smoothness and convexity of $f$ with respect   to    the Bregman divergence of $g$ are 
sufficient   to ensure the convergence of mirror descent.
We extend these insights to the current dynamic setting.

To derive the mirror descent algorithm for \eqref{problem}, define the Hamiltonian 
$H^{\tau}:[0,T]\times\R^d\times\R^d\times\R^{d\times d'}\times \mathfrak C\rightarrow \R $ by
\begin{equation}\label{eq: hamiltonian}
\begin{split}
H_t^{\tau}(x,y,z,m)\coloneqq b_t(x,m)^\top y +\text{tr}(\sigma_t^\top (x,m)z)+f_t(x,m)+\tau\,h(m)\,.
\end{split}
\end{equation}
Under appropriate conditions on the   coefficients, we prove in Lemma \ref{lem lemma form siska szpruch}  
that for all $\pi,\pi'\in \mathcal{A}_{\mathfrak C}$, 
\begin{align}
\label{eq:first_variation_F}
\begin{split}
d^+  J^0 (\pi)(\pi'-\pi)
&\coloneqq \lim_{\eps\searrow 0}\frac{  J^0(\pi+\eps(\pi'-\pi))-J^0(\pi)}{\eps}
=\E\int_0^T\int\frac{\delta H^{0}_s}{\delta m}(\Theta_s(\pi),\pi_s,a)(\pi'_s-\pi_s)(da)\,ds\,,
\end{split}
\end{align}
where
$\Theta(\pi)=(X(\pi), Y(\pi), Z(\pi))$  with $X(\pi)$ being the state process satisfying \eqref{sde}
and 
$ ( Y(\pi), Z(\pi))$ being the adjoint processes satisfying  
the following backward stochastic differential equation (BSDE):
\begin{equation}
\label{eq: adjoint equation}
\left\{
\begin{aligned}
dY_s(\pi)&=-(D_x H^0_s)(X_s(\pi),Y_s(\pi),Z_s(\pi),\pi_s)\,ds+Z_s(\pi)\,dW_s\,,\,\,s\in[0,T]\,,\\
Y_T(\pi)&=(D_xg)(X_T(\pi))\,, 
\end{aligned}
\right.
\end{equation}
and 
for all $\phi:\mathfrak C\to \R$, we denote by
$\frac{\delta \phi}{\delta m}:\mathfrak C\times A\to \R$
the flat derivative of $\phi$ (see Definition \ref{def flat derivative}).
Equation \eqref{eq:first_variation_F}  indicates  that 
$ \frac{\delta H^{0}_\cdot}{\delta m}(\Theta_\cdot(\pi),\pi_\cdot,\cdot)$
characterises the first variation of   $J^0$ at $\pi$. 

Now assume that $h$ admits a flat derivative $ \frac{\delta h}{\delta m}$ on $\mathfrak C$,
and consider the  Bregman divergence associated with $h$
such that for all  $m, m' \in \mathfrak C$, 
\begin{equation}
\label{eq:D_h_introd}
D_h(m'|m) = h(m')-h(m) - \int \frac{\delta h}{\delta m }(m,a) (m'-m)(da) \,.
\end{equation}
The   mirror descent algorithm for  \eqref{problem} is given as follows.

\begin{algorithm}[H]
\label{alg mmsa}
\DontPrintSemicolon
\SetAlgoLined

\KwInput{initial guess  $\pi^0  \in \mathcal{A}_{\mathfrak C}$ and  the parameter  $\lambda > 0$.
}

\For{$n =0, 1, 2,\ldots$}
{
{Given  $\pi^{n} \in \mathcal{A}_{\mathfrak C}$,  solve the following  equations:
for all $t\in [0,T]$, 
\begin{equation*}
\label{eq mirror FBSDE}
\left\{
\begin{aligned}
dX_t(\pi^{n})&=b_t(X_t(\pi^{n}),\pi^{n}_t)\,dt+\sigma_t(X_t(\pi^{n}),\pi^{n}_t)\,dW_t\,,
\; X_0(\pi^{n}) = x_0; \\
dY_t(\pi^{n})&=-(D_xH^0_t)(\Theta_t(\pi^{n}),\pi^{n}_t)\,dt+Z_t(\pi^{n})\,dW_t\, ,
\\
\quad Y_T(\pi^{n})&=(D_xg)(X_T(\pi^{n}))\,.
\end{aligned}
\right.
\end{equation*}
} 

{
Update the control  $\pi^{n+1}\in \mathcal{A}_{\mathfrak C} $ such that  for all $t\in [0,T]$, 
\begin{equation}\label{eq update the control bregman msa}
\pi^{n+1}_t \in \argmin_{m\in \mathfrak C}\int\frac{\delta H^{\tau}_t}{\delta m}(\Theta_t(\pi^{n}), \pi^{n}_t, a)(m-\pi^{n}_t)(da) + \lambda \,D_h(m|\pi^{n}_t) \,.
\end{equation}
}  
}
\caption{$h$-Bregman Mirror Descent Algorithm}
\end{algorithm}

By   \eqref{eq:F_H} and \eqref{eq:first_variation_F},
the minimisation step \eqref{eq update the control bregman msa}
can be formally interpreted 
is a pointwise representation of 
the mirror descent update 
\[
\pi^{n+1} 
\in \argmin_{\pi \in \mathcal{A}_{\mathfrak C}}
\left(
d^+ J^\tau (\pi^n)(\pi-\pi^n)
+\lambda D_{\mathcal H}(\pi|\pi^n)\right)\,,
\]
where $d^+ J^\tau (\pi^n) $ is the first variation of $J^\tau$ at $\pi^n$, 
and $D_{\mathcal H}(\pi|\pi^n)
=\E\int_0^T D_h(\pi_s|\pi^n_s) \, ds$ is the Bregman divergence of $\mathcal H$.
As a result, \eqref{eq update the control bregman msa}
optimises   the first-order approximation of $J^\tau$ around $\pi^n$,
and    
the Bregman divergence $\pi\mapsto \lambda D_{\mathcal H}(\pi|\pi^n)$
ensures that the optimisation is over the part of the domain where such an approximation is sufficiently accurate.
Similar ideas have been employed in the design of  gradient-based  algorithms in \cite{schulman2017proximal,sethi2022modified} for problems with finite-dimensional action spaces, and in \cite{szpruch2022optimal} for linear-quadratic problems with measure-valued controls.

\subsection*{Our contributions}

This paper identifies conditions under which
Algorithm \ref{alg mmsa}  converges   to an  optimal control of \eqref{problem}
and characterises its convergence rate. 

\begin{itemize}
\item
For a general   regularisation function $h$, we identify    regularity conditions 
on the coefficients  
such that  
Algorithm \ref{alg mmsa} is well-defined
and the   loss 
$J^\tau: \mathcal{A}_{\mathfrak C}\to \R$ is relative smooth with respect to the divergence 
$(\pi',\pi)\mapsto \E\int_0^T D_h(\pi'_s|\pi_s)   \, d s $ (Theorem \ref{th estimate for difference of J}). 
The coefficient  regularity in the measure component  
is characterised 
using  the Bregman divergence $D_h $   in \eqref{eq:D_h_introd},
and  the diffusion coefficient $\sigma$
is allowed 
to be degenerate, state-dependent and controlled.
\item
Leveraging the relative smoothness, we prove  that 
$\{J^\tau(\pi^{n})\}_{n\in \N}$ is decreasing  
(Theorem \ref{thm energy dissipation mirr}). 
If we further assume that the unregularised Hamiltonian $H^0$ in~\eqref{eq: hamiltonian} and the terminal cost $g$
are convex in  $x$ and  $m$, 
we prove that Algorithm \ref{alg mmsa} converges linearly when $\tau=0$ and exponentially when $\tau>0$ (Theorem \ref{thr convergence of modified MSA for convex case}). 
Note that in comparison with \cite{sethi2022modified}, the exponential convergence is achieved by leveraging the regularisation effect of $h$, without imposing strong convexity on the unregularised cost $J^0$.
To the best of our knowledge, this is the first work on the convergence of mirror descent   algorithms  for  continuous-time   control problems with measure-valued control processes and general regularisation functions.  

\item
We demonstrate the applicability of the convergence results with concrete examples of $h$, 
including  the relative entropy, 
the $\chi^2$-divergence, and  the entropic   Wasserstein cost. 
\end{itemize}

Here we  highlight that analysing mirror descent (i.e., Algorithm \ref{alg mmsa}) in the present dynamic setting introduces additional challenges beyond those encountered in static optimisation problems 
\cite{  aubin2022mirror}.
Specifically, as the coefficient regularity is measured using the Bregman divergence $D_h$, which is neither a metric on $\mathcal P(A)$ nor equivalent to commonly used metrics, special care is required for the well-posedness of Algorithm \ref{alg mmsa} and the differentiability of $J^0$ (see the discussion above Lemma \ref{lem regularity of the flat deriv of hamiltonian}). Moreover, establishing the relative smoothness of $J^0$ requires fine estimates on the solutions of \eqref{eq: adjoint equation}, which employ the theory of  bounded mean oscillation (BMO) martingales (see Remark \ref{rmk:bmo} for details).

We also emphasize that this work analyzes the convergence rate of the discrete-time mirror descent update \eqref{eq update the control bregman msa}, unlike \cite{vsivska2020gradient}, which proves the exponential convergence of a continuous-time Wasserstein gradient flow for the control problem~\eqref{sde}-\eqref{problem} with a relative entropy regularizer. 
Analyzing the discrete-time gradient updates necessitates novel techniques to address time discretization errors, 
which are not covered in the continuous-time gradient flow analysis 
 in \cite{vsivska2020gradient}.
Moreover, our convergence analysis extends the Pontryagin optimality principle and BSDE analysis from \cite{vsivska2020gradient} by 
measuring coefficient regularity 
with   a general Bregman divergence, rather than the Wasserstein metric used in \cite{vsivska2020gradient}. This approach offers greater flexibility in accommodating general model coefficients and general regularizers for the control problem (see Section \ref{sec:example_h} for more details).

\subsection*{Literature review}

Two primary approaches to solve a (stochastic) control problem are the dynamic programming principle, leading to    a nonlinear Hamilton-Jacobi-Bellman (HJB)   partial differential equation  (PDE)~\cite{krylov2008controlled}, and Pontryagin's optimality principle~\cite{boltyanskii1960theory, pontryagin1987mathematical}, leading to   a  coupled forward-backward stochastic differential equation  (FBSDE). Classical numerical methods involve first discretising these nonlinear equations and then solving the resulting discretised equations using iterative methods;
see e.g., \cite{dong2007rate, gyongy2009finite} for the PDE approach 
and  \cite{douglas1996numerical, milstein2007discretization, delarue2006forward, 
ma2008numerical, han2020convergence, ji2020three, reisinger2020posteriori} for the FBSDE approach.

Recently, there has been a growing interest in   designing   iterative algorithms to linearise HJB PDEs  or to decouple FBSDEs before discretisation. 
These algorithms have the 
 advantage 
 of solving linear equations at each step, enabling the use of efficient mesh-free algorithms,  particularly in   high dimensional settings \cite{han2017deep, kalise2018polynomial, sirignano2018dgm}.
For instance, policy iteration has been utilised to linearise HJB PDEs \cite{kerimkulov2020exponential, ito2021neural, huang2022convergence, tran2025policy},
which converges  exponentially   for drift-controlled problems \cite{kerimkulov2020exponential} and super-exponentially   for  control problems
in bounded domains
 \cite{ito2021neural}. 
Additionally, the method of successive approximation and its variants have been proposed to decouple FBSDEs \cite{li2018maximum, kerimkulov2021modified, ji2022modified, sethi2022modified},
which  converge linearly for convex losses \cite{kerimkulov2021modified} and exponentially for strongly convex losses \cite{sethi2022modified}. Note that both policy iteration and the method of successive approximation require the exact minimisation of the Hamiltonian over the entire action space at each iteration, which can be computationally expensive for high dimensional action spaces.
 
Gradient-based methods   enhance algorithm efficiency, especially for high dimensional action spaces, by updating controls using the gradient of the Hamiltonian \cite{sutton2018reinforcement, jia2022policy}.
Most existing works on the convergence of gradient-based algorithms  focus on discrete-time control problems (see \cite{mei2020global, mei2021leveraging, lan2023policy,  kerimkulov2023fisher} and references therein).
For continuous-time control problems with continuous state and action spaces, theoretical studies on the convergence of gradient-based algorithms are   limited.
For continuous-time gradient flows,
 the exponential convergence is proved in 
\cite{vsivska2020gradient} for open-loop control problems with  sufficiently convex losses,
and in 
\cite{sethi2024entropy}  for drift-control problems with stochastic policies.
For discrete update,
the exponential convergence of   policy gradient methods is proved in   \cite{giegrich2022convergence} for linear-quadratic problems with possibly nonconvex costs,
 and in \cite{reisinger2023linear} for 
drift-control problems with finite dimensional action spaces.

\subsection*{Organisation of the paper}
We now conclude Section~\ref{sec:intro} which forms the introduction.
In Section~\ref{sec main}, 
we present the assumptions, main results, and examples of regularisers along with associated Bregman divergences applicable to the paper's setting.
 Section~\ref{sec proofs} collects the  proofs of main results.
Finally,~\ref{appendix:bregman} recalls some known properties of Bregman divergence and~\ref{sec:proof_of_dir_der_of_J0} proves Lemma~\ref{lem lemma form siska szpruch}, the characterisation of the first variation of $J^0$.

\section{Main results}
\label{sec main}
This section 
summarises the model assumptions and presents the main results. 
Throughout this paper, 
let $T\in (0,\infty)$, 
let 
$(\Omega, \mathcal{F}, \mathbb{F},  \Prob)$ be  a filtered
probability space satisfying the usual conditions, 
on which a $d'$-dimensional standard Brownian motion 
$W=(W_t)_{t\in [0,T]}$ is defined. 
Let 
$(A,\rho_A) $ be  a separable metric space, 
and  let $\mathcal P(A)$ be  the space of probability measures on $A$ 
equipped   with 
the topology of the weak convergence of measures
and the associated Borel $\sigma$-algebra.

\subsection{Regulariser $h$ and the associated Bregman divergence}
Let $\mathfrak C $ be a    convex and   measurable subset  of $ \mathcal P(A)$,
and 
$h: \mathfrak C\to \mathbb{R} $ be  a convex   function,  i.e.,
$h(\varepsilon m + (1-\varepsilon)m')\le \varepsilon h(m)+(1-\varepsilon)h(m')$
for all $\varepsilon \in [0,1]$ and $m,m'\in  \mathfrak C$. 
We require  the function 
$h  $ to have a flat derivative 
$\frac{\delta  h}{\delta m}:  \mathfrak C\times A\to \mathbb{R}$, whose precise definition is given as follows.

\begin{definition}[Flat derivative on $\mathfrak C \subseteq  \mathcal{P}(A)$]
\label{def flat derivative}
We say a   function $f:\mathfrak C\to \R^d $ has  a flat derivative, 
if there exists a measurable function  $\frac{\delta f }{\delta m} :\mathfrak C\times A \to \mathbb{R}^d $,
called the flat  derivative of $f$, 
such that for all $m,m'\in\mathfrak C$, $\int \big| \frac{\delta f }{\delta m} (m, a) \big| m'(da) <\infty$,
\begin{equation}\label{eq in def of flat der}
\lim_{\varepsilon \searrow 0 }\frac{f(m^\varepsilon)- f(m) }{\varepsilon} =
\int  \frac{\delta f }{\delta m} (m, a) ( m'-m) (da) \quad \textnormal{with $m^\varepsilon =m + \varepsilon (m' - m)$\,,}
\end{equation}
and $\int  \frac{\delta f }{\delta m} (m, a) m(da) =0$.

Given a measurable space  $\mathcal Y$, 
we say a   function $f:\mathcal Y \times \mathfrak C\to \R^d $ has a   flat derivative
with respect to $m$ 
on $\mathfrak C$, 
if there exists a measurable function  $\frac{\delta f }{\delta m}: \mathcal Y \times  \mathfrak C\times A \to \mathbb{R}^d $ such that  $\frac{\delta f}{\delta m} (y,\cdot)$ is the flat derivative of $f(y,\cdot)$ for all $y\in \mathcal Y$.
\end{definition}

\begin{remark}
\label{rmk:FTC}
One can show that 
if  $f:\mathfrak C\to \R^d $ admits a flat derivative $\frac{\delta f }{\delta m}$,
then for all $m,m'\in \mathfrak C$, 
the function 
$[0,1]\ni \varepsilon \mapsto f(m^\varepsilon) $ is
continuous on $[0,1]$ and 
differentiable on $(0,1)$ with derivative 
$\frac{d}{d \varepsilon}f(m^\varepsilon) = \int \frac{\delta f }{\delta m} (m^\varepsilon , a) ( m'-m) (da) $
(see \cite[ Theorem 2.3]{jourdain2021central} and
\cite[Lemma 4.1]{guo2023towards}).
Hence by the fundamental theorem of calculus, 
$f(m')-f(m)=\int_0^1 \int \frac{\delta f }{\delta m} (m^\varepsilon , a) ( m'-m) (da)d \varepsilon$
provided that $\varepsilon \mapsto \int  \frac{\delta f }{\delta m} (m^\varepsilon , a) ( m'-m) (da)$ is integrable.
\end{remark}

Given the regulariser  $h$ and its flat derivative  
$\frac{\delta  h}{\delta m}$, 
define   the Bregman divergence 
$D_h(\cdot|\cdot): \mathfrak C\times \mathfrak C\to [0,\infty)$  by
\begin{equation}
\label{eq:bregman_def}
D_h(m'|m) = h(m')-h(m) - \int \frac{\delta h}{\delta m }(m,a) (m'-m)(da) \,,
\end{equation}
and 
for a  fixed   $\nu\in \mathfrak C$, 
define the set    $\mathcal A_{\mathfrak C}$ of   the admissible controls by
\begin{equation} 
\label{eq:admissible_AC}
\begin{split}
\mathcal{A_{\mathfrak C}}&\coloneqq \left\{\pi:\Omega \times [0,T] \rightarrow \mathfrak C 
\,\middle\vert\, 
\begin{aligned}
&\textnormal{$\pi$ is progressively  measurable,}
\\
&\textnormal{$\mathbb E\int_0^T  | h(\pi_t ) |\, d t <\infty$ and $\mathbb E\int_0^T  D_h(\pi_t|\nu) \, d t <\infty $. }
\end{aligned}
\right\}\,.
\end{split}	
\end{equation}
Here the condition   $\mathbb E\int_0^T  D_h(\pi_t|\nu) \, d t <\infty$ 
ensures for each control $\pi\in \mathcal A_{\mathfrak C}$,   the   state process is   square integrable (see Proposition \ref{prop existence and uniqueness of the solution}), 
while the condition $\mathbb E\int_0^T  | h(\pi_t ) |\, d t <\infty$ ensures that the regularised cost $J^\tau(\pi)$ is finite for all $\tau>0$. 
Note that $\mathcal{A_{\mathfrak C}}$ is  convex due to the  convexity of $h$ and   $\pi\mapsto D_h(\pi|\nu)$ (see Lemma \ref{lemma:properties-of-h-bregman} Item \ref{item:D_convex}).

In the sequel, we   
work with a general regularisation function  $h$,
quantify the coefficient  regularity   using the Bregman divergence $D_h $,
and  analyse the convergence of Algorithm \ref{alg mmsa}. 
Concrete examples of $h$, including   relative entropy, $\chi^2$-divergence and entropic optimal transport,  and the corresponding Bregman divergence $D_h$
will be provided
in Section \ref{sec:example_h}.

\subsection{Well-posedness of the iterates}

We start by  imposing suitable  regularity conditions on the coefficients 
for the well-posedness of 
the state dynamics \eqref{sde} and adjoint dynamics \eqref{eq: adjoint equation}. 

\begin{assumption}[Regularity of $b$ and $\sigma$]
\label{assumption controlled SDE for modified MSA}
The measurable functions 
$b:[0,T]\times\R^d\times \mathcal{P}(A)\to \R^d$ and 
$\sigma:[0,T]\times\R^d\times \mathcal{P}(A)
\rightarrow\R^{d\times d'}$
are   differentiable in $x$ and 
have     flat derivatives    $\frac{\delta b}{\delta m} 
$ and  $\frac{\delta \sigma}{\delta m}
$, respectively, 
with respect to $m$ on $\mathfrak C$.
There exists  $\nu_0\in \mathfrak C$ 
such that $\int_0^T |b_t(0,\nu_0)|\, dt<\infty$
and  $\int_0^T |\sigma_t(0,\nu_0)|^2\, dt<\infty$,
and there exists $K\ge 0$ 
such that for all  $  t\in[0,T]$, $ x, x'\in\R^d$ and $  m, m'\in \mathfrak C$,
\begin{equation*}
\begin{split}
|b_t(x,m)-b_t(x',m')|^2 +|\sigma_t(x,m)-\sigma_t(x',m')|^2 & \le K\left( |x-x'|^2+D_h(m|m')\right)\,.
\end{split}	
\end{equation*}
\end{assumption}

\begin{assumption}[Regularity of   $f$ and $g$]
\label{assum:regularity_f_g}
The measurable function 
$f: [0,T]\times \mathbb R^d \times \mathcal{P}(A) \to \mathbb R$ 
is differentiable in $x$
and has  a   flat derivative   $\frac{\delta f}{\delta m}
$  
with respect to $m$ on $\mathfrak C$.
The function $g:  \mathbb R^d\to \mathbb R$ is differentiable.
There exists  $\nu_0\in \mathfrak C$ 
such that $\int_0^T |f_t(0,\nu_0)|\, dt<\infty$,
and 
there exists   $K\ge0$ such that  
for all $t\in [0,T]$,
$  x\in\R^d$ and $ m\in \mathfrak C$,
$|(D_x g)(x)|+|(D_x f_t)(x,m)|\le K
$.
\end{assumption}

Under   Assumptions~\ref{assumption controlled SDE for modified MSA} and \ref{assum:regularity_f_g}, 
for each $\pi\in \mathcal{A_{\mathfrak C}}$,
the equations~\eqref{sde}
and~\eqref{eq: adjoint equation}
admit unique solutions, as shown in the following proposition. 

\begin{proposition}\label{prop existence and uniqueness of the solution}
Suppose  Assumptions~\ref{assumption controlled SDE for modified MSA} and \ref{assum:regularity_f_g} hold.
For each  $\pi\in\mathcal{A}_{\mathfrak C}$,
\eqref{sde} 
has a unique strong solution $X(\pi)\in L^2(\Omega ; C([0,T];\mathbb R^d))$
and  \eqref{eq: adjoint equation} has  a unique   solution   $(Y(\pi),Z(\pi)) \in  L^2(\Omega;  C([0,T];\mathbb R^d)) \times L^2(\Omega \times [0,T ]; \mathbb R^{d\times d'})$.
\end{proposition}

The proof of Proposition~\ref{prop existence and uniqueness of the solution} 
follows directly from standard well-posedness results of SDEs and BSDEs (see e.g., \cite{zhang2017backward}) and hence is omitted.
In particular,   the square integrability of $X(\pi)$ follows from the   condition $\mathbb E\int_0^T  D_h(\pi_t|\nu) \, d t < \infty$ of $\pi\in\mathcal{A}_{\mathfrak C}$.

We further assume that 
the mirror descent iterates $\{\pi^n\}_{n\in \N\cup\{0\}}$ in Algorithm \ref{alg mmsa}
are well-defined.

\begin{assumption}
\label{ass iterate_wd}
Let $\pi^0\in  \mathcal{A}_{\mathfrak C}$ such that 
for all  $\lambda>0$,  
the iterates  $\{\pi^n\}_{n\in \N}$ in Algorithm \ref{alg mmsa}
exist, belong to $ \mathcal{A}_{\mathfrak C}$ and satisfy
$\mathbb E\int_0^T D_h(\pi^{n+1}_t|\pi^n_t) \, d t < \infty$ for all $n\in \N\cup\{0\}$.
\end{assumption}

Assumption \ref{ass iterate_wd} assumes that the pointwise update   \eqref{eq update the control bregman msa} maintains the admissible control set $\mathcal{A}_{\mathfrak C}$. The condition that   $\mathbb E\int_0^T D_h(\pi^{n+1}_t|\pi^n_t) d t<\infty$ allows for using   the integrated Bregman divergence as a Lyapunov function for the convergence  analysis.
In general, 
Assumption \ref{ass iterate_wd} 
needs to be verified 
depending  on the specific choices of $h$.
We refer the reader to 
Section \ref{sec:example_h} for more details.

\subsection{Convergence of the iterates}

We     introduce additional  regularity assumptions  for the coefficients,
which   will be used  to establish the convergence   of Algorithm \ref{alg mmsa}.

\begin{assumption}[Regularity of  spatial derivatives]
\label{assumption:spatial_derivative}

There exists $K\ge 0$ 
such that 
for all  $\phi \in \{b,f, \sigma \}$ and 
for all  $t\in[0,T]$, $ x, x'\in\R^d$ and $  m, m'\in \mathfrak C$,
\begin{equation*}
\begin{split}
|(D_x \phi_t)(x,m)-(D_x  \phi_t)(x',m')|^2 & \le K\left( |x-x'|^2+D_h(m|m')\right)\,,\,\,\quad
|(D_xg)(x)-(D_xg)(x')|  \le K |x - x'|\,.
\end{split}	
\end{equation*}
\end{assumption}

\begin{assumption}[Regularity of  measure derivatives]\label{assum:regularity_flat_derivative}
For each  $\phi \in \{b,\sigma, f \}$,
the function  
$\frac{\delta \phi}{\delta m}$ 
is   differentiable in $x$,
has  a flat derivative $\frac{\delta^2 \phi}{\delta m^2}$ 
with respect to $m$ on $\mathfrak C$,
and satisfies for some  $K\ge 0$   that
for all 
$  t\in[0,T]$, $  x\in\R^d$ and $  m,m',m''\in\mathfrak C$,
\begin{equation*}
\begin{split}
\left|\int \frac{\delta \phi_t}{\delta m}(x,m'',a)(m-m')(da)\right|^2 + \left|\int \left(D_x\frac{\delta \phi_t}{\delta m}\right)(x,m'',a)(m-m')(da)\right|^2 & \le K D_h(m|m')\,.
\end{split}
\end{equation*}
Moreover, 
for each $\phi\in \{b,f\}$, there exists   $K\ge 0$ such that 
for all 
$  t\in[0,T]$, $  x\in\R^d$ and $ m,m', m''\in\mathfrak C$,  
\begin{equation*}
\left|\int\int\frac{\delta^2 \phi_t}{\delta m^2}(x,m'',a,a')(m-m')(da')(m-m')(da)\right|\le KD_h(m|m')\,.
\end{equation*}
For all   $  t\in[0,T]$, $  x\in\R^d$,  $ m\in\mathfrak C$ and $  a, a' \in A$, 
$\frac{\delta^2 \sigma_t}{\delta m^2}(x,m,a,a') = 0$. 
\end{assumption}

\begin{remark}
Assumption \ref{assum:regularity_flat_derivative}
requires that the derivative $\frac{\delta \sigma}{\delta m}$ is independent of $m$.
This condition is satisfied  when the diffusion coefficient is uncontrolled, or when it depends linearly on the measure-valued control. The latter scenario commonly arises in optimal resource allocation problems, where the control process represents the proportion of resources allocated to various projects, and the state process models the total amount of resources, which evolves according to the returns generated by the invested projects, depending linearly on the allocation strategy.

The condition on  $\frac{\delta \sigma}{\delta m}$  is used to prove  Lemma \ref{lem regularity of the flat deriv of hamiltonian}, which  is crucial for   establishing  the relative smoothness result in  Theorem \ref{th estimate for difference of J}.
Specifically, 
in  \eqref{eq:2nd_H}, 
we aim to show 
(assuming 
for simplicity that $d=d'=1$) that for all $\pi,\pi'\in \mathcal{A}_\mathfrak{C} $,
\begin{align}
\label{eq:Z_estimate_divergence}
\int_0^T\mathbb E\left[ 
\int\int  \frac{\delta^2 \sigma_s}{\delta m^2}(X_s(\pi),\pi^{\varepsilon',\varepsilon}_s,a,a')Z_s(\pi)(\pi'_s-\pi_s)(da')(\pi'_s-\pi_s)(da) 
\right] d s\le C \mathbb E\int_0^T D_h(\pi'_s|\pi_s)d s\,.    \end{align}
The BMO estimate of $Z$ in Lemma  \ref{lem Y Z bounded} and  the upper bound in Lemma  \ref{lemma:L2_BMO_L1} suggest showing 
\begin{align*}
\mathbb E\left[ \left(\int_0^T
\left|\int\int  \frac{\delta^2 \sigma_s}{\delta m^2}(X_s(\pi),\pi^{\varepsilon',\varepsilon}_s,a,a') (\pi'_s-\pi_s)(da')(\pi'_s-\pi_s)(da) \right|^2
 d s\right)^{1/2}\right]\le C \mathbb E\int_0^T D_h(\pi'_s|\pi_s)d s\,.  \end{align*}
However, this would require bounding the $L^2$-norm in time by an $L^1$-norm, which is generally infeasible unless the integrand is  zero.

The above discussion  reflects  a fundamental difficulty in analyzing gradient-based algorithms  for continuous-time control problems with nonlinearly controlled diffusion coefficients.
Existing works focus on cases where the diffusion coefficient is either uncontrolled (see, e.g., \cite{reisinger2023linear, sethi2024entropy}) or linearly controlled (see \cite{vsivska2020gradient, giegrich2022convergence, davey2025convergence, plank2025policy}). It remains an interesting open problem to identify sufficient conditions for designing convergent  gradient-descent algorithms when diffusion coefficients are nonlinearly controlled.
\end{remark}

Using Assumptions~\ref{assumption controlled SDE for modified MSA}, \ref{assum:regularity_f_g}, \ref{assumption:spatial_derivative} and  \ref{assum:regularity_flat_derivative},
we  prove  that    the performance difference of any two admissible  controls can be  upper bounded 
by a   first order term  and their Bregman divergence. 
By interpreting 
$\frac{\delta H^{\tau}_\cdot}{\delta m}(\Theta_\cdot(\pi),\pi_\cdot,\cdot)$  
as the derivative of $\pi\to J^\tau (\pi)$,   
the theorem shows that the cost functional is relatively smooth with respect to the Bregman divergence (see e.g., \cite{aubin2022mirror}).

\begin{theorem}[Relative smoothness] 
\label{th estimate for difference of J}
Suppose  Assumptions~\ref{assumption controlled SDE for modified MSA}, \ref{assum:regularity_f_g}, \ref{assumption:spatial_derivative} and 
\ref{assum:regularity_flat_derivative}   hold and $\tau\ge 0$.
Then there exists  $L\ge \tau $ such that for any $\pi,\pi' \in \mathcal A_\mathfrak{C}$, 
\begin{equation*}
\label{eq estimate for diff J with explicit update}
\begin{split}
J^{\tau}(\pi')-J^{\tau}(\pi)  & \le  \E\int_0^T \int\frac{\delta H^{\tau}_s}{\delta m}(\Theta_s(\pi),\pi_s,a)(\pi'_s-\pi_s)(d a)  ds + L\E\int_0^T\,D_h(\pi_s'|\pi_s)\,ds\,.
\end{split}
\end{equation*}
\end{theorem}

The proof 
of Theorem \ref{th estimate for difference of J} 
is  given in Section~\ref{sec:proof-relative-smoothness}. 
The constant $L$ depends on
the regularisation parameter $\tau$,
the time horizon $T$, 
the initial state $x_0$,
and  the dimension $d$ of the state process, the dimension $d'$ of the Brownian motion and the constants in Assumption~\ref{assumption controlled SDE for modified MSA}, \ref{assum:regularity_f_g}, \ref{assumption:spatial_derivative} and \ref{assum:regularity_flat_derivative}.

Based on Theorem \ref{th estimate for difference of J},
we prove    that the cost functional  decreases along 
the iterates 
$\{\pi^n\}_{n\in \mathbb N}$ from   Algorithm~\ref{alg mmsa}.

\begin{theorem}[Energy dissipation]
\label{thm energy dissipation mirr}
Suppose  Assumptions~\ref{assumption controlled SDE for modified MSA}, \ref{assum:regularity_f_g}, 
\ref{ass iterate_wd}, \ref{assumption:spatial_derivative} and 
\ref{assum:regularity_flat_derivative}
and $\tau\ge 0$.
Let  $\lambda \ge  L$ with $L$ from   Theorem \ref{th estimate for difference of J},
and 
let 
$\{\pi^n\}_{n\in \mathbb N}$ be the iterates from   Algorithm~\ref{alg mmsa}.
Then 
$J^{\tau}(\pi^{n+1}) \leq J^{\tau}(\pi^{n})$ for all $ n\in \N\cup\{0\} $.
Assume further that 
$\inf_{\pi \in\mathcal A_{\mathfrak C}} J^\tau(\pi)>-\infty$,
then 
$\lim_{n\rightarrow \infty}\mathbb E\int_0^T D_h(\pi^{n+1}_t| \pi^n_t)\,dt = 0$.
\end{theorem}

The proof of Theorem~\ref{thm energy dissipation mirr} is given  in Section~\ref{sec:proof-energy-dissipation-h-breg-mirr}.
Note that Theorem~\ref{thm energy dissipation mirr} 
allows  the cost functional $\pi\mapsto J^{\tau}(\pi)$ to be nonconvex
and the regularisation parameter $\tau$ to be zero. 
However, it does not imply the cost  function $\{J^{\tau} (\pi^{n} )\}_{n\in \N}$
converges to the optimal cost.

To ensure the convergence of Algorithm~\ref{alg mmsa} and determine its convergence rate, we impose    additional convexity assumptions  on the coefficients under which  the Pontryagin optimality principle serves as a sufficient criterion for optimality.

\begin{assumption}[Convexity]\label{ass on convexity of hamiltonian}
The function $g: \R^d\to \R$ is convex, 
and for all $(t,y,z)\in [0,T]\times \R^d\times \R^{d\times d'}$, 
the function $(x,m)\mapsto H^0_t(x,y,z,m)$ is convex, i.e., 
for all $x,x'\in \R^d$ and $m,m'\in  \mathfrak C$, 
\begin{align}
\begin{split}
&H^0_t(x',y,z,m') - H^0_t(x,y,z,m) 
\ge 
D_x H^0_t(x,y,z,m)^\top (x'-x)
+
\int\frac{\delta H^0_t}{\delta m}(x,y,z,m,a)(m'-m)(da)\,.
\end{split}
\end{align}

\end{assumption}

Using Assumption~\ref{ass on convexity of hamiltonian},
we  prove  that    the cost functional $\pi\mapsto J^\tau (\pi)$ is strongly convex with respect to the Bregman divergence 
$D_h(\cdot|\cdot)$, and the modulus of   convexity is the regularisation parameter $\tau\ge 0$.

\begin{theorem}[Relative convexity]
\label{th relative convexity}
Suppose  Assumptions~\ref{assumption controlled SDE for modified MSA}, \ref{assum:regularity_f_g} and~\ref{ass on convexity of hamiltonian} hold
and $\tau\ge 0$.
Then for all $\pi,\pi'\in \mathcal A_{\mathfrak C}$ with  
$\E\int_{0}^{T} D_h(\pi'_s|\pi_s)\, ds  <\infty$, 
\begin{equation}
\label{eq:relative-convexity-J}
\begin{split}
J^{\tau}(\pi')-J^{\tau}(\pi)  \ge   \E\int_0^T\int\frac{\delta H^{\tau}_s}{\delta m}(\Theta_s(\pi), \pi_s, a)(\pi'_s-\pi_s)(da)\,ds +  \tau\E\int_{0}^{T} D_h(\pi'_s|\pi_s) \,ds\,.
\end{split}
\end{equation}
\end{theorem}

The proof of Theorem~\ref{th relative convexity} is given in Section~\ref{sec:proof-rel-convexity}.

\begin{remark}
   We contrast Theorems~\ref{th estimate for difference of J} and~\ref{th relative convexity} with the performance difference lemma, a   result often used to analyze gradient-based algorithms for \emph{Markov feedback controls}.

It is   known that the control objective is typically non-convex with respect to Markov controls (see \cite[Proposition 2.4]{giegrich2022convergence}). To develop and analyze gradient-based algorithms for Markov controls, the performance difference lemma is a key tool. This lemma provides an exact expression for the difference between the value functions associated with any two Markov controls, in contrast to the approximate bounds presented in Theorems~\ref{th estimate for difference of J} and~\ref{th relative convexity}.
 Such a lemma was first introduced in~\cite[Ch. 7, p. 87]{howard1960dynamic} in the context of controlled Markov chains, and has since become a foundational tool in the analysis of gradient-based methods for Markov policies; see \cite{kakade2002approximately, kerimkulov2023fisher} for controlled Markov chains, and \cite{giegrich2022convergence, zhao2023policy, sethi2024entropy} for controlled diffusion processes.
In particular, the performance difference lemma enables a precise characterization of the cost landscape, which in turn makes it possible to recover optimal convergence rates for gradient flows--as if the objective were (strongly) convex \cite{kerimkulov2023fisher, sethi2024entropy}.

In contrast, this paper focuses on optimization over open-loop control processes. We show in Theorem~\ref{th relative convexity} that the control objective is convex with respect to the control processes. By leveraging this improved landscape property, the analysis of gradient-based algorithms becomes more direct, and the performance difference lemma is not required.

Moreover, it is unclear whether an analogous performance difference lemma holds in the open-loop setting. The standard proof of the lemma relies on the Feynman--Kac formula, which is applicable under Markovian assumptions. In our setting, however, controls are general adapted processes, potentially non-Markovian, for which such a representation may not be available.

\end{remark}

Combining Theorems \ref{th estimate for difference of J} and  \ref{th relative convexity},
we characterise the convergence rate of Algorithm  \ref{alg mmsa} depending on the regularisation parameter $\tau$.
Recall that  for all $\tau \ge 0$, $\{J^\tau(\pi^{n})\}_{n\in \N}$ is decreasing, as shown in 
Theorem \ref{thm energy dissipation mirr}.

\begin{theorem}[Linear/Exponential convergence of Algorithm \ref{alg mmsa}]
\label{thr convergence of modified MSA for convex case}
Suppose  Assumptions~\ref{assumption controlled SDE for modified MSA}, \ref{assum:regularity_f_g}, 
\ref{ass iterate_wd}, 
\ref{assumption:spatial_derivative},
\ref{assum:regularity_flat_derivative} and~\ref{ass on convexity of hamiltonian} hold
and $\tau\ge 0$.
Let $\lambda \ge L$    with $L$ from   Theorem \ref{th estimate for difference of J},
and
let 
$\{\pi^n\}_{n\in \mathbb N}$ be the iterates from   Algorithm~\ref{alg mmsa}.
\begin{itemize}
\item
If  $\tau=0$, then  for all  
$\pi\in \mathcal{A}_{\mathfrak C}$ such that   
$\E\int_0^TD_h(\pi_s|\pi^{n}_s)\,ds<\infty$
for all $n\in \N\cup\{0\}$, 
\begin{equation*}
J^0(\pi^{n})- J^0(\pi)\le\frac{\lambda}{n}\E\int_0^TD_h(\pi_s|\pi^{0}_s)\,ds\,,\quad n\in \N\,,
\end{equation*}
\item 
If $\tau>0$ and 
$\pi^\ast \in \mathcal{A}_{\mathfrak C}$ satisfies  
$J^\tau (\pi^\ast) = 
\min_{\pi\in A_{\mathfrak C}} J^\tau (\pi)$
and $\E\int_0^TD_h(\pi^\ast_s|\pi^{n}_s)\,ds<\infty$
for all $n\in \N\cup\{0\}$, 
then 
\begin{equation}
\label{eq:convergence_regularized}
\begin{split}
0 \leq J^\tau(\pi^{n})-J^\tau(\pi^\ast)\le    \lambda\left(1-\frac{\tau}{\lambda}\right)^n\E\int_0^TD_h(\pi^\ast_s|\pi^{0}_s)\,ds\,,\quad n\in \N\,. 
\end{split}
\end{equation}
\end{itemize} 

\end{theorem}

The proof of Theorem~\ref{thr convergence of modified MSA for convex case} is given  in Section~\ref{sec:proof-of-convergence}.

\begin{remark}
The   convergence results    in Theorem \ref{thr convergence of modified MSA for convex case}
require prior knowledge that the  target control is compatible  with  the iterates $\{\pi^n\}_{n\in \mathbb N}$,
as indicated by their finite Bregman divergence.

When $\tau=0$, 
the theorem provides an upper 
bound of the performance difference
$J^0(\pi^n)-J^0(\pi)$ for  
 any compatible control $\pi$.
 This 
indicates that $\{\pi^n\}_{n\in \mathbb N}$ asymptotically outperforms any compatible control $\pi$ with a rate of $\mathcal O(1/n)$. 
It   implies the convergence of Algorithm \eqref{alg mmsa}, if an optimal control of \eqref{problem} exists and is compatible to the mirror descent iterates.
This result extends the convergence result of mirror descent from the static optimization setting \cite{aubin2022mirror} to the dynamic framework considered here.

However, as in the static case \cite{aubin2022mirror}, we note   that
due to the absence of regularisation  ($\tau=0$),
the compatibility  condition 
may not hold in general and
has to be verified on a case-by-case basis,   depending on the choice of the Bregman divergence
 $D_h$ (and thus on the function $h$). 
See Remarks~\ref{remark finite cardinality A and KL} and~\ref{remark possibility of optional control being admissible ot} for conditions under which the compatibility condition is satisfied, thereby ensuring the 
$\mathcal O(1/n)$-convergence rate for the optimal controls.

When $\tau>0$, the scenario changes 
 as in this case, 
the existence of an optimal control $\pi^\ast$ and its compatibility 
can be  verified by exploiting the regularisation effect  of $h$.
In particular, 
under Assumption \ref{ass on convexity of hamiltonian}, 
the optimal control $\pi^\ast\in \mathcal{A}_{\mathfrak C}$ 
can be characterised by 
the following  Pontryagin system: for all $t\in [0,T]$, 
\begin{equation}
\label{eq:optimal_FBSDE}
\left\{
\begin{aligned}
dX_t(\pi^\ast)&=b_t(X_t(\pi^\ast),\pi^\ast_t)\,dt+\sigma_t(X_t(\pi^\ast),\pi^\ast_t)\,dW_t\,,  
\quad
X_0(\pi^\ast) = x\,,
\\
dY_t(\pi^\ast)&=-(D_xH^\tau_t)(\Theta_t(\pi^\ast),\pi^\ast_t)\,dt+Z_t(\pi^\ast)\,dW_t\, ,
\quad Y_T(\pi^\ast)=(D_xg)(X_T(\pi^\ast))\,,
\\
\pi^\ast_t &=\argmin_{m\in \mathfrak C} H^\tau_t(\Theta_t(\pi^\ast),m)\,.
\end{aligned}
\right.
\end{equation}
We anticipate that 
using the relative convexity of $\pi\mapsto J^\tau(\pi)$,
the well-posedness of \eqref{eq:optimal_FBSDE} can be studied by 
the Method of Continuation as in 
\cite{zhang2017backward, guo2023reinforcement},
which subsequently allows for verifying   the finiteness of  $\E\int_0^TD_h(\pi^\ast_s|\pi^{n}_s)\,ds$.
We leave a detailed analysis for future work. 
\end{remark}

Theorem \ref{thr convergence of modified MSA for convex case} 
provides the theoretical foundation for solving the unregularized control problem by introducing    gradually reducing   regularization parameters, known as annealing. 
For instance, let $\{\pi^n_\tau\}_{n\in \mathbb N}$
denote  the iterates generated by 
 Algorithm \ref{alg mmsa},  
 applied to the \emph{regularized objective}~\eqref{eq:F_H} with some $\tau>0$. 
 The performance of $\{\pi^n_\tau\}_{n\in \mathbb N}$ 
 in solving the unregularized objective   
 can be decomposed as  
\begin{align*}
    J^0(\pi^n_\tau)-\inf_{\pi \in \mathcal{A}_{\mathfrak C} }J^0(\pi) & =\left(J^0(\pi^n_\tau)-J^\tau(\pi^n_\tau)\right) + 
\left(J^\tau(\pi^n_\tau)-\inf_{\pi \in \mathcal{A}_{\mathfrak C}} J^\tau (\pi ) \right)+
\left(
\inf_{\pi \in \mathcal{A}_{\mathfrak C}} J^\tau(\pi  )  -  \inf_{\pi \in \mathcal{A}_{\mathfrak C}} J^0(\pi) \right).
\end{align*}
 {\color{magenta} 
 Assuming that the 
  regularisation function $h$ is nonnegative
 (e.g., $h$ is chosen as the relative entropy in Proposition~\ref{prop:relative_entropy},  or the $\chi^2$-divergence  in Proposition~\ref{prop:chi_2},
 or   the entropic optimal transport in Proposition~\ref{prop:eot}  with $c \geq 0$) we have}
 \begin{align*}
    0\le J^0(\pi^n_\tau)-\inf_{\pi \in \mathcal{A}_{\mathfrak C} }J^0(\pi)
    \le   
\left(J^\tau(\pi^n_\tau)-\inf_{\pi \in \mathcal{A}_{\mathfrak C}} J^\tau (\pi ) \right)+
\left(
\inf_{\pi \in \mathcal{A}_{\mathfrak C}} J^\tau(\pi  )  -  \inf_{\pi \in \mathcal{A}_{\mathfrak C}} J^0(\pi)
\right).
\end{align*}
 The first term on the right-hand side   
represents the optimization
error   for a regularised problem,
and can be quantified by \eqref{eq:convergence_regularized} 
in terms of $n$ and $\tau$.
The second term  is the regularization bias resulting from
the additional Bregman divergence in \eqref{eq:F_H}. One can choose 
  $\tau$   to balance these two terms and thereby optimize the convergence (see, e.g., \cite{sethi2024entropy} for cases involving Markovian policies and drift-control problems).

\subsection{Examples of  the regulariser $h$}
\label{sec:example_h}

This section offers various concrete examples of the regularisation function $h$ commonly utilised in machine learning. We will specify the action set $\mathfrak C$ and the associated Bregman divergence $D_h(\cdot|\cdot)$. Additionally, we will verify Assumption \ref{ass iterate_wd}, i.e., the well-definedness of the mirror descent step \eqref{eq update the control bregman msa} in Algorithm \ref{alg mmsa}.

The first example of $h$ is the commonly used   relative entropy, as  studied in~\cite{reisinger2021regularity,vsivska2020gradient,wang2020reinforcement,wang2020continuous,kerimkulov2023fisher}. 
The following proposition shows that 
the Bregman divergence $D_h(m'|m)$ is the 
Kullback--Leibler divergence between $m'$ and $m$.

\begin{proposition}[Relative   entropy]
\label{prop:relative_entropy}
Let $\varrho\in \mathcal P(A)$,  
let  $L^\infty(A)$ be the space of bounded measurable functions equipped with the supremum norm,
let 
$\mathfrak C = \{m \in \mathcal P(A) \mid 
m\ll \varrho, 
\|  \ln \frac{\mathrm d m}{\mathrm d \varrho}\|_{L^\infty(A)}<\infty\}$ 
and 
let 
$h:\mathfrak C \to \R$ be   the relative entropy given by  
\begin{equation*}
h(m)\coloneqq \int \ln\frac{\mathrm dm}{\mathrm d\varrho}(a)   m(da)\,,
\quad m\in  \mathfrak C\,.
\end{equation*}
Then
for all $m,m'\in  \mathfrak C$ and $a\in A$,  
$\frac{\delta h}{\delta m}(m,a) = \ln \frac{\mathrm d m}{\mathrm d \varrho}(a)-h(m)$ and 
\begin{equation}
\label{eq:bregman_KL}
D_h(m'|m)  =  \int \ln \frac{\mathrm d m'}{\mathrm d m }(a)m'(da)\,. 
\end{equation}

Suppose that  Assumptions~\ref{assumption controlled SDE for modified MSA} and \ref{assum:regularity_f_g} hold and  
the flat derivatives   $\frac{\delta b}{\delta m}, 
\frac{\delta \sigma}{\delta m}$ and $\frac{\delta f}{\delta m}$ 
are  uniformly bounded. 
Then Assumption \ref{ass iterate_wd} 
holds 
for all $\pi^0\in \mathcal{A}_{\mathfrak C}$ such that $\mathbb E\int_0^T \big\| \ln \frac{\mathrm d \pi^0_t}{\mathrm d \varrho} \big\|_{L^\infty(A)} dt<\infty$.
In this case, $\mathbb E\int_0^T \big\| \ln \frac{\mathrm d \pi^n_t}{\mathrm d \varrho} \big\|_{L^\infty(A)} dt<\infty$ for all $n\in \mathbb N\cup\{0\}$.
\end{proposition} 

The proof of Proposition  \ref{prop:relative_entropy} is given in Section \ref{sec:example_proof}.
\begin{remark}
By Pinsker's inequality and \eqref{eq:bregman_KL}, 
for all $m,m'\in \mathfrak C$, 
$ D_h(m'|m)\ge \frac{1}{2}\|m'-m\|^2_{\rm TV}$, where $\|\cdot\|_{\rm TV}$ is the total variation norm over the space of signed measures. 
Hence when $h$ is the relative entropy, 
the necessary regularity in the measure component 
in 
Assumptions~\ref{assumption controlled SDE for modified MSA}, \ref{assumption:spatial_derivative}  and
\ref{assum:regularity_flat_derivative} 
can be stated with   the squared norm $\|\cdot\|^2_{\rm TV}$ instead of the  Bregman divergence
$D_h(\cdot|\cdot)$.
\end{remark}

\begin{remark}
\label{remark finite cardinality A and KL}
Going back to Theorem~\ref{thr convergence of modified MSA for convex case}, 
in the unregularized setting 
with  $\tau = 0$,
a sufficient condition to ensure  an optimal  control $\pi^\ast \in \mathcal A_{\mathfrak C} $  (if it exists)
is compatible to the   iterates $\{\pi^n\}_{n\in \mathbb N}$
(i.e. $\E\int_0^TD_h(\pi_s|\pi^{n}_s)\,ds<\infty$
for all $n\in \N\cup\{0\}$), 
is that the action space $A$ is of finite cardinality (say $A=\{a_1,\ldots,a_N\}$),
 the Bregman divergence $D_h$ is chosen as in  Proposition~\ref{prop:relative_entropy} with $\rho$ being the uniform distribution on $A$,
 and the initial guess 
 $\pi^0$ of Algorithm \ref{alg mmsa}
  satisfies  $\mathbb E\int_0^T \big\| \ln \frac{\mathrm d \pi^0_t}{\mathrm d \varrho} \big\|_{L^\infty(A)} dt<\infty$ (e.g.,  $\pi^0 = \rho$).
In this case, the iterates 
$\{\pi^n\}_{n\in \mathbb N}$ converge  to $\pi^\ast$ with a rate $\mathcal O(1/n)$,
as established  in     Theorem~\ref{thr convergence of modified MSA for convex case}.

To see this,
for any $m\in \mathcal P(A)$ we have $0\leq h(m) \leq -\sum_{k=1}^N \ln \rho(a_k) m(a_k) = \ln N < \infty$.
Taking $\nu := \rho$ we similarly have $0\leq D_h(m|\nu) \leq \ln N$. 
Thus  the set $   \mathcal A_{\mathfrak C}$ defined in  \eqref{eq:admissible_AC}   
coincides with the set of all progressively measurable processes    taking  values in $\mathcal P(A)$.
Let $\pi^*\in \mathcal A_{\mathfrak C}$ be an  
  optimal   control.
To show that $\pi^*$ is compatible with $\{\pi^n\}_{n\in \mathbb N}$, 
 note that for $n\in \mathbb N$,
\[
\begin{split}
& \mathbb E\int_0^T  D_h( \pi^*_t|\pi^n_t) dt  
 =\mathbb E\int_0^T \left(  \sum_{k=1}^N \left(\ln \frac{\mathrm d \pi^*_t}{\mathrm d \varrho }(a_k)- \ln \frac{\mathrm d \pi^n_t}{\mathrm d \varrho }(a_k)\right) \pi^*_t(a_k) \right)dt\\      
& \leq - \mathbb E\int_0^T  \sum_{k=1}^N \ln \frac{\mathrm d \pi^n_t}{\mathrm d \varrho }(a_k) \pi^*_t(a_k) \leq  \mathbb E\int_0^T  \bigg\|\frac{\mathrm d \pi^n_t}{\mathrm d \varrho }\bigg\|_{L^\infty(A)}\,dt < \infty\,,
\end{split}
\]
where the final inequality follows from Proposition~\ref{prop:relative_entropy}, as  $\pi_0$ satisfies  $\mathbb E\int_0^T \big\| \ln \frac{\mathrm d \pi^0_t}{\mathrm d \varrho} \big\|_{L^\infty(A)} dt<\infty$.

\end{remark}

Next we consider  $h$ as   Pearson's $\chi^2$-divergence,
which plays a key role in information theory and   machine learning \cite{pearson1900x}. 
The associated Bregman divergence $D_h(m'|m)$ is the squared $L^2$-norm 
of the density of $m'-m$.

\begin{proposition}[$\chi^2$-divergence]
\label{prop:chi_2}

Let 
$\varrho \in \mathcal P(A)$
and  let  $L^2_\varrho(A)$ be 
the space of square integrable functions  
with respect to $\varrho$.
Let    $\mathfrak C = \{m \in \mathcal P(A) \mid 
m\ll \varrho, 
\frac{\mathrm d m}{\mathrm d \varrho} \in L^2_{\varrho}( A)\}$
and let 
$h:\mathfrak C  \to \R$ be   $\chi^2$-divergence given by 
\begin{equation*}
h(m) \coloneqq \int \frac{1}{2} \left(\frac{\mathrm d m }{\mathrm d\varrho }(a) -1\right) ^2   \varrho (da),\,\quad  m\in  \mathfrak C\, .	
\end{equation*}
Then
for all $m,m'\in  \mathfrak C$ and $a\in A$,  
$\frac{\delta h}{\delta m}(m,a) =  \frac{\mathrm d m}{\mathrm d \varrho}(a)-\int  \frac{\mathrm d m}{\mathrm d \varrho}(a) m(da)$ and 
\begin{equation}
\label{eq:bregman_quadratic}
D_h(m'|m)  =  \int \frac{1}{2}\left( \frac{\mathrm d m'}{\mathrm d \varrho }(a) - \frac{\mathrm d m}{\mathrm d \varrho }(a) \right)^2   \varrho(da)\,.
\end{equation}

Suppose that  Assumptions~\ref{assumption controlled SDE for modified MSA} and \ref{assum:regularity_f_g} hold and  
for each $\phi\in \{b,\sigma,f\}$,
\[
a\mapsto \sup_{(t, x,m)\in [0,T]\times  \mathbb R^d\times \mathfrak C}|\frac{\delta \phi_t}{\delta m}(x,m,a)|\in L^2_\varrho(A)\,.
\]
Then Assumption \ref{ass iterate_wd} 
holds 
for all $\pi^0\in \mathcal{A}_{\mathfrak C}$ such that  $\mathbb E\int_0^T\big \|   \frac{\mathrm d \pi^0_t}{\mathrm d \varrho} \big\|^2_{L^2_\varrho(A)} dt<\infty$.
\end{proposition} 

The proof of Proposition  \ref{prop:chi_2}  is given in Section \ref{sec:example_proof}.

\begin{remark}
Recall that for all $m,m'\in \mathfrak C$, 
$$
\|m'-m\|_{\rm TV}=\frac{1}{2}\int \left| \frac{\mathrm d m'}{\mathrm d \varrho }(a) - \frac{\mathrm d m}{\mathrm d \varrho }(a) \right| \varrho (da)\,,
$$
which along with   Jensen's inequality and 
\eqref{eq:bregman_quadratic} implies
$\|m'-m\|^2_{\rm TV}\le \frac{1}{2}D_h(m'|m) $. 
Hence 
when  $h$ is the  $\chi^2$-divergence,  
the required regularity in the measure component 
in 
Assumptions~\ref{assumption controlled SDE for modified MSA}, \ref{assumption:spatial_derivative}  and
\ref{assum:regularity_flat_derivative} 
can be stated with   the squared norm $\|\cdot\|^2_{\rm TV}$.

\end{remark}

Finally, we consider  $h$ as  the entropic   optimal transport   studied in \cite{feydy2019interpolating}. Let 
$A $ be   a compact metric space,
$c\in C(A\times A) $, and 
$\varrho \in \mathcal P(A)$.
Let $\mathfrak C = \mathcal P(A)$, 
and let $h: \mathfrak C\to \mathbb{R}$ be such that 
\begin{equation}
\label{eq:transport_cost}
h(m)\coloneqq \min_{\gamma\in \Pi(m,\varrho) } \left(\int_{A\times A}
c(x,y) d \gamma (x,y)+\kappa \operatorname{KL}(\gamma| m\otimes \varrho)\right),
\end{equation}
where 
$  \Pi(m,\varrho)$ is the 
set of   measures $m\in \mathcal P( A\times A)$ with $m$ and $\varrho$ 
as respective first and second marginals,
$\kappa>0$ is a fixed parameter, 
and  $\operatorname{KL}(\gamma| m\otimes \varrho)=\int_{A\times A}\ln \frac{\mathrm d \gamma }{\mathrm d m\otimes \varrho }d\gamma $ 
if $\gamma \ll m\otimes \varrho$ and $\infty$ otherwise. 
For each  $m\in \mathfrak C$, 
consider a pair  $(\phi,\psi)\in C(A)\times C(A)$ such that for all $(a,a')\in A\times A$,
\begin{align}
\label{eq:schrodinger}
\begin{split}
\phi(a)= -\kappa \ln\left(\int e^{(\psi(a')-c(a,a'))/\kappa }\varrho(d a')\right)\,,\,\,\,
\psi(a')= -\kappa \ln\left(\int e^{(\phi(a)-c(a,a'))/\kappa } m(d a)\right)\,.
\end{split}
\end{align}
The pair $(\phi,\psi)\in C(A)\times C(A)$ 
satisfying \eqref{eq:schrodinger}
is often referred to as the    Schr\"{o}dinger potentials
and 
is unique up to    the transformation
$(\phi+c', \psi-c')$ for $c'\in \R$
(see \cite{feydy2019interpolating}).
In the sequel, we fix $a_0\in A$ and 
choose  the unique solution to \eqref{eq:schrodinger} 
such that $\phi(a_0)=0$.
We denote this specific choice    by   $(\phi[m],\psi[m])$ to emphasise its dependence on $m$.

\begin{proposition}[Entropic   optimal transport]
\label{prop:eot}
Let 
$A $ be   a compact separable   metric space,
$c: A\times A\to \R $ be Lipschitz continuous
and  let   $\varrho \in \mathcal P(A)$.
Define   
$\mathfrak C = \mathcal P(A)$ and 
$h: \mathfrak C\to \mathbb{R}$   as in \eqref{eq:transport_cost}. 
Then for all $m\in  \mathfrak C$ and $a\in A$, 
$
\frac{\delta h}{\delta m }(m, a)= \phi [m](a)-\int \phi[m](a)m(da)$,
where  $(\phi[m],\psi[m])$ are the Schr\"{o}dinger potentials satisfying \eqref{eq:schrodinger}.  

Suppose that  Assumptions~\ref{assumption controlled SDE for modified MSA} and \ref{assum:regularity_f_g} hold,  
and  
for each $\phi\in \{b,\sigma,f\}$ and for all $(t, x,m)\in [0,T]\times  \mathbb R^d\times \mathfrak C$,
$ \frac{\delta \phi_t}{\delta m}(x,m,\cdot) \in C(A)$.
Then Assumption \ref{ass iterate_wd} 
holds 
for all $\pi^0\in \mathcal{A}_{\mathfrak C}$.

\end{proposition}

The proof of Proposition  \ref{prop:eot}  is given in Section \ref{sec:example_proof}.

\begin{remark}
\label{remark possibility of optional control being admissible ot} 

In the unregularized setting 
(with $\tau =0$),
when the function $h$ in \eqref{eq:admissible_AC} is chosen 
as the entropic optimal transport map described in Proposition
\ref{prop:eot},
any optimal    control $\pi^\ast \in \mathcal A_{\mathfrak C} $ (if it exists) 
is compatible to the   iterates $\{\pi^n\}_{n\in \mathbb N}$ as required by  Theorem~\ref{thr convergence of modified MSA for convex case},
  which implies that  
$\{\pi^n\}_{n\in \mathbb N}$ converges   to $\pi^\ast$
with a rate $\mathcal O(1/n)$.  

 To see this, note that 
by the compactness of $A$ and continuity of $c$,
 $\mathfrak C = \mathcal P(A)$ is compact, 
 and   $\mathfrak C\ni  \mu\mapsto \phi[\mu]\in C(A)$  are  continuous (see  
   \cite[Proposition 13]{feydy2019interpolating}).
  Then for any $m\in \mathfrak C$, as
  $\operatorname{KL}(\gamma| m\otimes \varrho)\ge 0 $,
\[
\inf_{x,y\in A} c(x,y) \leq h(m) \leq \int_{A\times A} c(x,y)d(m\otimes \rho)(x,y) \leq \sup_{x,y \in A} c(x,y)\,,
\]
and 
$\|\frac{\delta h}{\delta m }(m, \cdot)\|_{L^\infty}= \left|\phi [m]-\int \phi[m](a)m(da)\right\|_{L^\infty}\le 2 \sup_{m\in \mathfrak C }\|\phi[m]\|_{L^\infty}$. 
Then for all $n\in \mathbb N\cup \{0\}$,
\begin{align*}
   0\le  D_h( \pi^*_t|\pi^n_t) &=h(\pi^*_t) - h(\pi^n_t) - \int \frac{\delta h}{\delta m}(\pi^n_t,a)(\pi^*_t-\pi^n_t)(da)
   \\
   &
   \le \sup_{x,y \in A} c(x,y) - \inf_{x,y\in A} c(x,y) +2 \sup_{m\in \mathfrak C }\|\phi[m]\|_{L^\infty}.
\end{align*}
This proves that $\pi^*$ is compatible to the   iterates $\{\pi^n\}_{n\in \mathbb N}$. 
\end{remark}

\section{Proofs}\label{sec proofs}
This section proves the main results in Section \ref{sec main}.  
For notational simplicity, we denote by 
$C\in [0,\infty)$ a generic constant, which depends only on the constants appearing in the assumptions and may take a different value at each occurrence.

\subsection{Regularity of adjoint processes}

To prove  the relative smoothness of $\pi\mapsto J^\tau (\pi)$ in Theorem~\ref{th estimate for difference of J},
we first prove that 
the  adjoint processes $(Y(\pi), Z(\pi))$ depend Lipschitz continuously on $\pi$. 
The proof relies crucially on an 
a-priori uniform bound   of  $(Z(\pi))_{\pi\in \mathcal{A}_{\mathfrak C}}$ 
in terms of  the  bounded mean oscillation (BMO) martingales.
Compared with   classical $L^2$ estimates,
this improved regularity of    $Z(\pi)$ allows for handling more general dependences in the diffusion coefficient;
see Remark \ref{rmk:bmo} for more details.

We first  introduce some notation.
For each stopping time $\eta:\Omega\to [0,T]$, 
we denote by 
$\E_\tau [\cdot]$ the conditional expectation with respect to $\mathcal{F}_\tau$.
We define the norm $\|\cdot\|_{L^\infty}$  such that  $\|Z\|_{L^\infty}\coloneqq \operatorname{ess}\sup_{\omega}|Z(\omega)|$ for any random variable $Z:\Omega\to \R$,
and the norm $\|\cdot\|_{\mathbb{H}^{\infty}}$
such that 
$\|Z\|_{\mathbb{H}^{\infty}}\coloneqq 
\operatorname{ess}\sup_{(t,\omega)}|Z_t(\omega)|$  
for any progressively measurable $Z:\Omega\times [0,T]\to \R$. 
For a uniformly integrable martingale $M$ with $M_0=0$,   define  
\begin{equation*}
\|M\|_{\operatorname{BMO}}=\sup_{\eta}\left\|(\E_{\eta}\left[\langle M\rangle_T-\langle M\rangle_\eta\right])^{\frac{1}{2}}\right\|_{L^\infty}\,,
\end{equation*}
where $\langle M\rangle$ is the quadratic variation of $M$, and the supremum is taken over all stopping times $\eta:\Omega\to  [0,T]$. 
If $M'$ is another uniformly integrable martingale with $M'_0=0$ then we define the cross-variation of $M$ with $M'$ as $\langle M, M'\rangle_t = \frac14 [\langle M+M' \rangle_t - \langle M - M' \rangle_t]$ for $t\in[0,T]$.
For a square integrable 
adapted process $Z$, we denote by   $Z\cdot W$ the stochastic integral of $Z$ with respect to the Brownian motion $W$, i.e.,  $(Z\cdot W)_t\coloneqq \int_0^t Z_sdW_s$ for $t\in[0,T]$.

The following lemma proves the boundedness of the adjoint processes. 

\begin{lemma}\label{lem Y Z bounded}
Suppose  that there exists $K\ge 0$ 
such that 
for all $ t\in[0,T]$,
$ x\in\R^d$ and  $m\in \mathfrak C$,  
\begin{equation*}
|(D_x b_t)(x,m)|+|(D_x\sigma_t)(x,m)|+|(D_xf_t)(x,m)|+|(D_x g)(x)|\le K\,.
\end{equation*}
Then $\sup_{\pi\in\mathcal{A}_{\mathfrak C} }\|Y(\pi)\|_{\mathbb{H}^{\infty}}<\infty $ 
and $\sup_{\pi\in \mathcal{A}_{\mathfrak C}}\|Z(\pi)\cdot W\|_{\operatorname{BMO}}<\infty$.
\end{lemma}

\begin{proof}

We first establish the bound of $Y(\pi)$. 
By  the definition of $H^0$ in  \eqref{eq: hamiltonian}, 
for all $s\in [0,T]$ and $i=1,\ldots, d$,
\begin{equation*}
\begin{split}
D_{x_i}H_s^0(\Theta_s(\pi),\pi_s) 
&= \left(D_{x_i}b^j_s\right)(X_s(\pi),\pi_s) (Y_s(\pi))^j+\left(D_{x_i}\sigma^{jp}_s\right)(X_s(\pi),\pi_s) (Z_s(\pi))^{jp} + \left(D_{x_i}f_s\right)(X_s(\pi),\pi_s)\,, 
\end{split}
\end{equation*}
where we adopt the Einstein summation convention, i.e., 
repeated equal dummy indices indicate summation  over all the values of that index.
As $(Y(\pi),Z(\pi))$ satisfies the linear BSDE \eqref{eq: adjoint equation}, by \cite[Proposition 3.1]{harter2019stability}, 
\begin{equation*}
Y_t(\pi)=\E_t\left[S_t^{-1}S_T \left(D_xg\right)(X_T(\pi))+\int_t^T S_t^{-1}S_s\left(D_xf_s\right)(X_s(\pi),\pi_s)\,ds\right]\,,
\end{equation*}
where the process $S=(S^{ij})_{i,j=1}^d$ satisfies $S_0=I_d$ and   for all $i,j=1,\ldots,d$ and $t\in [0,T]$, 
\begin{equation*}
dS_t^{ij}=S^{il}_t \left(D_{x_l}b^j_t\right)(X_t(\pi),\pi_t)\,dt+S^{il}_t \left(D_{x_l}\sigma^{jp}_t\right)(X_t(\pi),\pi_t)\,dW_t^p\,, 
\end{equation*}
and $S^{-1}_t$, the inverse process of $S_t$ in the sense that $(S^{-1}_t)^{ij} = 1/S_t^{ij}$ exists in the strong sense.  
By the Cauchy-Schwarz inequality,  
\begin{equation*}
\begin{split}
|Y_t(\pi)|&\le\left(\E_t[|S_t^{-1}S_T|^2]\right)^{1/2} \left(\E_t[|\left(D_xg\right)(X_T(\pi))|]^2\right)^{1/2}\\
&\quad+\left(\E_t\left[ \int_t^T |S_t^{-1}S_s|^2\,ds\right]\right)^{1/2}\left(\E_t\left[ \int_t^T |\left(D_xf_s\right)(X_s(\pi),\pi_s)|^2\,ds\right]\right)^{1/2}\,.
\end{split}
\end{equation*}
Due to the assumptions of the lemma we can apply~\cite[Corollary 3.4]{harter2019stability} together with~\cite[Proposition 3.2]{harter2019stability} thus obtaining,
${\sup_{\pi \in \mathcal A_{\mathfrak C}}}\E_t\sup_{t\le s\le T}|S_t^{-1}S_s|^2<\infty$ for all $ t\in[0,T]$, and
hence  
\begin{equation*}
\sup_{\pi\in\mathcal{A}_{\mathfrak C} }\|Y(\pi)\|_{\mathbb{H}^{\infty}}<\infty\,.
\end{equation*}

We proceed to establish the estimate of $Z(\pi)$ for a fixed $\pi \in  \mathcal{A}_{\mathfrak C} $. Applying It\^o's formula to $s\mapsto |Y_s(\pi)|^2$, for all $s\in [0,T]$, 
\begin{equation*}
\begin{split}
\int_s^T|Z_t(\pi)|^2\,dt &\le |\left(D_x g\right)(X_T(\pi))|^2-2\int_s^TZ_t(\pi)^\top Y_t(\pi)\,dW_t +\int_s^T2Y_t(\pi)\left(D_xH^{0}_t\right)(\Theta_t(\pi), \pi_t)\,dt\,.
\end{split}
\end{equation*}
For any stopping time $\eta:\Omega\to [0,T]$, 
by taking the conditional expectation,
\begin{equation*}
\begin{split}
\E_\eta \left[ \int_\eta^T|Z_t(\pi)|^2\,dt\right] 
& \le \E_\eta [| (D_x g )(X_T(\pi))|^2] +\E_\eta \int_\eta ^T\Big[2Y_t(\pi)^\top  \Big( (D_xb_t )(X_t(\pi),\pi_t) Y_t(\pi)+ \\
& \qquad \qquad \qquad \qquad \qquad + (D_x\sigma_t )(X_t(\pi),\pi_t)  Z_t(\pi)+ (D_xf_t )(X_t(\pi),\pi_t)\Big)\Big]\,dt\,.
\end{split}
\end{equation*}
Using  the assumptions of the lemma and    Young's inequality, 
\begin{equation*}
\begin{split}
\E_\eta \left[ \int_\eta ^T|Z_t(\pi)|^2\,dt \right]
& \le 
\| (D_x g )(X_T(\pi))\|^2_{L^\infty}
+
2\|Y(\pi)\|_{\mathbb{H}^\infty}^2 
 \| (D_xb )(X(\pi),\pi)\|_{\mathbb{H}^\infty} + 2\|Y(\pi)\|_{\mathbb{H}^\infty} \| (D_xf )(X(\pi),\pi)\|_{\mathbb{H}^\infty}\\
&\quad +\E_\eta \left[ \int_\eta ^T  \|\left(D_x\sigma\right)(X(\pi),\pi)\|_{\mathbb{H}^\infty}(\gamma^{-1}\|Y(\pi)\|_{\mathbb{H}^\infty}^2+\gamma|Z_t(\pi)|^2) \right]\,dt
\\
& 
\le K^2
+ 2K \|Y(\pi)\|_{\mathbb{H}^\infty}^2 
+ 2K \|Y(\pi)\|_{\mathbb{H}^\infty}
+K \E_\eta  \left[ \int_\eta ^T   (\gamma^{-1}\|Y(\pi)\|_{\mathbb{H}^\infty}^2+\gamma|Z_t(\pi)|^2) \right]\,dt\,.
\end{split}
\end{equation*}
By choosing $\gamma$ such that $0<1-K\gamma<1$ and 
using    the fact that 
$\sup_{\pi\in\mathcal{A}_{\mathfrak C} }\|Y(\pi)\|_{\mathbb{H}^{\infty}}<\infty $,
we conclude that there exists $C>0$, independent of $\pi$ and $\tau$, 
such that $ \E_\eta \left[\int_\eta ^T|Z_t(\pi)|^2\right]\,dt\le C$. 
Hence, by the definition of the $\operatorname{BMO}$ norm,
\[
\sup_{\pi\in\mathcal{A}_{\mathfrak C} }\|Z(\pi)\cdot W\|_{\operatorname{BMO}}=\sup_{\pi\in\mathcal{A}_{\mathfrak C} }\sup_{\eta}\left\|\left(\E_{\eta}\left[\int_{\eta}^T|Z_t(\pi)|^2\,dt\right]\right)^{1/2}\right\|_{L^\infty}<\infty \,,
\]
where the supremum is taken over all stopping times $\eta:\Omega\to  [0,T]$.
\end{proof}

We proceed to investigate  the Lipschitz continuity of $\pi\mapsto (Y(\pi),Z(\pi))$. 
The following lemma shows that the state process $X(\pi)$ is Lipschitz continuous  in $\pi$  and is a standard SDE stability-type estimate. The proof is given in~\ref{sec proof of sde stability estimate}.

\begin{lemma}
\label{lem standard sde estimate}
Suppose  that there exists $K\ge 0$ 
such that for all  $t\in[0,T]$,
$  x,x'\in\R^d$ and $ m,m'\in \mathfrak C$,
\begin{equation*}
|b_t(x,m)-b_t(x',m')|^2+|\sigma_t(x,m)-\sigma_t(x',m')|^2\le K(|x-x'|^2+ D_h(m|m'))\,.
\end{equation*} 
Then there exists   $C\ge 0$ such that
for all $\pi,\pi'\in \mathcal{A}_{\mathfrak C} $, 
\begin{equation*}
\E\left[ \sup_{0\le t\le T}|X_t(\pi)-X_t(\pi')|^2\right]\le C\E\int_0^T\,D_h(\pi_s|\pi'_s)\,ds\,.
\end{equation*}
\end{lemma}

We further recall the   well-known Fefferman's inequality 
for BMO martingales (see \cite[Theorem  2.5]{kazamaki2006continuous}).
\begin{lemma}[Fefferman's inequality] \label{thr Fefferman's inequality}
If $M$ is a BMO martingale  and  $N$ is a martingale such that $\E\left[ \langle N\rangle_T^{\frac{1}{2}} \right]<\infty$, then
\begin{equation*}
\E\left[\int_0^T|d\langle M, N\rangle_t|\right]\le\sqrt{2}\|M\|_{\operatorname{BMO}}\,\E\left[\langle N\rangle_T^{\frac{1}{2}}\right]\,.
\end{equation*}
\end{lemma}

 Fefferman's inequality yields the following lemma, which is important for the regularity analysis of the cost objective. 
\begin{lemma}
\label{lemma:L2_BMO_L1}
Let   $f, g:\Omega\times [0,T]\to \mathbb R $ be square-integrable   progressively measurable processes such that 
    $\|f\cdot W\| _{\operatorname{BMO}}<\infty $, where  $W$ is   a Brownian motion.
    Then $\mathbb E[\int_0^T |f_tg_t| d t]\le \sqrt{2} \|f\cdot W\| _{\operatorname{BMO}}
    \mathbb E\left[\left(\int_0^T |g_t|^2 d t \right)^{\frac{1}{2}}\right]$.
\end{lemma}
\begin{proof}
    By  It\^o's isometry,
    $\mathbb E[\int_0^T |f_tg_t| d t]=\mathbb E[\int_0^Td\langle |f|\cdot W,  |g|\cdot W\rangle_t  ] $. The desired result follows from Lemma \ref{thr Fefferman's inequality} and   the fact that
    $\||f|\cdot W\| _{\operatorname{BMO}}=\|f\cdot W\| _{\operatorname{BMO}}$.
\end{proof}

     \begin{remark}
          Lemma \ref{lemma:L2_BMO_L1} 
         leverages the BMO regularity of $f\cdot W$
         and 
         bounds the $L^1$-norm of $(\omega, t)\mapsto f_tg_t(\omega)$  using the $L^2$-norm of $g$ in time  and $L^1$-norm in the probability space. 
         This norm on $g$
         is weaker than the standard $L^2$-norm   
         $ \mathbb E\left[ \int_0^T |g_t|^2 d t \right]^{{1}/{2}}$
         arising from  the Cauchy--Schwarz inequality, and is crucial for our regularity estimates, particularly for  the bound of  $I_4$ in Lemma \ref{lem stability for Y and Z}. 
         
         A similar estimate cannot  be obtained by replacing   Fefferman's inequality with  the Kunita--Watanabe inequality, stated as follows: 
 \begin{equation*}
 \int_0^T|d\langle M, N\rangle_t| \le 
 \left(\int_0^T d\langle M\rangle_t\right)^{1/2}
 \left(\int_0^T d\langle N\rangle_t\right)^{1/2}\,,
\end{equation*}
since   after taking expectations, it is unclear how to isolate the contribution of 
$M$
  without invoking its BMO norm. Applying the Cauchy--Schwarz inequality in this context leads to
  the  appearance of
  $\mathbb E[ \langle N\rangle_T]^{1/2}$, which coincides with  the undesired
$L^2$-norm of $g$, 
as discussed above. 
     \end{remark}

Based on Lemmas \ref{lem Y Z bounded}, 
\ref{lem standard sde estimate} and
\ref{thr Fefferman's inequality}, 
we prove that the adjoint processes depend Lipschitz continuously on $\pi$ in terms of the Bregman divergence.

\begin{lemma}\label{lem stability for Y and Z}
Suppose Assumption \ref{assumption controlled SDE for modified MSA}, \ref{assum:regularity_f_g} and \ref{assumption:spatial_derivative} hold. 
Then there exists   $C > 0$ such that for all $\pi,\pi'\in\mathcal{A}_{\mathfrak C}$,
\begin{equation*}
\E\sup_{t\in[0,T]} |Y_t(\pi)-Y_t(\pi')|^2 + \E\int_0^T |Z_t(\pi)-Z_t(\pi')|^2dt \le C\E\int_0^T D_h(\pi_t|\pi'_t)\,dt\,.
\end{equation*}
\end{lemma}

\begin{proof}
Fix $\pi,\pi'\in\mathcal{A}_{\mathfrak C}$. 
Let  $\beta>0$ be a constant whose value will be specified later.
By It\^o's formula, 
\begin{equation}
\label{eq:bsdes_proof_before_BDG}
\begin{split}
&e^{\beta s}|Y_s(\pi)-Y_s(\pi')|^2 
+\int_s^T e^{\beta t} \left(|Z_t(\pi)-Z_t(\pi')|^2+\beta|Y_t(\pi)-Y_t(\pi')|^2\right)dt
\\
&
= e^{\beta T}| (D_xg )(X_T(\pi))- (D_x g )(X_T(\pi'))|^2   
-2\int_s^T e^{\beta t} (Y_t(\pi)-Y_t(\pi'))^\top (Z_t(\pi)-Z_t(\pi')) dW_t\\
&\quad +\int_s^T e^{\beta t}\bigg(2(Y_t(\pi)-Y_t(\pi'))^\top\left(\left(D_x H^{0}_t\right)(\Theta_t(\pi), \pi_t)-\left(D_xH^{0}_t\right)(\Theta_t(\pi'), \pi'_t)\right)\,.
\end{split}
\end{equation}
The proof will be divided into two steps. 
The first step is to
choose a sufficiently large $\beta$
to absorb the terms involving $|Z_t(\pi)-Z_t(\pi')|^2$ and $|Y_t(\pi)-Y_t(\pi')|^2$
arising from upper bounding 
 the right-hand side of~\eqref{eq:bsdes_proof_before_BDG},
and to 
establish that for any $\varepsilon > 0$,
\begin{equation}
\label{eq: final diff Y plus diff Z new}
\begin{split}
&\E \int_0^T \bigg(\frac12|Z_t(\pi)-Z_t(\pi')|^2 + |Y_t(\pi)-Y_t(\pi')|^2\bigg) \,dt
\\
&
\le  C(1+\varepsilon^{-1})\E \int_0^T D_h(\pi_t|\pi'_t)\,dt 
+ \varepsilon \hat C\E \left[ \sup_{t\in[0,T]} |Y_t(\pi)-Y_t(\pi')|^2\right]\,,
\end{split}
\end{equation}
where $\hat C:=  \sup_{\pi \in \mathcal A_{\mathfrak C}} \|Z(\pi)\cdot W\|_{\text{BMO}} < \infty$, due to Lemma~\ref{lem Y Z bounded}, and $C$ is a constant that depends on $K$, $T$, $d$ and $d'$ but is independent of $\varepsilon$.
In the second step we will take the supremum in~\eqref{eq:bsdes_proof_before_BDG}, use Burkholder--David--Gundy inequality and the estimate~\eqref{eq: final diff Y plus diff Z new} with a sufficiently small  $\varepsilon>0$ to  conclude the proof.

{\em Step 1:} 
From~\eqref{eq:bsdes_proof_before_BDG} we get that
\begin{equation}
\begin{split}
& \int_0^T e^{\beta t}\Big[|Z_t(\pi)-Z_t(\pi')|^2 + \beta|Y_t(\pi) - Y_t(\pi')|^2\Big]\,dt \\
& \leq e^{\beta T}| (D_xg )(X_T(\pi))- (D_x g )(X_T(\pi'))|^2   -2\int_0^T e^{\beta t}(Y_t(\pi)-Y_t(\pi'))^\top (Z_t(\pi)-Z_t(\pi'))  dW_t\\
&\quad +\int_0^T e^{\beta t}2(Y_t(\pi)-Y_t(\pi'))^\top \Big(\left(D_x H^{0}_t\right)(\Theta_t(\pi), \pi_t)-\left(D_xH^{0}_t\right)(\Theta_t(\pi'), \pi'_t)\Big)\,dt\,.
\end{split}
\end{equation}
Observe that 
\[
\mathbb E\int_0^T e^{2\beta s}|Y_s(\pi)-Y_s(\pi')|^2 |Z_s(\pi)-Z_s(\pi')|^2\,ds \leq 2e^{2\beta T}\max(\|Y(\pi)\|_{\mathbb H^\infty}^2, \|Y(\pi')\|_{\mathbb H^\infty}^2)\| Z(\pi)-Z(\pi')\|_{\mathbb H^2}^2
\]
which is finite due to Lemma~\ref{lem Y Z bounded} and Proposition~\ref{prop existence and uniqueness of the solution}.
Thus
$\int_0^t e^{\beta s} (Y_s(\pi)-Y_s(\pi'))^\top (Z_s(\pi)-Z_s(\pi')) dW_s$ is a martingale.  
Hence, taking expectation yields   that 
\begin{equation}
\label{eq diff of Y and Z}
\begin{split}
& \mathbb E \int_0^T e^{\beta t}\Big[|Z_t(\pi)-Z_t(\pi')|^2 + \beta|Y_t(\pi) - Y_t(\pi')|^2\Big]\,dt 
\leq e^{\beta T}\mathbb E|(D_x g)(X_T(\pi))- (D_x g )(X_T(\pi'))|^2  
\\
&\quad + \mathbb E\int_0^T e^{\beta t}2(Y_t(\pi)-Y_t(\pi'))^\top\Big(\left(D_x H^{0}_t\right)(\Theta_t(\pi), \pi_t)-\left(D_xH^{0}_t\right)(\Theta_t(\pi'), \pi'_t)\Big)\,dt\,.
\end{split}
\end{equation}
To derive an upper bound of the right-hand side of~\eqref{eq diff of Y and Z}, observe that  
\begin{equation*}
\begin{split}
&\left(D_x H^{0}_t\right)(\Theta_t(\pi), \pi_t)-\left(D_xH^{0}_t\right)(\Theta_t(\pi'), \pi'_t)\\
&=\left(D_x b_t\right)(X_t(\pi), \pi_t)  (Y_t(\pi)-Y_t(\pi') )
+ \text{tr}\Big[(D_x \sigma_t) (X_t(\pi), \pi_t)   (Z_t(\pi)-Z_t(\pi') ) \Big] 
\\
&\quad +
\big( (D_x b_t )(X_t(\pi), \pi_t)- (D_x b_t )(X_t(\pi'), \pi'_t)  \big) Y_t(\pi') \\
&\quad + \text{tr}\Big[\Big(\left(D_x \sigma_t\right)(X_t(\pi), \pi_t)-\left(D_x \sigma_t\right)(X_t(\pi'), \pi'_t)\Big)    Z_t(\pi')\Big]
+ \left(D_x f_t\right)(X_t(\pi), \pi_t)-\left(D_x f_t\right)(X_t(\pi'), \pi'_t)\,.
\end{split}
\end{equation*}
Hence  by Assumptions~\ref{assumption controlled SDE for modified MSA}, \ref{assum:regularity_f_g} and \ref{assumption:spatial_derivative} and by Young's inequality, for all $\gamma>0$,
\begin{equation}
\label{eq diff y diff H}
\begin{split}
&\E \left[\int_0^T 2e^{\beta t}\Big|(Y_t(\pi)-Y_t(\pi'))^\top ( (D_x H^{0}_t )(\Theta_t(\pi), \pi_t)- (D_xH^{0}_t )(\Theta_t(\pi'), \pi'_t))\Big|\,dt\right]\\
&\leq I_1+I_2+I_3+I_4+I_5\,,
\end{split}
\end{equation}
where
\begin{equation}
\label{eq bsde estimate all the I123-terms}
\begin{split}
I_1&\coloneqq   \E \left[\int_0^T 2e^{\beta t}|(D_x b)(X(\pi),\pi)| \, |Y_t(\pi)-Y_t(\pi')|^2 \, dt\right]\,,	
\\
I_2 &\coloneqq \E\left[\int_0^T e^{\beta t} | (D_x\sigma )(X(\pi),\pi)| \left(\gamma|Z_t(\pi)-Z_t(\pi')|^2 +\gamma^{-1}|Y_t(\pi)-Y_t(\pi')|^2\right)\,dt\right]\,,
\\
I_3 &\coloneqq \E\left[\int_0^Te^{\beta t}\Big(\gamma^{-1}|Y_t(\pi)-Y_t(\pi')|^2+\gamma |  (D_x f_t )(X_t(\pi), \pi_t)- (D_x f_t )(X_t(\pi'), \pi'_t)|^2\Big)dt\right]\,,
\end{split}
\end{equation}
and
\begin{equation*}
\begin{split}
I_4 &\coloneqq \E\left[ \int_0^T 2e^{\beta t} \Big|(Y_t(\pi)-Y_t(\pi'))^\top 
\text{tr}\big[\big( (D_x \sigma_t )(X_t(\pi), \pi_t)- (D_x \sigma_t )(X_t(\pi'), \pi'_t)\big)   Z_t(\pi') \big]\Big|
dt	\right]\,,	
\\
I_5 &\coloneqq \E\left[\int_0^T2e^{\beta t}\Big|(Y_t(\pi)-Y_t(\pi'))^\top \Big( (D_x b_t )(X_t(\pi), \pi_t)- (D_x b_t )(X_t(\pi'), \pi'_t)\Big)Y_t(\pi')\Big|\,dt\right]\,.	
\end{split}
\end{equation*}
We focus on the  most troublesome terms   $I_4$ and $I_5$.
Observe that 
by Lemma \ref{lemma:L2_BMO_L1}, 
\begin{align*}
& I_4 \leq 2 \sum_{j=1}^d  
\sum_{i=1}^{d'}\sum_{k=1}^d \E\int_0^T e^{\beta t} \Big|(Y^j_t(\pi)-Y^j_t(\pi'))\big( (D_{x_j} \sigma^{ki}_t )(X_t(\pi), \pi_t)- (D_{x_j} \sigma^{ki}_t )(X_t(\pi'), \pi'_t)\big) \Big|\, \Big| Z^{ki}_t(\pi')\Big|
dt
\\
&  
\leq \sum_{j=1}^d  
\sum_{i=1}^{d'}\sum_{k=1}^d \|Z^{ki}(\pi') \!\cdot\! W^i \|_{\text{BMO}} \E \left[ \bigg(\int_0^T \!\!\! e^{2\beta t} \big|Y_t^j(\pi)-Y_t^j(\pi')\big|^2 \big|(D_{x_j} \sigma_t^{ki} )(X_t(\pi), \pi)- (D_{x_j} \sigma_t^{ki} )(X_t(\pi'), \pi')\big|^2dt\bigg)^{\frac{1}{2}}\!\right]
\\
&
\leq \|Z(\pi')\!\cdot\! W \|_{\text{BMO}} \sum_{j=1}^d  
\sum_{i=1}^{d'}\sum_{k=1}^d \E \bigg[ \bigg(\int_0^T \!\! e^{2\beta t} \big|Y_t^j(\pi)-Y_t^j(\pi')\big|^2 \big|(D_{x_j} \sigma_t^{ki} )(X_t(\pi), \pi)- (D_{x_j} \sigma_t^{ki} )(X_t(\pi'), \pi')\big|^2\,dt\bigg)^{\frac{1}{2}}\bigg]\,.
\end{align*}
By Young's inequality,   for all 
  $i\in \{1,\ldots, d'\}$, $ j,k\in \{1,\ldots, d\}$, and 
$\varepsilon > 0$,
\[
\begin{split}
  &  \E \bigg[ \sup_{0\leq t\leq T} \big|Y_t^j(\pi)-Y_t^j(\pi')\big| \bigg(\int_0^T e^{2\beta t} \big|(D_{x_j} \sigma_t^{ki} )(X_t(\pi), \pi)- (D_{x_j} \sigma_t^{ki} )(X_t(\pi'), \pi')\big|^2 dt\bigg)^{1/2}\bigg]\\ 
& \leq  \E \bigg[ \varepsilon \sup_{0\leq t\leq T} \big|Y_t^j(\pi)-Y_t^j(\pi')\big|^2 + \varepsilon^{-1} \int_0^T e^{2\beta t} \big|(D_{x_j} \sigma_t^{ki} )(X_t(\pi), \pi)- (D_{x_j} \sigma_t^{ki} )(X_t(\pi'), \pi')\big|^2\,dt\bigg] \,.
\end{split}
\]
Setting  $\hat C:=  \sup_{\pi \in \mathcal A_{\mathfrak C}} \|Z(\pi)\cdot W\|_{\text{BMO}} < \infty$ (see   Lemma~\ref{lem Y Z bounded})
and taking the summation over $i,j,k$ yield
\begin{align}
\label{eq diff Y Z diff sigma}
\begin{split}
I_4 & 
 \le  
  \varepsilon  \hat C d' d \E \bigg[  \sup_{0\leq t\leq T} \big|Y_t(\pi)-Y_t(\pi')\big|^2\bigg] + \hat C 
\E \int_0^T e^{2\beta t} \big|(D_{x} \sigma_t)(X_t(\pi), \pi)- (D_x \sigma_t )(X_t(\pi'), \pi')\big|^2\,dt 
\\
&\leq
 \varepsilon  \hat C d' d \E \left[ \sup_{t\in[0,T]} |Y_t(\pi)-Y_t(\pi')|^2\right] + K \hat C\varepsilon ^{-1} \E \left[ \int_0^Te^{2\beta t}  \big[|X_t(\pi) - X_t(\pi')|^2 + D_h(\pi_t|\pi'_t)\big] dt \right]\,.
\end{split}
\end{align}
where the last inequality used the regularity assumption of $D_x\sigma$.
Hereafter, 
we denote by 
  $C\ge 0$ a generic constant that depends  on $K$, $T$, $d$, $d'$ but is independent of $\beta, \gamma$ and $\varepsilon$, and    may vary from line to line. 
Again, using Lemma~\ref{lem Y Z bounded}, Young's inequality and   Assumption~\ref{assumption:spatial_derivative}, 
\begin{equation}\label{eq diffY Y diffb}
\begin{split}
|I_5| 
& \le \|Y(\pi')\|_{\mathbb{H}^{\infty}}\E \bigg[ \int_0^T e^{\beta t}2|Y_t(\pi)-Y_t(\pi')| | (D_x b_t )(X_t(\pi), \pi_t)- (D_x b_t )(X_t(\pi'), \pi'_t) |\,dt \bigg]\\
& \le C\E \bigg[ \int_0^T e^{\beta t}\left(  |Y_t(\pi)-Y_t(\pi')|^2+  | (D_x b_t )(X_t(\pi), \pi_t)- (D_x b_t )(X_t(\pi'), \pi'_t) |^2\right)\,dt \bigg]\\
&\le C\E \bigg[\int_0^T e^{\beta t}\left( |Y_t(\pi)-Y_t(\pi')|^2+ K(|X_t(\pi)-X_t(\pi')|^2+ D_h(\pi_t|\pi'_t) \right)\,dt \bigg]\,.
\end{split}
\end{equation}
Hence using~\eqref{eq diff y diff H},~\eqref{eq bsde estimate all the I123-terms},~\eqref{eq diff Y Z diff sigma}, and~\eqref{eq diffY Y diffb} and Assumption~\ref{assum:regularity_f_g},   
for all $\beta\ge 0$ and $  \gamma, \varepsilon>0$,
\begin{equation}
\label{eq: final diff Y plus diff Z}
\begin{split}
&\E \left[\int_0^T 2e^{\beta t}\Big|(Y_t(\pi)-Y_t(\pi'))^\top ( (D_x H^{0}_t )(\Theta_t(\pi), \pi_t)- (D_xH^{0}_t )(\Theta_t(\pi'), \pi'_t))\Big|\,dt\right]\\
& \leq C \E \left[\int_0^T 2e^{\beta t}|Y_t(\pi)-Y_t(\pi')|^2\,dt\right] \\
& \quad + C\E\left[\int_0^T e^{\beta t}\left(\gamma|Z_t(\pi)-Z_t(\pi')|^2 +\gamma^{-1}|Y_t(\pi)-Y_t(\pi')|^2\right)\,dt\right] \\
& \quad + \E\left[\int_0^Te^{\beta t}\Big(\gamma^{-1}|Y_t(\pi)-Y_t(\pi')|^2+ C \gamma \big[|X_t(\pi)-X_t(\pi')|^2 + |D_h(\pi_t|\pi'_t)\big]\Big)dt\right] \\
& \quad + \varepsilon d' d  \hat C\E \left[ \sup_{t\in[0,T]} |Y_t(\pi)-Y_t(\pi')|^2\right] +   C\varepsilon ^{-1} \E \left[ \int_0^Te^{2\beta t} \big[|X_t(\pi) - X_t(\pi')|^2 + D_h(\pi_t|\pi'_t)\big]  dt \right] \\
& \quad + C\E \bigg[\int_0^T e^{\beta t}\left( |Y_t(\pi)-Y_t(\pi')|^2+ D_h(\pi_t|\pi'_t)\right)\,dt \bigg]\,.    
\end{split}
\end{equation}
Therefore, using~\eqref{eq: final diff Y plus diff Z} in~\eqref{eq diff of Y and Z} and collecting terms along with Lemma~\ref{lem standard sde estimate} shows that
\begin{equation}
\label{eq:lipschitz_estimate1}
\begin{split}
&\E \left[\int_0^T e^{\beta t}\big[|Z_t(\pi)-Z_t(\pi')|^2 + \beta |Y_t(\pi)-Y_t(\pi')|^2 \big]\,dt\right] \\
&\le  C \E \left[\int_0^T  e^{\beta t}(1 +  \gamma^{-1} )|Y_t(\pi)-Y_t(\pi')|^2\,dt\right]  + C\E\left[\int_0^T e^{\beta t}\left(\gamma|Z_t(\pi)-Z_t(\pi')|^2 \right)\,dt\right] \\
& \quad 
+ C (1+\gamma+ \varepsilon^{-1}) e^{2\beta T} \E\left[\int_0^T  D_h(\pi_t|\pi'_t)dt\right]  + \varepsilon  d' d  \hat C\E \left[ \sup_{t\in[0,T]} |Y_t(\pi)-Y_t(\pi')|^2\right]\,,
\end{split}
\end{equation}
Taking the constant $C$  in \eqref{eq:lipschitz_estimate1},
and 
choosing $ \gamma=1/(2C)$ and   $\beta = 1 +C(1  + C)$ yield
for all $\varepsilon>0$,
\begin{equation}
\label{eq: pre final diff Y plus diff Z new}
\begin{split}
&\E \left[\int_0^T \bigg(\frac12|Z_t(\pi)-Z_t(\pi')|^2 + |Y_t(\pi)-Y_t(\pi')|^2\bigg) \,dt\right]
\\
&
\le  C\left(1+\frac{1}{2C}+\varepsilon^{-1}\right)\E \int_0^T D_h(\pi_t|\pi'_t)\,dt 
+ \varepsilon d' d  \hat C \E \left[ \sup_{t\in[0,T]} |Y_t(\pi)-Y_t(\pi')|^2\right]\,.
\end{split}
\end{equation}
This shows~\eqref{eq: final diff Y plus diff Z new} holds thus concluding Step 1.

{\em Step 2:}
We return to~\eqref{eq:bsdes_proof_before_BDG} and  
upper bound   $\mathbb E\left[ \sup_{s\in [0,T]}  |Y_s(\pi)-Y_s(\pi')|^2   \right]$.
For each $s\in [0,T]$,
\begin{equation}
\begin{split}
& e^{\beta s}|Y_s(\pi)-Y_s(\pi')|^2 + \int_s^T e^{\beta t}\Big(|Z_t(\pi)-Z_t(\pi')|^2 + \beta|Y_t(\pi)-Y_t(\pi')|^2\Big)\,dt
\\
&
\leq e^{\beta T}| (D_xg )(X_T(\pi))- (D_x g )(X_T(\pi'))|^2  
\\
&\quad 
+2\sup_{s\in[0,T]}\bigg| \int_s^T e^{\beta t}(Y_t(\pi)-Y_t(\pi'))^\top (Z_t(\pi)-Z_t(\pi')) dW_t\bigg|\\
&\quad +  \sup_{s\in[0,T]}\bigg|\int_s^T e^{\beta t} 2(Y_t(\pi)-Y_t(\pi'))^\top \Big( (D_x H^{0}_t )(\Theta_t(\pi), \pi_t)- (D_xH^{0}_t )(\Theta_t(\pi'), \pi'_t) \Big)\,dt\bigg|\,.	
\end{split}
\end{equation}
Taking the supremum over $s\in [0,T]$ and  then taking the expectation yield 
\begin{equation}
\begin{split}
& \mathbb E\left[ \sup_{s\in [0,T]} e^{\beta s}|Y_s(\pi)-Y_s(\pi')|^2   \right]
\leq  e^{\beta T} \mathbb E\Big[|(D_xg )(X_T(\pi))- (D_x g )(X_T(\pi'))|^2\Big]  
\\
&\qquad \qquad \qquad + 2\mathbb E \left[\sup_{s\in[0,T]}\bigg| \int_s^T e^{\beta t} (Y_t(\pi)-Y_t(\pi')) ^\top (Z_t(\pi)-Z_t(\pi')) dW_t\bigg|\right]\\
&\qquad \qquad \qquad + \mathbb E\int_0^T e^{\beta t} 2 \big| (Y_t(\pi)-Y_t(\pi'))^\top \Big( (D_x H^{0}_t )(\Theta_t(\pi), \pi_t)- (D_xH^{0}_t )(\Theta_t(\pi'), \pi'_t) \Big)  \big|\,dt\,.	
\end{split}
\end{equation}
Setting $\beta=0$ and 
using~\eqref{eq: final diff Y plus diff Z}, Lemma~\ref{lem standard sde estimate} and the conclusion of Step 1, which is~\eqref{eq: final diff Y plus diff Z new}, shows that
\begin{equation}
\label{eq proof of lem 3.4 step 2 before BDG}
\begin{split}
 \mathbb E \left[\sup_{s\in [0,T]}  |Y_s(\pi)-Y_s(\pi')|^2  \right] 
& \leq 2\mathbb E  \left[ \sup_{s\in[0,T]}\bigg| \int_s^T   (Y_t(\pi)-Y_t(\pi'))^\top (Z_t(\pi)-Z_t(\pi'))dW_t\bigg|\right] \\
&\quad + C(1+\varepsilon^{-1})\mathbb E  \left[ \int_0^T D_h(\pi_t|\pi'_t)\,dt\right] + \varepsilon \hat C \E \left[ \sup_{t\in[0,T]} |Y_t(\pi)-Y_t(\pi')|^2\right]\,.
\end{split}
\end{equation}
By the  Burkholder--Davis--Gundy inequality and Young's inequality,    for all $\gamma>0$,
\[
\begin{split}
& \mathbb E\left[\sup_{s\in[0,T]}\bigg| \int_s^T  (Y_t(\pi)-Y_t(\pi'))^\top (Z_t(\pi)-Z_t(\pi')) dW_t\bigg|\right]
\leq C  \mathbb E \Bigg[ \bigg(\int_0^T   |Z_t(\pi)-Z_t(\pi')|^2|Y_t(\pi)-Y_t(\pi')|^2\,dt\bigg)^{1/2} \Bigg]\\
& \leq C  \mathbb E \Bigg[ \sup_{t\in [0,T]}|Y_t(\pi)-Y_t(\pi')|  \bigg(\int_0^T   |Z_t(\pi)-Z_t(\pi')|^2\,dt\bigg)^{1/2} \Bigg]\\
& \leq C\left(\gamma
\mathbb E \left[ \sup_{s\in [0,T]}|Y_s(\pi)-Y_s(\pi')|^2\right]
+ \gamma^{-1} 
\mathbb E \left[\int_0^T   |Z_t(\pi)-Z_t(\pi')|^2\,dt \right] \right)\,.
\end{split}
\]
Choosing $\gamma := 1/(2C)$, using the above estimate in~\eqref{eq proof of lem 3.4 step 2 before BDG} yield, for all $\varepsilon > 0$, that
\[
\begin{split}
& \frac12\mathbb E\left[ \sup_{s\in [0,T]}  |Y_s(\pi)-Y_s(\pi')|^2\right]\\
& \le C  \mathbb E\left[ \int_0^T  |Z_t(\pi)-Z_t(\pi')|^2\,dt \right] + C(1+\varepsilon^{-1})\mathbb E\left[   \int_0^T D_h(\pi_t|\pi'_t)\,dt \right] + \varepsilon \hat C\E \left[ \sup_{t\in[0,T]} |Y_t(\pi)-Y_t(\pi')|^2\right]
\end{split}
\]
Using the conclusion of Step 1, which is~\eqref{eq: final diff Y plus diff Z new}, one more time we have, for all $\varepsilon > 0$, that
\[
\mathbb E\left[ \sup_{s\in [0,T]}  |Y_s(\pi)-Y_s(\pi')|^2\right] 
\leq  C(1+\varepsilon^{-1})\mathbb E\left[   \int_0^T D_h(\pi_t|\pi'_t)\,dt \right] + \varepsilon 2\hat C\E \left[ \sup_{t\in[0,T]} |Y_t(\pi)-Y_t(\pi')|^2\right]\,. 
\]
Finally, taking $\varepsilon = \frac1{2 \hat C}$ leads to
\[
\begin{split}
\frac12\mathbb E \sup_{s\in [0,T]}  |Y_s(\pi)-Y_s(\pi')|^2
& \leq  C\mathbb E\left[   \int_0^T D_h(\pi_t|\pi'_t)\,dt \right]\,.
\end{split}
\]
This, together with~\eqref{eq: final diff Y plus diff Z new} concludes the proof. 
\end{proof}

\begin{remark}
\label{rmk:bmo}
The above argument leverages  the theory of BMO martingales 
and allows for incorporating more general   diffusion coefficients, beyond the scope  of the  classical $L^2$-approach.  
Indeed,  to obtain the desired Lipschitz regularity,
standard  $L^2$-estimate in  \cite[Theorem 4.2.3]{zhang2017backward} 
requires us to 
bound  
$\E\big[\big( \int_0^T |(D_x \sigma_s)(X_s(\pi),\pi_s)-(D_x \sigma_s)(X_s(\pi'),\pi'_s)||Z_s(\pi) | \,ds \big)^2 \big]$.
Under   Assumption~\ref{assumption:spatial_derivative}, 
it reduces to upper bound 
$
\E\big[\sup_{s\in [0,T]}| X_s(\pi)- X_s(\pi')|^2 \big( \int_0^T   |Z_s(\pi) | \,ds \big)^2 \big]$.
As $Z(\pi)$ is not uniformly bounded, 
it is unclear how to bound the term using $\E\int_0^T D_h(\pi_t|\pi'_t)\,dt$.  
Here we  develop a more refined analysis and handle the $Z$  dependence using 
the uniform BMO bound   in Lemma \ref{lem Y Z bounded}.
Similar remarks 
apply to the proof of Lemma \ref{lem regularity of the flat deriv of hamiltonian}.
\end{remark}

\subsection{Proof of Theorem~\ref{th estimate for difference of J}}
\label{sec:proof-relative-smoothness}

This section proves the relative smoothness of the cost functional $\pi\mapsto J^\tau(\pi)$, namely Theorem~\ref{th estimate for difference of J}.

We start by  recalling that the G\^ateaux derivative of $J^0$ can be expressed by   the flat derivative of the Hamiltonian.
The proof relies on charactersing the pathwise derivative of  $\eps\mapsto X(\pi+\eps(\pi'-\pi))$. Compared with  \cite[Lemma 3.5]{vsivska2020gradient}, additional complexities arise due to the possible  unboundedness  of   flat derivatives of the coefficients $(b, \sigma, f)$ with respect to  the control $m \in \mathfrak C$ and the fact that the regularity of these coefficients in  $m$ is characterized by the Bregman divergence $D_h$.
The detailed arguments can be found in~\ref{sec:proof_of_dir_der_of_J0}.

\begin{lemma}
\label{lem lemma form siska szpruch}  

Suppose  Assumptions \ref{assumption controlled SDE for modified MSA}, \ref{assum:regularity_f_g},
\ref{assumption:spatial_derivative}
and \ref{assum:regularity_flat_derivative} hold. 
Then for all   $\pi,\pi'\in\mathcal{A}_\mathfrak{C}$ such that $\mathbb E \int_0^T D_h(\pi'_t|\pi_t)\,dt < \infty$,
\begin{equation*}
\frac{d}{d\varepsilon}
J^0(\pi+\varepsilon(\pi'-\pi))\big|_{\varepsilon=0}=\E\int_0^T\int\frac{\delta H^{0}_s}{\delta m}(\Theta_s(\pi),\pi_s,a)(\pi'_s-\pi_s)(da)\,ds\,.
\end{equation*}
\end{lemma}

The following    lemma  proves        the regularity   of $\frac{\delta H^{0}}{\delta m}$ with respect to $\pi$, which will be used in the proof of Theorem  \ref{th estimate for difference of J}.

\begin{lemma}\label{lem regularity of the flat deriv of hamiltonian}

Suppose  Assumptions \ref{assumption controlled SDE for modified MSA}, \ref{assum:regularity_f_g}, \ref{assumption:spatial_derivative} and \ref{assum:regularity_flat_derivative} hold. Then  there exists $C\ge 0$ such that for all $\pi,\pi'\in\mathcal{A}_\mathfrak{C}$,  
\begin{equation*}
\begin{split}
&\E  \int_0^1\int_0^T\left[\int\left(\frac{\delta H^{0}_s}{\delta m}(\Theta_s(\pi^{\varepsilon}), \pi^{\varepsilon}_s,a)-\frac{\delta H^{0}_s}{\delta m}(\Theta_s(\pi), \pi_s,a)\right)(\pi'_s-\pi_s)(da) \right]ds\,d\varepsilon
\le C\E\int_0^T D_h(\pi'_s|\pi_s)\,ds\,,
\end{split}
\end{equation*}
where $\pi^{\varepsilon}=\pi+\varepsilon(\pi'-\pi)$ for $\varepsilon\in(0,1)$.
\end{lemma}
\begin{proof}

Note that 
\begin{equation*}
\begin{split}
&\E \int_0^1\int_0^T \int\left(\frac{\delta H^{0}_s}{\delta m}(\Theta(\pi^{\varepsilon}),\pi_s^{\varepsilon},a) - \frac{\delta H^{0}_s}{\delta m}(\Theta(\pi),\pi_s,a)\right)(\pi'_s-\pi_s)(da)\,ds\,d\varepsilon 
= I^{(1)}+I^{(2)}\,, 
\end{split}
\end{equation*}
where 
\begin{align*}
I^{(1)} &\coloneqq \E  \int_0^1\int_0^T \int\left(\frac{\delta H^{0}_s}{\delta m}(\Theta_s(\pi^{\varepsilon}), \pi^{\varepsilon}_s,a)-\frac{\delta H^{0}_s}{\delta m}(\Theta_s(\pi), \pi^{\varepsilon}_s,a)\right)(\pi'_s-\pi_s)(da) ds\,d\varepsilon\,,\\
I^{(2)} &\coloneqq \E  \int_0^1\int_0^T \int\left(\frac{\delta H^{0}_s}{\delta m}(\Theta_s(\pi), \pi^{\varepsilon}_s,a)-\frac{\delta H^{0}_s}{\delta m}(\Theta_s(\pi), \pi_s,a)\right)(\pi'_s-\pi_s)(da) ds\,d\varepsilon\,.
\end{align*}

To estimate the term $I^{(1)}$, observe that for each $\varepsilon\in (0,1)$,
\begin{equation*}
\begin{split}
&\int\left[\frac{\delta H^{0}_s}{\delta m}(\Theta_s(\pi^{\varepsilon}), \pi^{\varepsilon}_s,a)-\frac{\delta H^{0}_s}{\delta m}(\Theta_s(\pi), \pi^{\varepsilon}_s,a)\right](\pi'_s-\pi_s)(da) \\
& 
= \int\Bigg[\left(\frac{\delta b_s}{\delta m}(X_s(\pi^{\varepsilon}),\pi^{\varepsilon}_s,a)-\frac{\delta b_s}{\delta m}(X_s(\pi),\pi^{\varepsilon}_s,a)\right)Y_s(\pi^{\varepsilon}) +\frac{\delta b_s}{\delta m}(X_s(\pi),\pi^{\varepsilon}_s,a)\Big(Y_s(\pi^{\varepsilon})-Y_s(\pi)\Big) \\
&  \qquad+\left(\frac{\delta \sigma_s}{\delta m}(X_s(\pi^{\varepsilon}),\pi^{\varepsilon}_s,a)-\frac{\delta \sigma_s}{\delta m}(X_s(\pi),\pi^{\varepsilon}_s,a)\right) Z_s(\pi^{\varepsilon})
+\frac{\delta \sigma_s}{\delta m}(X_s(\pi),\pi^{\varepsilon}_s,a)   (Z_s(\pi^{\varepsilon})-Z_s(\pi) )\\
&  \qquad+\frac{\delta f_s}{\delta m}(X_s(\pi^{\varepsilon}),\pi^{\varepsilon}_s,a)-\frac{\delta f_s}{\delta m}(X_s(\pi),\pi^{\varepsilon}_s,a)\Bigg](\pi'_s-\pi_s)(a)\,da\,.
\end{split}
\end{equation*}
Fix $\varepsilon\in (0,1)$. By the fundamental theorem of calculus we have that
\begin{equation*}
\frac{\delta b_s}{\delta m}(X_s(\pi^{\varepsilon}),\pi^{\varepsilon}_s,a)-\frac{\delta b_s}{\delta m}(X_s(\pi),\pi^{\varepsilon}_s,a)=\int_0^1 \left( D_x \frac{\delta b_s}{\delta m}(X_s^{\varepsilon, \varepsilon'},\pi^{\varepsilon}_s,a)\right)(X_s(\pi^{\varepsilon})-X_s(\pi))\,d\varepsilon'\,,
\end{equation*}
where $X_s^{\varepsilon, \varepsilon'}\coloneqq X_s(\pi)+\varepsilon'(X_s(\pi^{\varepsilon})-X_s(\pi))$ for all $\varepsilon'\in(0,1)$. Thus
\begin{equation*}
\begin{split}
\int&\left[\frac{\delta b_s}{\delta m}(X_s(\pi^{\varepsilon}),\pi^{\varepsilon}_s,a)-\frac{\delta b_s}{\delta m}(X_s(\pi),\pi^{\varepsilon}_s,a)\right] Y_s(\pi^{\varepsilon})(\pi'_s-\pi_s)(d a) \\
&\le  |X_s(\pi^{\varepsilon})-X_s(\pi)||Y_s(\pi^{\varepsilon})| \int_0^1\left|\int \left( D_x \frac{\delta b_s}{\delta m}(X_s^{\varepsilon, \varepsilon'},\pi^{\varepsilon}_s,a)\right) (\pi'_s-\pi_s)(da) \right|\,d\varepsilon'\,.
\end{split}
\end{equation*}
Applying similar arguments to   $\sigma$ and $f$ and using Young's inequality, we have that
\begin{align*}
&\E \int_0^1\int_0^T\left[\int\left(\frac{\delta H^{0}_s}{\delta m}(\Theta_s(\pi^{\varepsilon}), \pi^{\varepsilon}_s,a)-\frac{\delta H^{0}_s}{\delta m}(\Theta_s(\pi), \pi_s^{\varepsilon},a)\right)(\pi'_s-\pi_s)(da) \right]ds\,d\varepsilon
\le I_1 + I_2 + I_3 + I_4\,,
\end{align*}
where
\begin{equation*}
\begin{split}
I_1 & \coloneqq  \int_0^1\E\int_0^T\int\bigg(\frac{\delta b_s}{\delta m}(X_s(\pi),\pi^{\varepsilon}_s,a)(Y_s(\pi^{\varepsilon})-Y_s(\pi))\\
&\quad
+\frac{\delta \sigma_s}{\delta m}(X_s(\pi),\pi^{\varepsilon}_s,a) (Z_s(\pi^{\varepsilon})-Z_s(\pi))\bigg)(\pi'_s-\pi_s)(da) \,ds\,d\varepsilon\,,	\\
I_2 & \coloneqq  \int_0^1\E\int_0^T|Y_s(\pi^{\varepsilon})|\bigg[
\frac{1}{2}|X_s(\pi^{\varepsilon})-X_s(\pi)|^2 \\
& \quad + \frac{1}{2}\left(\int_0^1\left|\int \left( D_x \frac{\delta b_s}{\delta m}(X_s^{\varepsilon, \varepsilon'},\pi^{\varepsilon}_s,a)\right) (\pi'_s-\pi_s)(a)\,da\right|\,d\varepsilon'\right)^2\bigg]\,ds\,d\varepsilon\,,		
\\
I_3 &\coloneqq  \int_0^1\E\int_0^T|Z_s(\pi^{\varepsilon})| |X_s(\pi^{\varepsilon})-X_s(\pi)| 
   \int_0^1\left|\int \left( D_x \frac{\delta \sigma_s}{\delta m}(X_s^{\varepsilon, \varepsilon'},\pi^{\varepsilon}_s,a)\right) (\pi'_s-\pi_s)(da) \right|\,d\varepsilon'\,ds\,d\varepsilon	
\\
I_4 &\coloneqq  \int_0^1\E\int_0^T\bigg[\frac{1}{2}|X_s(\pi^{\varepsilon})-X_s(\pi)|^2
+\frac{1}{2}\left(\int_0^1\left|\int \left( D_x \frac{\delta f_s}{\delta m}(X_s^{\varepsilon, \varepsilon'},\pi^{\varepsilon}_s,a)\right) (\pi'_s-\pi_s)(da) \right|\,d\varepsilon'\right)^2\bigg]\,ds\,d\varepsilon\,.	
\end{split}
\end{equation*}
Observe that by Fubini's theorem and the Cauchy--Schwarz inequality, 
\begin{equation*}
\begin{split}\label{eq prod of diff Y and diff pi}
|I_1|&\le \int_0^1\left(\E\int_0^T|Y_s(\pi^{\varepsilon})-Y_s(\pi)|^{2}\,ds\right)^{1/2} 
\left(\E\int_0^T\left|\int\frac{\delta b_s}{\delta m}(X_s(\pi),\pi^{\varepsilon}_s,a)(\pi'_s-\pi_s)(da) \right|^2\,ds\,\right)^{1/2}\,d\varepsilon\\
&\quad +\int_0^1\left(\E\int_0^T|Z_s(\pi^{\varepsilon})-Z_s(\pi)|^{2}\,ds\right)^{1/2}
 \left(\E\int_0^T\left|\int\frac{\delta \sigma_s}{\delta m}(X_s(\pi),\pi^{\varepsilon}_s,a)(\pi'_s-\pi_s)(da) \right|^2\,ds\,\right)^{1/2}\,d\varepsilon\,,
\end{split}
\end{equation*}
which along with 
Lemma~\ref{lem stability for Y and Z},  Assumption~\ref{assum:regularity_flat_derivative}, and the convexity of $  D_h(\cdot|\pi_s)$ implies   
\begin{equation*}
\begin{split}
|I_1|& \le 2 \int_0^1\left(C\E\int_0^TD_h(\pi^{\varepsilon}_s|\pi_s) \,ds\right)^{1/2} \left(K\E\int_0^TD_h(\pi'_s|\pi_s)\,ds\,\right)^{1/2}\,d\varepsilon\\
&\le   C  \int_0^1\varepsilon^{1/2}\left( \E\int_0^T D_h(\pi'_s|\pi_s)\,ds\right) \,d\varepsilon \le   C\E\int_0^T D_h(\pi'_s|\pi_s)\,ds\,,
\end{split}
\end{equation*}
where the last inequality 
used 
$D_h(\pi^{\varepsilon}_s|\pi_s)\le \varepsilon D_h(\pi'_s|\pi_s)+(1-\varepsilon) D_h(\pi_s|\pi_s)
= \varepsilon D_h(\pi'_s|\pi_s) $ 
since  
$\pi^{\varepsilon}=\pi+\varepsilon(\pi'-\pi)$ for $\varepsilon\in(0,1)$.

For the term $I_2$, recall that by Lemma~\ref{lem Y Z bounded}, 
$\sup_{\pi\in\mathcal{A}_{\mathfrak C}}\|Y(\pi)\|_{\mathbb{H}^\infty}<\infty$. 
Then  using  Jensen's inequality,  Assumption \ref{assum:regularity_flat_derivative}, the  convexity of the Bregman divergence and Lemma~\ref{lem standard sde estimate}, 
\begin{equation*}
\begin{split}
I_2&\le \frac{\sup_{\varepsilon\in [0,1]}\|Y(\pi^{\varepsilon})\|_{\mathbb{H}^\infty}}{2} \int_0^1\E\left[\sup_{0\le s\le T}|X_s(\pi^{\varepsilon})-X_s(\pi)|^2+K\int_0^TD_h(\pi'_s|\pi_s)\,ds\right]\,d\varepsilon\\
&\le \frac{C}{2} \int_0^1\E\left[\int_0^TD_h(\pi^{\varepsilon}_s|\pi_s)\,ds+ \int_0^TD_h(\pi'_s|\pi_s)\,ds\right]\,d\varepsilon
\le C \E\int_0^TD_h(\pi'_s|\pi_s)\,ds\,.
\end{split}
\end{equation*}
Similar arguments show that  
\begin{equation*}
\begin{split}
I_4\le C\E\int_0^TD_h(\pi'_s|\pi_s)\,ds\,.
\end{split}
\end{equation*}

To estimate $I_3$, we   define for all $s\in [0,T]$ and  $\varepsilon',\varepsilon \in (0,1)$, 
\begin{equation*}
A^{\varepsilon,\varepsilon'}_s \coloneqq  \left|\int \left( D_x \frac{\delta \sigma_s}{\delta m}(X_s^{\varepsilon, \varepsilon'},\pi^{\varepsilon}_s,a)\right) (\pi'_s-\pi_s)(da) \right|\,.
\end{equation*}
By Assumption \ref{assum:regularity_flat_derivative}, 
$|A^{\varepsilon,\varepsilon'}_s|^2\le  KD_h(\pi'_s|\pi_s)$. 
Using    Lemmas~\ref{lem Y Z bounded} and \ref{lemma:L2_BMO_L1}, Young's inequality, Lemma~\ref{lem standard sde estimate} and the convexity of the Bregman divergence,
\begin{equation*}
\begin{split}
I_3 &= \int_0^1\E\int_0^T|Z_s(\pi^{\varepsilon})||X_s(\pi^{\varepsilon})-X_s(\pi)|  \int_0^1A^{\varepsilon,\varepsilon'}_s\,d\varepsilon'\,ds\,d\varepsilon\\
&\le C\int_0^1\|Z(\pi^{\varepsilon})\cdot W\|_{\operatorname{BMO}}\, \E\left[\left(\int_0^T|X_s(\pi^{\varepsilon})-X_s(\pi)|^2  \left(\int_0^1A^{\varepsilon,\varepsilon'}_s\,d\varepsilon'\right)^2\,ds\right)^{\frac{1}{2}}\right]\,d\varepsilon\\
&\le C\int_0^1\|Z(\pi^{\varepsilon})\cdot W\|_{\operatorname{BMO}}\, 
\E\left[\sup_{s\in [0,T]}|X_s(\pi^{\varepsilon})-X_s(\pi)|\left(\int_0^T \int_0^1 (A^{\varepsilon,\varepsilon'}_s)^2\,d \varepsilon'  ds\right)^{\frac{1}{2}}\right] \,d\varepsilon\\
&\le  C\int_0^1 \E\left[\sup_{s\in [0,T]}|X_s(\pi^{\varepsilon})-X_s(\pi)|^2\,+  \int_0^T\int_0^1 (A^{\varepsilon,\varepsilon'}_s)^2\,d \varepsilon'  ds\right] d\varepsilon\\
&\le  C\int_0^1\E\left[ \int_0^TD_h(\pi^{\varepsilon}_s|\pi_s)\,ds+ \int_0^TD_h(\pi'_s|\pi_s)\,ds\right]\,d\varepsilon
\le 
C  \E\int_0^TD_h(\pi'_s|\pi_s)\,ds\,.
\end{split}
\end{equation*}
Combining estimates for $I_1,\,I_2,\,I_3$ and $I_4$ yields that 
$I^{(1)}\le C  \E\int_0^TD_h(\pi'_s|\pi_s)\,ds$.

To estimate $I^{(2)}$, by setting $\pi^{\varepsilon',\varepsilon} = \varepsilon'\pi^\varepsilon+(1-\varepsilon')\pi$
with $\pi^\varepsilon = \pi + \varepsilon(\pi'-\pi)$,  
the fundamental theorem of calculus gives
\begin{equation*}
\begin{split}
 I^{(2)}
&= \E \int_0^1\int_0^T\int\left(\frac{\delta H^{0}_s}{\delta m}(\Theta(\pi),\pi_s^{\varepsilon},a) - \frac{\delta H^{0}_s}{\delta m}(\Theta(\pi),\pi_s,a)\right)(\pi'_s-\pi_s)(da)\,ds\,d\varepsilon\\
& =  \int_0^1 \varepsilon\, \E\int_0^T\int \int_0^1 \int \frac{\delta^2 H^{0}_s}{\delta m^2}(\Theta(\pi),\pi_s^{\varepsilon', \varepsilon},a,a')(\pi'_s-\pi_s)(da')\,d\varepsilon'(\pi'_s-\pi_s)(da)\,ds\,d\varepsilon\,,
\end{split}
\end{equation*}
where we used  $\pi^\varepsilon- \pi = \varepsilon(\pi'-\pi)$.
For each $s\in [0,T]$ and $\varepsilon, \varepsilon'\in (0,1)$, 
\begin{equation}
\label{eq:2nd_H}
\begin{split}
& \int\frac{\delta^2 H^{0}_s}{\delta m^2}(\Theta_s(\pi), \pi^{\varepsilon',\varepsilon}_s,a,a')(\pi'_s-\pi_s)(da')\\
&= \int\Bigg[\frac{\delta^2 b_s}{\delta m^2}(X_s(\pi),\pi^{\varepsilon',\varepsilon}_s,a,a')Y_s(\pi)  +\frac{\delta^2 f_s}{\delta m^2}(X_s(\pi),\pi^{\varepsilon',\varepsilon}_s,a,a')\Bigg](\pi'_s-\pi_s)(da')\,,
\end{split}
\end{equation}
where we used 
$\frac{\delta^2 \sigma_s}{\delta m^2}(X_s(\pi),\pi^{\varepsilon',\varepsilon}_s,a,a') =0$ due to Assumption \ref{assum:regularity_flat_derivative}.
Applying Assumption \ref{assum:regularity_flat_derivative} and   Lemma \ref{lem Y Z bounded} yields that
$
I^{(2)} \le  C  \E\int_0^TD_h(\pi'_s|\pi_s)\,ds\,.
$
This gives the desired estimate and  finishes the proof. 
\end{proof}

Now we are ready to prove Theorem~\ref{th estimate for difference of J}. 

\begin{proof}[Proof of Theorem~\ref{th estimate for difference of J}]

Fix   $\pi,\pi'\in\mathcal{A_{\mathfrak C}}$. We assume without loss of generality that 
$\E\int_{0}^{T} D_h(\pi'_s|\pi_s) \, ds<\infty$, since otherwise the inequality holds trivially. 

By Lemma \ref{lem lemma form siska szpruch}  and  the fundamental theorem of calculus (see Remark \ref{rmk:FTC}), 
\begin{equation*} 
\begin{split}
J^\tau(\pi')-J^\tau(\pi)&=\int_0^1\E\int_0^T\left[\int\frac{\delta H^{0}_s}{\delta m}(\Theta_s(\pi^{\varepsilon}), \pi^{\varepsilon}_s,a)(\pi'_s-\pi_s)(a)\,da\right]\,ds\,d\varepsilon
+ \tau\,\E\int_0^T\left[h(\pi'_s)-h(\pi_s)\right]\,ds\,.
\end{split}
\end{equation*}
where $\pi^\varepsilon=\pi+\varepsilon(\pi'-\pi)$ for all $\varepsilon\in(0,1)$.
By the definition of $D_h(\pi'_s|\pi_s)$, 
\begin{equation*}
\begin{split}
\E\int_{0}^{T}\left[h(\pi'_s) - h(\pi_s)\right]\,ds =\E\int_{0}^{T}\left(D_h(\pi'_s|\pi_s) + \int \frac{\delta h}{\delta m}(\pi_s,a)(\pi'_s-\pi_s)(da)\right)\,ds\,,
\end{split}
\end{equation*}
Hence we have 
\begin{equation}\label{eq expansion of the cost functional}
\begin{split}
& J^{\tau}(\pi')-J^{\tau}(\pi) - \E\int_0^T\int\frac{\delta H^{\tau}_s}{\delta m}(\Theta_s(\pi),\pi_s,a)(\pi'_s-\pi_s)(da)\,ds\\
& = \int_0^1\E\int_0^T\int\left(\frac{\delta H^{0}_s}{\delta m}(\Theta_s(\pi^{\varepsilon}),\pi_s^{\varepsilon},a) - \frac{\delta H^{0}_s}{\delta m}(\Theta_s(\pi),\pi_s,a)\right)(\pi'_s-\pi_s)(da)\,ds\,d\varepsilon
+ \tau\,\E\int_{0}^{T}D_h(\pi'_s|\pi_s)\,ds
\\
&\le  C\,\E\int_{0}^{T}D_h(\pi'_s|\pi_s)\,ds
\,,
\end{split}
\end{equation}
where the last inequality used 
Lemma \ref{lem regularity of the flat deriv of hamiltonian} and Fubini's theorem. 
This completes the proof. 
\end{proof}

\subsection{Proof of Theorem~\ref{thm energy dissipation mirr}}
\label{sec:proof-energy-dissipation-h-breg-mirr}

\begin{proof}[Proof of Theorem~\ref{thm energy dissipation mirr}]
Let $\{\pi^{n}\}_{n\ge 0}$ be the iterates from   Algorithm~\ref{alg mmsa}. Applying    Lemma~\ref{th estimate for difference of J} for $\pi'=\pi^{n+1}$ and $\pi=\pi^{n}$ gives that 
\begin{equation}\label{eq lq difference of J in the proof}
\begin{split}
J^\tau(\pi^{n+1})-J^\tau(\pi^{n})
&\le\E\int_0^T\int\frac{\delta H^{\tau}_s}{\delta m}(\Theta_s(\pi^{n}), \pi^{n}_s, a)(\pi^{n+1}_s-\pi^{n}_s)(da)\,ds
+L\E\int_0^T D_h(\pi^{n+1}_s|\pi^{n}_s) \,ds\,.
\end{split}
\end{equation}
The definition of $\pi^{n+1}$ in \eqref{eq update the control bregman msa} shows that  for all $s\in[0,T]$,
\begin{equation*}
\begin{split}
&\int\frac{\delta H^{\tau}_s}{\delta m}(\Theta_s(\pi^{n}), \pi^{n}_s, a)(\pi^{n+1}_s-\pi^{n}_s)(da)+\lambda\,D_h(\pi^{n+1}_s|\pi^{n}_s)
\\
&\le \int\frac{\delta H^{\tau}_s}{\delta m}(\Theta_s(\pi^{n}), \pi^{n}_s, a)(\pi^{n}_s-\pi^{n}_s)(da)+\lambda\,D_h(\pi^{n}_s|\pi^{n}_s)
= 0\,.
\end{split}
\end{equation*}
Thus for all $n\in \N\cup\{0\}$, 
\begin{equation}
\label{eq proof energy diss 1}
\begin{split}
&J^\tau(\pi^{n+1})-J^\tau(\pi^{n})\le (L-\lambda) \E\int_0^TD_h(\pi^{n+1}_s|\pi^{n}_s)\,ds\,,
\end{split}
\end{equation}
which along with $D_h(\pi^{n+1}_s|\pi^{n}_s)\ge 0$ and $\lambda\ge L$ shows that 
$J^\tau(\pi^{n+1})-J^\tau(\pi^{n})\le 0$.

We further   re-write~\eqref{eq proof energy diss 1} as 
\begin{equation}
\begin{split}
\E\int_0^TD_h(\pi^{n+1}_s|\pi^{n}_s)\,ds \leq (\lambda - L)^{-1}(J^\tau(\pi^{n}) - J^\tau(\pi^{n+1}))\,.
\end{split}
\end{equation}
Summing up over $n=0,1,\ldots,m-1$ and observing the telescoping sum yields 
\begin{equation}
\begin{split}
\sum_{n=0}^{m-1} \E\int_0^TD_h(\pi^{n+1}_s|\pi^{n}_s)\,ds 
& \leq (\lambda - L)^{-1}(J^\tau(\pi^{0}) - J^\tau(\pi^{m}))
 \leq (\lambda - L)^{-1}\left(J^\tau(\pi^{0}) - \inf_{\pi \in\mathcal A_{\mathfrak C}} J^\tau(\pi)\right)<\infty\,.
\end{split}
\end{equation}
Thus 
by $\inf_{\pi \in\mathcal A_{\mathfrak C}} J^\tau(\pi)>-\infty$,
$\sum_{n=0}^\infty \E\int_0^TD_h(\pi^{n+1}_s|\pi^{n}_s)\,ds < \infty$ which yields the desired conclusion.
\end{proof}

\subsection{Proof of Theorem~\ref{th relative convexity}}
\label{sec:proof-rel-convexity}

\begin{proof}[Proof of Theorem~\ref{th relative convexity}]
Fix   $\pi,\pi'\in\mathcal{A_{\mathfrak C}}$. 
By the definition of $J^\tau$,
\begin{equation*}
\begin{split}
 J^{\tau}(\pi)-J^{\tau}(\pi') 
&=\E\int_{0}^{T}\left(f_s(X_s(\pi),\pi_s)-f_s(X_s(\pi'),\pi_s^{'})+\tau\,h(\pi_s)-\tau\,h(\pi_s^{'})\right)\,ds + \E\left[g(X_T(\pi))-g(X_T(\pi'))\right]\,.
\end{split}
\end{equation*}
Applying  the convexity of $g$ and   It\^o's formula to $t\mapsto Y_t(\pi)^\top X_t(\pi)$ give that 
\begin{equation*}
\begin{split}
  J^{\tau}(\pi)-J^{\tau}(\pi') 
&\le\E\int_{0}^{T}\left( f_s(X_s(\pi),\pi_s)-f_s(X_s(\pi'),\pi_s')+\tau\,h(\pi_s)-\tau\,h(\pi_s')\right) \,ds\\
&\qquad +\E\left[(D_xg)(X_T(\pi))^\top (X_T(\pi)-X_T(\pi'))\right]\\
&\le\E\int_{0}^{T}\left( 
f_s(X_s(\pi),\pi_s)-f_s(X_s(\pi'),\pi_s')+\tau\,h(\pi_s)-\tau\,h(\pi_s')\right) ds\\
&\qquad +\E\left[\int_0^T(Y_s(\pi))^\top d(X_s(\pi)-X_s(\pi'))+\int_0^T(X_s(\pi)-X_s(\pi'))^\top dY_s(\pi)\right]\\
&\qquad +\E\int_0^T\operatorname{tr}((\sigma_s(X_s(\pi),\pi_s)-\sigma_s(X_s(\pi'),\pi'_s))^\top Z_s(\pi))\,ds\,.
\end{split}
\end{equation*}
This implies that 
\begin{equation*}
\begin{split}
& J^{\tau}(\pi)-J^{\tau}(\pi') 
\le \E\int_{0}^{T}\left( 
f_s(X_s(\pi),\pi_s)-f_s(X_s(\pi'),\pi_s')+\tau\,h(\pi_s)-\tau\,h(\pi_s')\right) ds\\
& \,\, +\E\int_0^T(Y_s(\pi))^\top (b_s(X_s(\pi),\pi_s)-b_s(X_s(\pi'),\pi_s'))\,ds\\
& \,\, - \E\int_0^T(X_s(\pi)-X_s(\pi'))^\top (D_xH^0_s)(\Theta_s(\pi^{n}),\pi^{n}_s)\,ds
+\E\int_0^T\operatorname{tr}((\sigma_s(X_s(\pi),\pi_s)-\sigma_s(X_s(\pi'),\pi'_s))^\top Z_s(\pi))\,ds\,,
\end{split}
\end{equation*}
where   It\^{o} integrals vanish under the expectation due to   Lemma \ref{lem Y Z bounded}. 
Observe that
\begin{equation*}
\begin{split}
& H^{\tau}_s(\Theta_s(\pi), \pi_s)  
=
f_s(X_s(\pi),\pi_s) + \tau\,h(\pi_s)
+ (Y_s(\pi))^\top b_s(X_s(\pi),\pi_s) + \operatorname{tr}(\sigma_s(X_s(\pi),\pi_s)^\top Z_s(\pi))\,,
\end{split}
\end{equation*}
and
\begin{equation*}
\begin{split}
&  H^{\tau}_s(X_s(\pi'),Y_s(\pi),Z_s(\pi), \pi_s') 
=f_s(X_s(\pi'),\pi_s') + \tau\,h(\pi_s') + (Y_s(\pi))^\top b_s(X_s(\pi'),\pi_s') + \operatorname{tr}(\sigma_s(X_s(\pi'),\pi'_s)^\top Z_s(\pi))\,.
\end{split}
\end{equation*}
Thus,
\begin{equation}\label{eq lq inequality with star}
\begin{split}
J^{\tau}(\pi)-J^{\tau}(\pi') &\le \E\int_0^T\Big[H^{\tau}_s(\Theta_s(\pi),\pi_s)-H^{\tau}_s(X_s(\pi'),Y_s(\pi),Z_s(\pi), \pi_s')\Big]\,ds \\
&\qquad - \E\left[\int_0^T(X_s(\pi)-X_s(\pi'))^\top (D_xH^0_s)(\Theta_s(\pi),\pi_s)\,ds\right]\,.
\end{split}
\end{equation}
By Assumption~\ref{ass on convexity of hamiltonian}, 
\begin{equation*}
\begin{split}
&\E\int_0^T\Big[H^{0}_s(\Theta_s(\pi),\pi_s)-H^{0}_s(X_s(\pi'),Y_s(\pi),Z_s(\pi), \pi_s')\Big]\,ds\\
&\le \E\left[\int_0^T(X_s(\pi)-X_s(\pi'))^\top (D_xH^0_s)(\Theta_s(\pi),\pi_s)\,ds\right] + \E\int_0^T\int\frac{\delta H^{0}_s}{\delta m}(\Theta_s(\pi^{n}), \pi_s, a)(\pi_s-\pi'_s)(da)\,ds\,.
\end{split}
\end{equation*}
Substituting the above inequality to 
\eqref{eq lq inequality with star} yields
\begin{equation*} 
\begin{split}
  J^{\tau}(\pi)-J^{\tau}(\pi') 
 & \le \E\int_0^T\int\frac{\delta H^{0}_s}{\delta m}(\Theta_s(\pi), \pi_s, a)(\pi_s-\pi'_s)(da)\,ds
+\tau \E\int_0^T[h(\pi_s)-h(\pi'_s)]\,ds
\\
&= \E\int_0^T \int\frac{\delta H^{\tau}_s}{\delta m}(\Theta_s(\pi), \pi_s, a)(\pi_s-\pi'_s)(da)\,ds - \tau\E\int_{0}^{T}\left[D_h(\pi'_s|\pi_s) \right]\,ds\,.
\end{split}
\end{equation*}
where the last identity used 
\begin{equation*}
\E\int_{0}^{T}\left[h(\pi_s) - h(\pi'_s)\right]\,ds 
=\E\int_{0}^{T}\left[-D_h(\pi'_s|\pi_s) + \int \frac{\delta h}{\delta m}(\pi_s,a)(\pi_s-\pi'_s)(da)\right]\,ds\,.
\end{equation*}
This completes the proof. 
\end{proof}

\subsection{Proof of Theorem~\ref{thr convergence of modified MSA for convex case}}

\label{sec:proof-of-convergence}
\begin{proof}[Proof of Theorem~\ref{thr convergence of modified MSA for convex case}]
For all  $s\in [0,T]$ and   $\omega \in \Omega$,  define 
\[
\mathfrak C \ni m \mapsto G_s(m) = \frac1\lambda\int \frac{\delta H^\tau_s}{\delta m}(\Theta_s(\pi), \pi_s, a)(m - \pi_s)(da) \in \mathbb R\,,
\]
which  is linear and hence convex. 
The three point lemma applied to $G_s$ (see Lemma~\ref{lem three point} or~\cite[Lemma 2]{aubin2022mirror})
implies 
for all  $s\in [0,T]$, $\omega \in \Omega$ and $\pi,\pi'\in \mathfrak C$, 
\begin{equation*}
\begin{split}
&  \int\frac{\delta H^{\tau}_s}{\delta m}(\Theta_s(\pi), \pi_s, a)( \pi'_s-\pi_s)(da) + \lambda\,D_h(  \pi'_s|\pi_s)
\\
& \ge 
\int\frac{\delta H^{\tau}_s}{\delta m}(\Theta_s(\pi), \pi_s, a)(\overline \pi_s-\pi_s)(da)+  \lambda \,D_h(\pi'_s|\overline \pi_s)  
+ \lambda\,D_h(\overline \pi_s|\pi_s) \,,
\end{split}
\end{equation*}
provided that    $\overline \pi$  satisfies 
$$
\overline \pi \in \argmin_{m\in \mathfrak C}\int\frac{\delta H^{\tau}_s}{\delta m}(\Theta_s(\pi), \pi_s, a)(m -\pi_s)(da) + \lambda\,D_h(m|\pi_s)\,.
$$
Hence let 
$\pi'\in \mathcal{A}_{\mathfrak C}$ such that 
$\E\int_0^TD_h(\pi'_s|\pi^{n}_s)\,ds<\infty$
for all $n\in \N\cup\{0\}$.
By taking $\pi=\pi^n$ and $\overline \pi=\pi^{n+1}$,
and using 
Theorem \ref{th estimate for difference of J},  
\eqref{eq update the control bregman msa}
and the fact that $L\le \lambda $, 
\begin{equation} \label{eq three point plus smoothness}
\begin{split}
J^\tau(\pi^{n+1}) & \le 
J^\tau(\pi^{n}) +   \E\int_0^T\int\frac{\delta H^{\tau}_s}{\delta m}(\Theta_s(\pi^{n}), \pi^{n}_s, a)(\pi^{n+1}_s-\pi^{n}_s)(da)\,ds
+ L \E\int_0^TD_h(\pi^{n+1}_s|\pi^{n}_s)\,ds  
\\
&\le  J^\tau(\pi^{n}) +   \E\int_0^T\int\frac{\delta H^{\tau}_s}{\delta m}(\Theta_s(\pi^{n}), \pi^{n}_s, a)(\pi'_s-\pi^{n}_s)(da)\,ds\\
&\quad+(L-\lambda)\E\int_0^TD_h(\pi^{n+1}_s|\pi^{n}_s)\,ds + \lambda\E\int_0^TD_h(\pi'_s|\pi^{n}_s)\,ds 
- \lambda\E\int_0^TD_h(\pi'_s|\pi^{n+1}_s)\,ds\\
&\leq J^\tau(\pi^{n}) +   \E\int_0^T\int\frac{\delta H^{\tau}_s}{\delta m}(\Theta_s(\pi^{n}), \pi^{n}_s, a)(\pi'_s-\pi^{n}_s)(da)\,ds\\
&\quad + \lambda\E\int_0^TD_h(\pi'_s|\pi^{n}_s)\,ds - \lambda\E\int_0^TD_h(\pi'_s|\pi^{n+1}_s)\,ds\,.
\end{split}
\end{equation}
By Theorem \eqref{th relative convexity} and the assumption that 
$\E\int_{0}^{T}D_h(\pi^{\ast}_s|\pi^n_s)\,ds<\infty$, 
\begin{equation*}
\begin{split}
J^{\tau}(\pi') - \tau\E\int_{0}^{T}D_h(\pi'_s|\pi^n_s)\,ds  
&
\geq J^{\tau}(\pi^n) + \E\int_0^T\int\frac{\delta H^{\tau}_s}{\delta m}(\Theta_s(\pi^n), \pi^n_s, a)(\pi'_s-\pi^n_s)(da)\,ds  \,.
\end{split}
\end{equation*}
Plugging this into \eqref{eq three point plus smoothness} gives 
\begin{equation}\label{eq key inequlity conv}
\begin{split}
J^\tau(\pi^{n+1})  -  J^{\tau}(\pi') &\leq \lambda \left(\E\int_0^TD_h(\pi'_s|\pi^{n}_s)\,ds - \E\int_0^TD_h(\pi'_s|\pi^{n+1}_s)\,ds\,   \right)   - \tau\E\int_{0}^{T}D_h(\pi'_s|\pi^n_s) \,ds\,. 
\end{split}
\end{equation}
When 
$\tau=0$,  \eqref{eq key inequlity conv} implies
\begin{equation}
\sum_{n=0}^{m-1} (J^\tau(\pi^{n+1})  -  J^{\tau}(\pi^{\ast})) \leq \lambda \left(\E\int_0^TD_h(\pi'_s|\pi^{0}_s)\,ds - \E\int_0^TD_h(\pi'_s|\pi^{m}_s)\,ds\,   \right) \,.
\end{equation}
By
Theorem~\ref{thm energy dissipation mirr},
$J^\tau(\pi^{m})\leq J^\tau(\pi^{n+1})$ for $n=0,\ldots,m-1$, and hence 
\begin{equation}
m \left( J^\tau(\pi^{m})  -  J^{\tau}(\pi') \right) \leq \lambda \E\int_0^TD_h(\pi'_s|\pi^{0}_s)\,ds  \,.
\end{equation}

Let $\tau>0$ and 
fix $\pi^\ast\in \argmin_{\pi\in \mathcal A_{\mathfrak C}}J^\tau (\pi)$
such that 
$\E\int_0^TD_h(\pi^\ast_s|\pi^{n}_s)\,ds<\infty$ for all $n\in \N\cup\{0\}$.
Setting 
$\pi'=\pi^\ast$ in \eqref{eq key inequlity conv} and using  $ J^\tau(\pi^{n+1})  \ge   J^{\tau}(\pi^{\ast})$  implies
\begin{equation}
\E\int_0^TD_h(\pi^\ast_s|\pi^{n+1}_s)\,ds  \leq \left(1- \frac{\tau}{\lambda}\right)\E\int_{0}^{T} D_h(\pi^{\ast}_s|\pi^n_s) \,ds \,. 
\end{equation}
Iterating this inequality for all $m$  times gives 
\begin{equation}
\label{}
\E\int_0^TD_h(\pi^\ast_s|\pi^{m}_s)\,ds  \leq \left(1- \frac{\tau}{\lambda}\right)^m\E\int_{0}^{T} D_h(\pi^{\ast}_s|\pi^0_s) \,ds \,. 
\end{equation}
By 
\eqref{eq key inequlity conv},
\begin{equation}
J^\tau(\pi^{m+1})  -  J^{\tau}(\pi^{\ast}) \leq \lambda \E\int_0^TD_h(\pi^\ast_s|\pi^{m}_s)\,ds \,.    
\end{equation}
This concludes the proof.
\end{proof}

\subsection{Proofs of Propositions \ref{prop:relative_entropy},  \ref{prop:chi_2} and \ref{prop:eot}}
\label{sec:example_proof}

To prove Proposition \ref{prop:relative_entropy},
we   first show that  the relative entropy has a flat derivative on the set of probability measures whose  
log-density admits a polynomial growth.

\begin{lemma}
\label{lemma:derivative_of_entropy_on_C}
Let $p\in \N\cup\{0\}$, 
let 
$(A, \rho_A)$ be a metric space, 
let  $\mathcal{P}_p(A)$ be  the set of probability   measures on $A$  with finite $p$-moments,
and 
let  $B_p(A)$ be the space of measurable functions $f :A \rightarrow \mathbb R$
such that  $\sup_{a\in A }\frac{|f(a)|}{1+\rho_A(a_0,a)^p}<\infty$ for some $a_0\in A$. 
Let 
$\lambda \in \mathcal P_p(A)$,
let    $\mathfrak C = \{m \in \mathcal P_p(A) \mid 
m\ll \lambda, 
\ln \frac{\mathrm d m}{\mathrm d \lambda}\in B_p( A)\}$
and let 
$\operatorname{KL}(\cdot|\cdot):\mathfrak C\times \mathfrak C \to \R$ be     
such that 
\begin{equation*}
\operatorname{KL}(\mu |\nu) \coloneqq \int \ln\frac{\mathrm d\mu }{\mathrm d\nu}(a)   \mu (da),\,\quad  \mu,\nu\in  \mathfrak C\, .	
\end{equation*}
Then for all $\mu, \nu \in \mathfrak C$ and $a\in A$,
\begin{equation}
\frac{\delta \operatorname{KL}(\cdot|\nu)}{\delta m}(\mu,a) = \ln \frac{\mathrm d\mu}{\mathrm d\nu}(a)- \operatorname{KL}(\mu|\nu)\,.	
\end{equation}	
\end{lemma}
\begin{proof}
Note that $\mathfrak C$ is convex,
and 
for all $\mu, \mu'\in  \mathfrak C$, 
$\int \big|\ln \frac{\mathrm d\mu}{\mathrm d\nu}(a)\big|\mu'(da)<\infty$ as well as 
$\int \big(\ln \frac{\mathrm d\mu}{\mathrm d\nu}(a)- \operatorname{KL}(\mu|\nu)\big)\mu(da)=0$.
Let $\mu'\in \mathfrak C$, and for each $\varepsilon\in (0,1)$, 
let $\mu^\varepsilon = \mu + \varepsilon(\mu'-\mu)$. 
It remains to  show that
\begin{equation}
\label{eq entropy der}
\lim_{\varepsilon \to 0} \frac1\varepsilon \left(\operatorname{KL}(\mu^\varepsilon|\nu) - \operatorname{KL}(\mu|\nu)\right)
= \int \ln \frac{\mathrm d\mu}{\mathrm d\nu}(a) (\mu' - \mu)(da)\,. 	
\end{equation}

We   first prove that 
\begin{equation}
\label{eq entropy der lb}
\limsup_{\varepsilon\searrow 0}  \frac1\varepsilon \left(\operatorname{KL}(\mu^\varepsilon|\nu)-\operatorname{KL}(\mu|\nu)\right) 
\geq \int \ln \frac{\mathrm d\mu}{\mathrm d\nu}(a) (\mu' - \mu)(da)\,.	
\end{equation}
Indeed, for all $\varepsilon \in (0,1)$,
\[
\begin{split}
& \frac1\varepsilon \left(\operatorname{KL}(\mu^\varepsilon|\nu)-\operatorname{KL}(\mu|\nu)\right) 
=  \frac1\varepsilon\int\left[
\frac{\mathrm d\mu^\varepsilon}{\mathrm d\nu}(a)\ln \frac{\mathrm d\mu^\varepsilon}{\mathrm d\nu}(a) - \frac{\mathrm d\mu}{\mathrm d\nu}(a)\ln \frac{\mathrm d\mu}{\mathrm d\nu}(a)
\right]\nu(da) \\
& = \frac1\varepsilon \int \bigg[\bigg(\frac{\mathrm d\mu^\varepsilon}{\mathrm d\nu}(a) - 	\frac{\mathrm d\mu}{\mathrm d\nu}(a)\bigg)\ln \frac{\mathrm d\mu}{\mathrm d\nu}(a) + \frac{\mathrm d\mu^\varepsilon}{\mathrm d\nu}(a)\bigg(\ln \frac{\mathrm d\mu^\varepsilon}{\mathrm d\nu}(a)-\ln \frac{\mathrm d\mu}{\mathrm d\nu}(a)\bigg)\bigg]\nu(da)\\
& = \int \ln \frac{\mathrm d\mu}{\mathrm d\nu}(a)(\mu'-\mu)(da) + \frac1\varepsilon \int \frac{\mathrm d \mu^\varepsilon}{\mathrm d \mu}(a) \ln \frac{\mathrm d \mu^\varepsilon}{\mathrm d \mu}(a)\mu(da)\,.
\end{split}
\]
Since $x\ln x \geq x -1$ for $x> 0$, it holds   for all $\varepsilon\in (0,1)$ that 
\[
\begin{split}
  \frac1\varepsilon \left(\operatorname{KL}(\mu^\varepsilon|\nu)-\operatorname{KL}(\mu|\nu)\right)  & \geq \int \ln \frac{\mathrm d\mu}{\mathrm d\nu}(a)(\mu'-\mu)(da) + \frac1\varepsilon \int \bigg(\frac{\mathrm d \mu^\varepsilon}{\mathrm d \mu}(a) -1\bigg)\mu(da) \\
& = \int \ln \frac{\mathrm d\mu}{\mathrm d\nu}(a)(\mu'-\mu)(da) + \frac1\varepsilon \bigg( \int \mu^\varepsilon(da) - \int \mu(da)\bigg)\,.
\end{split}
\]
Taking the limit inferior for $\varepsilon\to 0$ yields
\eqref{eq entropy der lb}.

Next we   show that 
\begin{equation}
\label{eq entropy der limsup ub}
\limsup_{\varepsilon \searrow 0}\frac1\varepsilon \left(\operatorname{KL}(\mu^\varepsilon|\nu)-\operatorname{KL}(\mu|\nu)\right) 
\leq \int \ln \frac{\mathrm d\mu}{\mathrm d\nu}(a) (\mu' - \mu)(da)\,.	
\end{equation} 
To prove~\eqref{eq entropy der limsup ub},
note that  
\begin{equation}
\label{eq entropy der limsup ub 1}
\begin{split}
& \frac1\varepsilon \left(\operatorname{KL}(\mu^\varepsilon|\nu) - \operatorname{KL}(\mu|\nu)\right) \\
& = \frac {1}{\varepsilon} 
\int \left[\big(\ln \frac{\mathrm d \mu^\varepsilon}{\mathrm d \lambda}(a) -  \ln\frac{\mathrm d \nu}{\mathrm d\lambda}(a)\big)\frac{\mathrm d\mu^\varepsilon}{\mathrm d\lambda}(a)  - \big(\ln \frac{\mathrm d \mu}{\mathrm d \lambda}(a) - \ln \frac{\mathrm d \nu}{\mathrm d \lambda}(a)\big) \frac{\mathrm d \mu}{\mathrm d \lambda}(a)\right] \lambda(da)
 = \int f_\varepsilon(a) \,\lambda(da)\,,
\end{split}
\end{equation}
where
$$
f_\varepsilon (a) = \frac1\varepsilon\left[\frac{\mathrm d \mu^\varepsilon}{\mathrm d \lambda} \ln \frac{\mathrm d \mu^\varepsilon}{\mathrm d \lambda}(a) - \frac{\mathrm d \mu}{\mathrm d \lambda}\ln \frac{\mathrm d \mu}{\mathrm d \lambda}(a) - \ln \frac{\mathrm d \nu}{\mathrm d \lambda}(a)\big( \frac{\mathrm d \mu^\varepsilon}{\mathrm d \lambda} - \frac{\mathrm d \mu}{\mathrm d \lambda}(a)\big)    \right]\,.
$$
Observe that  
\begin{equation}
\label{eq entropy der limsup ub 2}
- \frac1\varepsilon \ln \frac{\mathrm d \nu}{\mathrm d \lambda} (a)\left(\frac{\mathrm d \mu^\varepsilon}{\mathrm d \lambda}(a) - \frac{\mathrm d \mu}{\mathrm d \lambda}(a)\right) 
= -\ln \frac{\mathrm d \nu}{\mathrm d \lambda}(a)\left(\frac{\mathrm d \mu'}{\mathrm d \lambda}(a) - \frac{\mathrm d \mu}{\mathrm d \lambda}(a)\right)\,.	
\end{equation}
Moreover, using the convexity of   $a\mapsto a\ln a$ on $(0,\infty)$ and the definition of $\mu^\varepsilon$,  
\begin{equation}
\label{eq entropy der limsup ub 3}	
\frac1\varepsilon\left[
\frac{\mathrm d \mu^\varepsilon}{\mathrm d \lambda}(a)\ln\frac{\mathrm d \mu^\varepsilon}{\mathrm d \lambda}(a) - \frac{\mathrm d \mu}{\mathrm d \lambda}(a)\ln \frac{\mathrm d \mu}{\mathrm d \lambda}(a)
\right] 
\leq \frac{\mathrm d \mu'}{\mathrm d \lambda}(a)\ln\frac{\mathrm d \mu'}{\mathrm d \lambda}(a) - \frac{\mathrm d \mu}{\mathrm d \lambda}(a)\ln \frac{\mathrm d \mu}{\mathrm d \lambda}(a)\,.
\end{equation}
This implies that 
$f_\varepsilon \le  g$ with $g\coloneqq  \frac{\mathrm d \mu'}{\mathrm d \lambda}\ln\frac{\mathrm d \mu'}{\mathrm d \lambda} - \frac{\mathrm d \mu}{\mathrm d \lambda}\ln \frac{\mathrm d \mu}{\mathrm d \lambda}
-\ln \frac{\mathrm d \nu}{\mathrm d \lambda} (\frac{\mathrm d \mu'}{\mathrm d \lambda} - \frac{\mathrm d \mu}{\mathrm d \lambda})
$, and   
\begin{align*}
\begin{split}
& \int |g(a)|\lambda (da) 
= \int \left|\frac{\mathrm d \mu'}{\mathrm d \lambda}(a)\ln\frac{\mathrm d \mu'}{\mathrm d \lambda}(a) - \frac{\mathrm d \mu}{\mathrm d \lambda}(a)\ln \frac{\mathrm d \mu}{\mathrm d \lambda}(a) -\ln \frac{\mathrm d \nu}{\mathrm d \lambda}(a) (\frac{\mathrm d \mu'}{\mathrm d \lambda}(a) - \frac{\mathrm d \mu}{\mathrm d \lambda}(a)) \right|\lambda(da) \\
&\le 
\int 
\left(\left| 
\frac{\mathrm d \mu'}{\mathrm d \lambda}(a)\left(\ln\frac{\mathrm d \mu'}{\mathrm d \lambda}(a) -
\ln \frac{\mathrm d \nu}{\mathrm d \lambda}(a)\right)\right|
+
\left|  \frac{\mathrm d \mu}{\mathrm d \lambda}(a)
\left(\ln \frac{\mathrm d \mu}{\mathrm d \lambda}(a) -\ln \frac{\mathrm d \nu}{\mathrm d \lambda}(a) \right) \right|
\right)\lambda(da) \\
&\le  
\int  \left|\ln\frac{\mathrm d \mu'}{\mathrm d \lambda}(a) -
\ln \frac{\mathrm d \nu}{\mathrm d \lambda}(a)\right| \mu'(da)
+
\int \left| 
\ln \frac{\mathrm d \mu}{\mathrm d \lambda}(a) -\ln \frac{\mathrm d \nu}{\mathrm d \lambda}(a)  \right|\mu(da)
< \infty\,,
\end{split}	
\end{align*}
as $\nu, \mu,\mu'\in \mathfrak C$.
Hence combining \eqref{eq entropy der limsup ub 1} and the reverse Fatou's lemma yields 
\begin{equation}
\label{eq entropy der limsup ub 4}
\begin{split}
& \limsup_{\varepsilon\searrow 0} \frac1\varepsilon \left(\operatorname{KL}(\mu^\varepsilon|\nu) - \operatorname{KL}(\mu|\nu)\right) \\
& \leq \int  \limsup_{\varepsilon\searrow 0} \frac1\varepsilon\left[\frac{\mathrm d \mu^\varepsilon}{\mathrm d \lambda} \ln \frac{\mathrm d \mu^\varepsilon}{\mathrm d \lambda}(a) - \frac{\mathrm d \mu}{\mathrm d \lambda}\ln \frac{\mathrm d \mu}{\mathrm d \lambda}(a) - \ln \frac{\mathrm d \nu}{\mathrm d \lambda}(a)\big( \frac{\mathrm d \mu^\varepsilon}{\mathrm d \lambda} - \frac{\mathrm d \mu}{\mathrm d \lambda}(a)\big)    \right]\,\lambda(da)\\
& = \int 
\left( \limsup_{\varepsilon\searrow 0} \frac1\varepsilon\left [\frac{\mathrm d \mu^\varepsilon}{\mathrm d \lambda} \ln \frac{\mathrm d \mu^\varepsilon}{\mathrm d \lambda}(a) - \frac{\mathrm d \mu}{\mathrm d \lambda}\ln \frac{\mathrm d \mu}{\mathrm d \lambda}(a)\right] -\ln \frac{\mathrm d \nu}{\mathrm d \lambda}(a) (\frac{\mathrm d \mu'}{\mathrm d \lambda}(a) - \frac{\mathrm d \mu}{\mathrm d \lambda}(a))
\right)\,\lambda(da)\,.
\end{split}
\end{equation}
Using  the derivative of $x\mapsto x \log x$ at $x>0$,
\[
\lim_{\varepsilon \to 0} \frac1\varepsilon\left( \frac{\mathrm d \mu^\varepsilon}{\mathrm d \lambda}(a) \ln \frac{\mathrm d \mu^\varepsilon}{\mathrm d \lambda}(a) - \frac{\mathrm d \mu}{\mathrm d\lambda}(a)\ln \frac{\mathrm d \mu}{\mathrm d\lambda}(a) \right)  
= (1 + \ln \frac{\mathrm d \mu}{\mathrm d\lambda}(a))(\frac{\mathrm d \mu'}{\mathrm d\lambda}(a)-\frac{\mathrm d \mu}{\mathrm d\lambda}(a))\,,
\]
which together with 
\eqref{eq entropy der limsup ub 4} implies that 
\[
\begin{split}
& \limsup_{\varepsilon\searrow 0} \frac1\varepsilon \left(\operatorname{KL}(\mu^\varepsilon|\nu) - \operatorname{KL}(\mu|\nu)\right) 
\leq 	\int \left( \ln \frac{\mathrm d \mu}{\mathrm d\lambda}(a)(\frac{\mathrm d \mu'}{\mathrm d\lambda}(a)-\frac{\mathrm d \mu}{\mathrm d\lambda}(a))-\ln \frac{\mathrm d \nu}{\mathrm d \lambda}(a) (\frac{\mathrm d \mu'}{\mathrm d \lambda}(a) - \frac{\mathrm d \mu}{\mathrm d \lambda}(a))
\right)\,\lambda(da)\\
& = \int \left[\ln \frac{\mathrm d \mu}{\mathrm d\lambda}(a) - \ln \frac{\mathrm d \nu}{\mathrm d \lambda}(a) \right] (\mu'-\mu)(da)\,.
\end{split}
\]
This proves    \eqref{eq entropy der limsup ub}.
Combining \eqref{eq entropy der lb} and  \eqref{eq entropy der limsup ub}
yields  
\eqref{eq entropy der} and finishes the proof. 
\end{proof}

\begin{proof}[Proof of Proposition \ref{prop:relative_entropy}]
The flat derivative  of $h$ on $\mathfrak C$ has been shown in   Lemma~\ref{lemma:derivative_of_entropy_on_C}. 
The  associated Bregman divergence   is  given by
\[
\begin{split}
D_h(m'|m)
& = h(m')-h(m)-\int \frac{\delta h}{\delta m}(m,a)(m'-m)(da)\\
& = \int \ln \frac{\mathrm d m'}{\mathrm d \varrho }(a)m'(da) - \int \ln \frac{\mathrm d m}{\mathrm d \varrho }(a)m(da) - \int\ln \frac{\mathrm d m}{\mathrm d \varrho }(a)(m'-m)(da)
 = \int \ln \frac{\mathrm d m'}{\mathrm d m }(a)m'(da)\,. 
\end{split}
\]

We now  assume  that   $\pi^0\in \mathcal{A}_{\mathfrak C}$ satisfies $\mathbb E\int_0^T \big\| \ln \frac{\mathrm d \pi^0_t}{\mathrm d \varrho} \big\|_{L^\infty(A)} dt<\infty$
and  verify Assumption \ref{ass iterate_wd}.
By \cite[Lemma 1.4.3]{dupuis1997weak}, 
$\pi^1$ in \eqref{eq update the control bregman msa}
is given by
\begin{equation*}
\pi_t^{1}(da)= \frac{\exp{\left(-\frac1\lambda \frac{\delta H^{\tau}_t}{\delta m}(\Theta_t(\pi^0),  \pi^0_t, a) \right)}}{\int  \exp{\left(-\frac1\lambda \frac{\delta H^{\tau}_t}{\delta m}(\Theta_t(\pi^0),  \pi^0_t, a') \right)\pi^0_t(da')}}\pi^0_t(da)\,,
\end{equation*}
which is progressively measurable. 
For each $t\in [0,T]$,
\begin{align*}
\ln \frac{\mathrm d \pi^1_t}{\mathrm d \varrho}(a) 
& =-\frac1\lambda \frac{\delta H^{\tau}_t}{\delta m}(\Theta_t(\pi^0),  \pi^0_t, a)
-\ln\left(\int  \exp{\left(-\frac1\lambda \frac{\delta H^{\tau}_t}{\delta m}(\Theta_t(\pi^0),  \pi^0_t, a') \right)\pi^0_t(da')}\right)
+\ln \frac{\mathrm d \pi^0_t }{\mathrm d \varrho }(a)\,. 
\end{align*}
Taking the sup-norm over $a$ on both sides and using 
the definition of $H^\tau$ in  \eqref{eq: hamiltonian} and 
the boundedness of $\frac{\delta b}{\delta m}, 
\frac{\delta \sigma}{\delta m}$ and $\frac{\delta f}{\delta m}$ show   that there exists $C\ge 0$ such that 
\begin{align*}
\left\| \ln \frac{\mathrm d \pi^1_t}{\mathrm d \varrho}\right\|_{L^\infty(A)}
& \le C\bigg(1+|Y_t(\pi^0)|+|Z_t(\pi^0)|
+
\left\|   \ln \frac{\mathrm d \pi^0_t }{\mathrm d\varrho }\right\|_{L^\infty(A)} \bigg)  \,,
\end{align*}
which together  with the  square integrability of $Y(\pi^0)$ and $Z(\pi^0)$ implies  that 
$\mathbb E\int_0^T \big\| \ln \frac{\mathrm d \pi^1_t}{\mathrm d \varrho} \big\|_{L^\infty(A)} dt<\infty$.
For fixed $\nu\in \mathfrak C$, 
the fact that 
$\mathbb E\int_0^T  D_h(\pi^1_t|\nu) \, d t < \infty$ follows from  
\begin{align*}
\mathbb{E}\int_0^T h(\pi^1_t)dt
&= \mathbb E\int_0^T  \int\left|   \ln \frac{\mathrm d \pi^1_t}{\mathrm d \varrho}(a) \right| \pi^1(da) dt 
\le \mathbb E\int_0^T \left\| \ln \frac{\mathrm d \pi^1_t}{\mathrm d \varrho} \right\|_{L^\infty(A)} dt<\infty\,,
\end{align*}
and   
\begin{align*}
\left|\mathbb E\int_0^T  \int \frac{\delta h}{\delta m }(\nu,a)\pi^1_t (da) dt \right|
&=
\left|\mathbb E\int_0^T  \int \ln \frac{\delta \nu}{\delta \varrho }(a)\pi^1_t (da) dt \right|
\le E\int_0^T  \left\| \ln \frac{\mathrm d \nu}{\mathrm d \varrho} \right\|_{L^\infty(A)}dt <\infty\, .
\end{align*}
Finally, 
\begin{align*}
\mathbb E\int_0^T  D_h( \pi^1_t|\pi^0_t) dt  
& =\mathbb E\int_0^T \left(  \int \left(\ln \frac{\mathrm d \pi^1_t}{\mathrm d \varrho }(a)-\ln \frac{\mathrm d \pi^0_t}{\mathrm d \varrho }\right) \pi^1_t(da) \right)dt  
\\
& \le  \mathbb E\int_0^T    \left(\left\|  \ln \frac{\mathrm d \pi^1_t}{\mathrm d \varrho }\right\|_{L^\infty(A)}+\left\|\ln \frac{\mathrm d \pi^0_t}{\mathrm d \varrho }\right\|_{L^\infty(A)} \right) dt <\infty\,.
\end{align*}
This verifies Assumption \ref{ass iterate_wd} for $n=1$ and also proves that  $\mathbb E\int_0^T \big\| \ln \frac{\mathrm d \pi^1_t}{\mathrm d \varrho} \big\|_{L^\infty(A)} dt<\infty$.
An inductive argument shows that 
Assumption \ref{ass iterate_wd} is satisfied for all $n\in \N\cup\{0\}$. 
\end{proof}

\begin{proof}[Proof of Proposition \ref{prop:chi_2}]
Let $m,m'\in  \mathfrak C$ and for each $\varepsilon\in (0,1)$, 
let $m^\eps = m+\eps (m'-m)$. 
Note that for each  $\varepsilon\in (0,1)$, 
\begin{align*}
h(m^\eps)&= \int \frac{1}{2} \left(\frac{\mathrm d m }{\mathrm d\varrho }(a) -1+\eps \frac{\mathrm d (m'-m) }{\mathrm d\varrho }(a) \right) ^2   \varrho (da)
\\
&= h(m) +  \int  \left(\frac{\mathrm d m }{\mathrm d\varrho }(a) -1\right)  \eps \frac{\mathrm d (m'-m) }{\mathrm d\varrho }(a)    \varrho (da)
 +\eps^2   \int \frac{1}{2} \left( \frac{\mathrm d (m'-m) }{\mathrm d\varrho }(a) \right) ^2   \varrho (da)\,
\end{align*}
which along with $m,m'\in \mathcal P(A)$  implies that for all   $c\in \R$, 
\begin{align*}
\lim_{\eps\to 0}\frac{h(m^\eps)-h(m)}{\eps}&=    \int  \left(\frac{\mathrm d m }{\mathrm d\varrho }(a) -c \right)  
(m'-m) (da)\,.
\end{align*}
As $h(m) = \int \frac{1}{2} \left(\frac{\mathrm d m }{\mathrm d\varrho }(a) \right) ^2   \varrho (da)-1$, 
the divergence $D_h(m'|m)$ is given by
\begin{align*}
D_h(m'|m) &=   \int \frac{1}{2} 
\left[
\left(\frac{\mathrm d m' }{\mathrm d\varrho }(a) \right) ^2 
-\left(\frac{\mathrm d m }{\mathrm d\varrho }(a) \right) ^2  
-2   \frac{\mathrm d m }{\mathrm d\varrho }(a)  
\frac{\mathrm d (m'-m) }{\mathrm d\varrho }(a)  
\right]
\varrho (da)
\\
&
=   \int \frac{1}{2} 
\left(\frac{\mathrm d m' }{\mathrm d\varrho }(a) -\frac{\mathrm d m }{\mathrm d\varrho }(a) \right) ^2 
\varrho (da)  \,.
\end{align*}

We now  assume  that   $\pi^0\in \mathcal{A}_{\mathfrak C}$ satisfies    $\mathbb E\int_0^T\big \|   \frac{\mathrm d \pi^0_t}{\mathrm d \varrho} \big\|^2_{L^2_\varrho(A)} dt<\infty$ 
and  verify Assumption \ref{ass iterate_wd}.
Observe that by the integrability assumptions of  $\frac{\delta b}{\delta m}, 
\frac{\delta \sigma}{\delta m}$, $\frac{\delta f}{\delta m}$ and $\frac{\mathrm d \pi^0}{\mathrm d \varrho}$, 
for a.e.~$(\omega,t)\in  \Omega\times [0,T]$,
$\frac{\delta H^{\tau}_t}{\delta m}(\Theta_t(\pi^0), \pi^{0}_t, \cdot)\in L^2_\varrho(A)$,
and    
$\pi^1$ in \eqref{eq update the control bregman msa}
satisfies 
\begin{equation*} 
\pi_t^{1}= 
\argmin_{m\in \mathfrak C}\int
\left(
\frac{\delta H^{\tau}_t}{\delta m}(\Theta_t(\pi^0), \pi^{0}_t, a)\frac{\mathrm d m}{\mathrm d\varrho}( a) +  
\frac{\lambda}{2} 
\left(\frac{\mathrm d m}{\mathrm d\varrho }(a) -\frac{\mathrm d \pi^0_t }{\mathrm d\varrho }(a) \right) ^2 
\right)   \varrho (da)\,.
\end{equation*}
The first order condition (see e.g., \cite[Section 5.1.1]{bonnans2013perturbation}) shows that for a.e.~$(\omega,t)\in [0,T]\times \Omega$,
\begin{equation} 
\label{eq:first-order_L2}
\left\langle 
\frac{\delta H^{\tau}_t}{\delta m}(\Theta_t(\pi^0), \pi^{0}_t, \cdot) +  
\lambda 
\left(\frac{\mathrm d \pi^1_t}{\mathrm d\varrho }  -\frac{\mathrm d \pi^0_t }{\mathrm d\varrho }  \right),
\phi-  \frac{\mathrm d \pi^1_t}{\mathrm d\varrho } 
\right\rangle_{L^2_\varrho}\ge 0,  
\quad \forall \phi\in \mathcal C\,, 
\end{equation}
where $\left\langle \cdot, \cdot \right\rangle_{L^2_\varrho} $ is the inner product on $L^2_\varrho(A)$,
and $\mathcal C$ is the nonempty closed convex set defined by 
$$
\mathcal C=\left\{\phi\in L^2_\varrho(A)\bigg\vert \textnormal{$\phi\ge 0$ $\varrho$-a.s.~on $A$
and $\int \phi(a)\varrho(da)=1$}\right\}\,.
$$
Define the projection map 
$\Pi_{\mathcal C}: L^2_\varrho(A)\mapsto \mathcal C$ such that 
$\Pi_{\mathcal C}(\varphi)=\argmin_{\phi\in \mathcal C}\|\phi-\varphi\|_{L^2_\varrho(A)}$ for all $\varphi \in L^2_\varrho(A)$, 
which satisfies 
$$
\left\langle 
\Pi(\varphi)-\varphi,
\phi-  \Pi(\varphi)
\right\rangle_{L^2_\varrho}\ge 0,  
\quad \forall \phi\in \mathcal C\,. 
$$
Then 
$$
\frac{\mathrm d \pi^1}{\mathrm d \varrho} = \Pi_{\mathcal C}\left( \frac{\mathrm d \pi^0_t }{\mathrm d\varrho }  -\frac{1}{\lambda }\frac{\delta H^{\tau}_t}{\delta m}(\Theta_t(\pi^0), \pi^{0}_t, \cdot)  \right)\,.
$$
As $\|\Pi_{\mathcal C}(\varphi_1)-\Pi_{\mathcal C}(\varphi_2)\|_{L^2_\varrho(A)}
\le \|\varphi_1-\varphi_2\|_{L^2_\varrho(A)}$ for all $\varphi_1,\varphi_2\in L^2_\varrho(A)$
(see e.g., \cite[Theorem 4.3-1]{ciarlet2013linear}), 
$\pi^1$ is     progressively measurable.
Moreover, since $\frac{\mathrm d \pi^0_t}{\mathrm d \varrho} = \Pi_{\mathcal C}\left(\frac{\mathrm d \pi^0_t}{\mathrm d \varrho}\right)$,
there exists $C\ge 0$ such that for a.e.~$(\omega,t)\in  \Omega\times [0,T] $,
\begin{align*}
\left \|   \frac{\mathrm d \pi^1_t}{\mathrm d \varrho} \right\|_{L^2_\varrho(A)} 
& \le  
\left \| \frac{\mathrm d \pi^0_t }{\mathrm d\varrho }\right\|_{L^2_\varrho(A)} + \frac{1}{\lambda }
\left \| \frac{\delta H^{\tau}_t}{\delta m}(\Theta_t(\pi^0), \pi^{0}_t, \cdot)\right\|_{L^2_\varrho(A)}
 \le  
C\left(1+ |Y_t(\pi^0)|+|Z_t(\pi^0)|+  \left \| \frac{\mathrm d \pi^0_t }{\mathrm d\varrho }\right\|_{L^2_\varrho(A)} \right)\,,
\end{align*}
where the last inequality used  
the definition of $H^\tau$ in  \eqref{eq: hamiltonian},
the flat derivative of $h$ at $\pi^0_t$ 
and  
the integrability assumptions of  $\frac{\delta b}{\delta m}, 
\frac{\delta \sigma}{\delta m}$ and $\frac{\delta f}{\delta m}$.
Using the   square integrability of $Y(\pi^0)$ and $Z(\pi^0)$ shows   that 
$\mathbb E\int_0^T\big \|   \frac{\mathrm d \pi^1_t}{\mathrm d \varrho} \big\|^2_{L^2_\varrho(A)} dt<\infty$.
The fact that  
$\mathbb E\int_0^T  D_h(\pi^1_t|\pi^0) \, d t < \infty$
and 
for fixed $\nu\in \mathfrak C$, 
$\mathbb E\int_0^T  D_h(\pi^1_t|\nu) \, d t < \infty$ follows directly from the definition of the Bregman divergence. 
This verifies Assumption \ref{ass iterate_wd} for $n=1$ 
and proves 
$\mathbb E\int_0^T\big \|   \frac{\mathrm d \pi^1_t}{\mathrm d \varrho} \big\|^2_{L^2_\varrho(A)} dt<\infty$.
An inductive argument shows that 
Assumption \ref{ass iterate_wd} is satisfied for all $n\in \N\cup\{0\}$. 
\end{proof}

\begin{proof}[Proof of Proposition \ref{prop:eot}]
Under the assumptions of $c$ and $A$, 
by \cite[Proposition 2]{feydy2019interpolating},
$m\to h(m)$ is continuous  (with respect to the weak topology)
and admits the   flat derivative given in the statement. 

We now  verify Assumption \ref{ass iterate_wd}.
To this end, consider the map 
$F: [0,T]\times \mathbb{R}^d\times \mathbb{R}^d\times  \mathbb{R}^{d\times d'}
\times \mathfrak C \times \mathfrak C \to \R$,
\begin{align}
\begin{split}
F(t,x,y,z,m, m') 
&=\int\frac{\delta H^{\tau}_t}{\delta m}(x,y,z,  m, a) m'(da) +    \lambda \left(h(m')-\int \phi [m](a)m'(da) \right)
\\
&=\int
\left(
\frac{\delta H^{0}_t}{\delta m}(x,y,z,  m, a) + (\tau-\lambda) \phi [m](a)
\right) m'(da)
+ \lambda  h(m') \,,
\end{split}
\end{align}
where the last identity used the definition of $H^\tau$ in  \eqref{eq: hamiltonian}.
By \cite[Proposition 13]{feydy2019interpolating}, 
the map $\mathfrak C\ni  \mu\mapsto \phi[\mu]\in C(A)$  is continuous,
where $\mathfrak C=\mathcal P(A)$
is equipped with the topology of  weak convergence of measures and
$C(A)$ is equipped with the topology of uniform convergence of functions. 
By the measurability of $\frac{\delta H^0}{\delta m}$ and the continuity of $\mathfrak C\ni  \mu\mapsto \phi[\mu]\in C(A)$, 
for all $m'\in \mathfrak C $, 
$(t,x,y,z,m)\mapsto F(t,x,y,z,m, m') $ is measurable. 
Moreover, by the continuity of $a\mapsto \frac{\delta H^0_t}{\delta m}(x,y,z,m,a)$ and the continuity of $m\mapsto h(m)$, 
for all $(t,x,y,z,m ) \in [0,T]\times \mathbb{R}^d\times \mathbb{R}^d\times  \mathbb{R}^{d\times d'}
\times \mathfrak C $, 
$m'\mapsto F(t,x,y,z,m, m') $ is continuous. 
This proves  that $F$ is a Carath\'{e}odory function. 
Since $A$ is   compact and separable, 
$\mathfrak C = \mathcal P(X)$ is a compact and  separable metrisable space (see e.g., 
\cite[Section 15]{guide2006infinite}).
Hence,  
by \cite[Theorem 18.19]{guide2006infinite}, 
there exists a measurable function $\mathfrak f:[0,T]\times \mathbb{R}^d\times \mathbb{R}^d\times  \mathbb{R}^{d\times d'}
\times \mathfrak C \to \mathfrak C$ such that for all 
$t\in [0,T]$, $(x,y,z)\in \mathbb{R}^d\times \mathbb{R}^d\times  \mathbb{R}^{d\times d'}$ 
and $m\in \mathfrak C$,
\begin{align}
\label{eq:maximum_selector}
\mathfrak f(t,x,y,z,m)\in 
\argmin_{m' \in \mathfrak C} F(t,x,y,z,m,m')\,.
\end{align}

Given    $\pi^0\in \mathcal{A}_{\mathfrak C}$, it is easy to see that 
$\pi^1_t \coloneqq \mathfrak f(t,X_t(\pi^0),Y_t(\pi^0), Z_t(\pi^0),\pi^0_t)$ 
for a.e.~$(\omega,t)\in  \Omega\times [0,T] $ satisfies \eqref{eq update the control bregman msa}. 
Note that the measurability of  $\mathfrak f$ implies the measurability of $\pi^1$. 
Moreover, by the compactness of $\mathfrak C$ and the continuity of $h$,
for a given $\nu \in \mathfrak C$,
\begin{align*}
\mathbb E\int_0^T  D_h(\pi_t|\nu) \, d t 
\le 2 T
\left(
\sup_{m\in \mathfrak C }h(m)+\|\phi [\nu]\|_{L^\infty} \right)< \infty\,,
\end{align*}
and similarly 
$$\mathbb E\int_0^T  D_h(\pi^1_t|\pi^0_t) \, d t
\le
2 T
\sup_{m\in \mathfrak C }\left(h(m)+\|\phi [m]\|_{L^\infty} \right)
<\infty,
$$
where the last inequality used the compactness of 
$\mathfrak C$ and the continuity of $\mathfrak C\ni  \mu\mapsto \phi[\mu]\in C(A)$.
An inductive argument allows for verifying  
Assumption \ref{ass iterate_wd}  for all $n\in \N\cup\{0\}$. 
\end{proof}

\section*{Acknowledgements}
Supported by the Alan Turing Institute under EPSRC grant no.~EP/N510129/1.

The authors would like to thank the anonymous reviewers for their insightful comments.
These have helped to greatly improve the revised version of the article.

\appendix

\section{Properties for  Bregman divergence over measures}
\label{appendix:bregman}

This section collects   basic properties of Bregman divergence used in this paper.
Throughout this section, we fix
a convex and measurable set  $\mathfrak C\subset \mathcal P(A)$
and a convex function 
$h:\mathfrak C\to \R$ that has    a flat derivative $\frac{\delta h}{\delta m}: \mathfrak C \times A\to \R$. 
Define 
$D_h(\cdot|\cdot):\mathfrak C\times \mathfrak C\to \R$ as in \eqref{eq:bregman_def}.

The following properties of $D_h$ follow from the definition. 
\begin{lemma} 
\label{lemma:properties-of-h-bregman}
For  all $m,m' \in \mathfrak C$,
\begin{enumerate}
\item 
\label{item:D_positive}
$ 
D_h(m'|m) \geq 0
$.
\item 
the function $\mathfrak C \ni m\mapsto D_h(m|m')\in \R$ is  convex and has a flat derivative
\[ \label{item:D_convex}  
\frac{\delta D_h(\cdot|m')}{\delta m}
=\frac{\delta h}{\delta m}(m,a)- \frac{\delta h}{\delta m}(m',a)\,.
\]
\item   
for all convex functions $g:\mathfrak C\to \R$  with   flat derivatives and all $\alpha,\beta \in \R$,
\[ \label{item: additivity}
D_{\alpha h+\beta g}(m'|m)
=\alpha  D_{h }(m'|m)+\beta  D_{g }(m'|m)\,.
\]
\item  for all $ \nu \in \mathfrak C $,   
\[ \label{item:breg of breg}
D_{D_h(\cdot|\nu)}(m'|m ) = D_h(m'| m)\,.
\]

\end{enumerate}
\end{lemma}

\begin{proof}
For  Item \ref{item:D_positive},  note that for all $0<s<t\le 1$,
$m+s (m'-m) =\frac{s}{t}(m+t (m'-m))+  \left( 1-\frac{s}{t}\right)  m  $, and hence by the convexity of $h$,
$$
h(m+s (m'-m))    \le \frac{s}{t}h(m+t (m'-m) ) + \left( 1-\frac{s}{t}\right)    h(m) \,,
$$
which implies that 
$[0,1]\ni \varepsilon\mapsto \frac{h(m+\varepsilon (m'-m))-h(m)}{\varepsilon}\in \R$  is increasing.
Hence   by the definition of $\frac{\delta h}{\delta m}$,
$$
{h(m' )-h(m)}\ge \lim_{\varepsilon \searrow  0}\frac{h(m+\varepsilon (m'-m))-h(m)}{\varepsilon}= \int \frac{\delta h}{\delta m}(m'-m)(da)\,,
$$
which proves  
$D_h(m'|m)\ge 0$. 

Items \ref{item:D_convex} and   \ref{item: additivity} follow  directly from the convexity and differentiability of $h$ and the definition of $D_h(\cdot|\nu)$.

For Item \ref{item:breg of breg},
by the definitions of $D_{D_h(\cdot|\nu)}$
and $\frac{\delta D_h(\cdot|\nu)}{\delta m}$
in Item \ref{item: additivity}, 
\begin{equation}
\begin{split}
 D_{D_h(\cdot|\nu)}(m'|m )  
 &= D_{h}(m'|\nu) - D_h( m| \nu ) 
- \int \frac{\delta D_h(\cdot|\nu)}{\delta m}(  m, a  )(m'-  m)(da) \\
&= D_{h}(m'|\nu) - D_h( m| \nu ) 
- \int \left(\frac{\delta h}{\delta m}(m,a)- \frac{\delta h}{\delta m}(\nu,a)\right)(m'-  m)(da)
\\
&=  h(m') - h(\nu) - \int \frac{\delta h}{ \delta m}(\nu,a)(m'-  \nu)(da) \\
&\quad  - 
\left(
h(m)  - h(\nu) - \int \frac{\delta h}{ \delta m}(\nu,a)(m -  \nu)(da) 
\right)\\
&\quad  - 
\int \frac{\delta h}{\delta m}(m,a)(m'-  m)(da)
+ \frac{\delta h}{\delta m}(\nu,a)(m'-  m)(da)
\\
&=  h(m') - h(m)  -  \int \frac{\delta h}{\delta m}(m,a)(m'-  m)(da)
=D_h(m'|m)\,.
\end{split}
\end{equation}
This completes the proof.
\end{proof}

We then prove a ``three point lemma" for the Bregman divergence $D_h(\cdot|m)$.

\begin{lemma} 
\label{lem three point}
Let  $\nu\in \mathfrak C $
and 
let 
$G:\mathfrak C \rightarrow \mathbb R$ be    convex and have  a flat derivative $\frac{\delta G}{\delta m}$.
Suppose that 
there exists  $\overline m\in \mathfrak C$ such that 
$ \overline m \in \argmin_{m\in \mathfrak C}
\left( G(m) + D_h(m|\nu) \right)$. 
Then for all  $ m'\in \mathfrak C$,
\begin{equation}
G(m') + D_h(m'|\nu) \geq G(\overline m ) + D_h(m'|\overline m ) + D_h(\overline m|  \nu)\,.
\end{equation}
\end{lemma}

\begin{proof}

Fix $m'\in  \mathfrak C$. 
Define 
$\mathcal G: \mathfrak C\to \R$ by 
$\mathcal G(m) = G(m) + D_h(m|\nu) $
for all $m\in \mathfrak C$.
The convexity and differentability of $G$ and $h$ imply that 
$\mathcal G$ has a flat derivative. 
As $\overline m$ minimises $\mathcal G$ over $\mathfrak C$,
\[
\int \frac{\delta \mathcal G}{\delta m}(\overline m,a)(m'- \overline m)(da) 
=\lim_{\varepsilon \searrow  0}\frac{ \mathcal G (\overline{m}+\varepsilon (m'-\overline{m} ))- \mathcal G(\overline{m})}{\varepsilon}\ge 0\,,
\]
which implies that 
\begin{align*}
D_{\mathcal G}(m'| \overline m) = \mathcal G(m') -\mathcal G (\overline m ) -  \int \frac{\delta \mathcal G}{ \delta m} (\overline m ,a)(m' - \overline m )(da) \le   \mathcal G(m') - \mathcal G (\overline m )\,.
\end{align*}
Adding $\mathcal G (\overline m )$ to both sides of the inequality and 
using the definition of $\mathcal G$ yield that 
\begin{align*}
G(m')+D_h(m'|\nu) 
&\ge 
D_{\mathcal G}(m'| \overline m)  +G(\overline m) +D_h( \overline m |\nu) 
\\
&=
D_{  G}(m'| \overline m) +D_{ D_h(\cdot|\nu)}(m'| \overline m)  +G(\overline m) +D_h( \overline m |\nu) 
\\
&\ge 
D_{  h}(m'| \overline m)  +G(\overline m) +D_h( \overline m |\nu)\,,
\end{align*}
where the first identity used 
Lemma \ref{lemma:properties-of-h-bregman} Item \ref{item: additivity} and 
the last inequality used 
Lemma \ref{lemma:properties-of-h-bregman} Items \ref{item:D_positive} and 
\ref{item:breg of breg}.
This completes the proof.
\end{proof}

\section{Proof of Lemma~\ref{lem lemma form siska szpruch}}
\label{sec:proof_of_dir_der_of_J0}

Throughout this section, fix 
$\pi,\pi' \in \mathcal A_{\mathfrak C}$  such that $\mathbb E \int_0^T D_h (\pi'_t|\pi_t)\,dt < \infty$.
For each $\varepsilon\in [0,1]$, let $\pi^\varepsilon = \pi + \varepsilon(\pi'-\pi)$, which lies in $ \mathcal A_{\mathfrak C}$ due to 
the convexity of  $\mathcal A_{\mathfrak C}$.  
For each $\varepsilon\in [0,1]$, let  $ X(\pi^\varepsilon) $ be a  solution to the state dynamics  \eqref{sde} controlled by $\pi^\eps$. Consider the following SDE:
for all $t\in [0,T]$,
\begin{equation}
\label{eq V proc linear} 
\begin{split}
dV_t & = \left[(D_x b_t)(X_t,\pi_t)V_t + \int \frac{\delta b_t}{\delta m}(X_t,\pi_t,a)(\pi'_t - \pi_t)(da) \right]	\,dt \\
&\quad   + \left[(D_x \sigma_t)(X_t,\pi_t)V_t + \int \frac{\delta \sigma_t}{\delta m}(X_t,\pi_t,a)(\pi'_t - \pi_t)(da) \right]\,dW_t\,;\quad  V_0 = 0\,.
\end{split}	
\end{equation} 
Under Assumptions \ref{assumption controlled SDE for modified MSA} and \ref{assum:regularity_flat_derivative},
$D_x b$ and $D_x \sigma$ are uniformly bounded, and 
\begin{align*}
&\left|\int \frac{\delta b_t}{\delta m}(X_t,\pi_t,a)(\pi'_t - \pi_t)(da) \right|^2   
+\left|\int \frac{\delta \sigma_t}{\delta m}(X_t,\pi_t,a)(\pi'_t - \pi_t)(da) \right|^2   
\\
&\le 2K D_h(\pi'_t|\pi_t)\,.
\end{align*}
This, along with $\mathbb E \int_0^T D_h (\pi'_t|\pi_t)\,dt < \infty$ and \cite[Theorem 3.4.3]{zhang2017backward} implies that 
$\E[\sup_{t\in [0,T]}|V_t|^2]<\infty$.

The following lemma proves that  $V$ is the derivative of  $\varepsilon\mapsto X^\eps$. 

\begin{lemma}
\label{lemma:dir_der_control}
Suppose that Assumptions \ref{assumption controlled SDE for modified MSA}, \ref{assumption:spatial_derivative} and \ref{assum:regularity_flat_derivative} hold. 
Then 
\[
\lim_{\varepsilon \searrow 0}\mathbb E \left[\sup_{t\in[0,T]} \left|\frac{X_t(\pi^\varepsilon) - X_t(\pi)}{\varepsilon} - V_t \right|^2\right] = 0\,.
\]
\end{lemma}

\begin{proof}
To simplify the notation,
let $X=X(\pi)$ and 
for each $\varepsilon\in (0,1)$, let 
$X^\eps=X(\pi^\varepsilon)$ and 
$V^\eps =\frac{X^\eps - X}{\varepsilon} - V$. 
We also assume without loss of generality that all processes are one-dimensional. 
For each $\varepsilon\in (0,1)$,
using \eqref{sde} and \eqref{eq V proc linear}, we have 
$V^\varepsilon=0$ and for all $t\in [0,T]$, 
\begin{equation}
\label{eq X_eps-X-V} 
\begin{split}
dV^\varepsilon_t 
& = \bigg(\frac{b_t(X^\eps_t,\pi^\eps_t)-b_t(X_t,\pi_t)}{\varepsilon}
-(D_x b_t)(X_t,\pi_t)V_t - \int \frac{\delta b_t}{\delta m}(X_t,\pi_t,a)(\pi'_t - \pi_t)(da) \bigg)	\,dt \\
&   + \bigg( \frac{\sigma_t(X^\eps_t,\pi^\eps_t)-\sigma_t(X_t,\pi_t)}{\varepsilon}
- (D_x \sigma_t)(X_t,\pi_t)V_t + \int \frac{\delta \sigma_t}{\delta m}(X_t,\pi_t,a)(\pi'_t - \pi_t)(da) \bigg)\,dW_t\,.
\end{split}	
\end{equation} 
Note that using $X^\eps =X+\eps (V^\eps+V)$
\[
\begin{split}
& b_t(X^\varepsilon_t, \pi^\varepsilon_t) - b_t(X_t,\pi_t) 
= b_t(X^\varepsilon_t, \pi^\varepsilon_t) - b_t(X^\varepsilon_t, \pi_t) + b_t(X^\varepsilon_t,\pi_t) - b_t(X_t,\pi_t) \\
& = \varepsilon\int_0^1 \int \frac{\delta b_t}{\delta m}(X^\varepsilon_t, \pi_t +\lambda (\pi^\varepsilon_t -  \pi_t),a)(\pi'_t-\pi_t)(da)\,d\lambda
 + \varepsilon \int_0^1 (D_x b_t)(X_t + \lambda (X^\varepsilon_t -  X_t), \pi_t) (V^\varepsilon_t + V_t)\,d\lambda 
\,. 	
\end{split}
\]  
Hence, by rearranging the terms,
\[
\begin{split}
& 
\frac{b_t(X^\eps_t,\pi^\eps_t)-b_t(X_t,\pi_t)}{\varepsilon}
-(D_x b_t)(X_t,\pi_t)V_t - \int \frac{\delta b_t}{\delta m}(X_t,\pi_t,a)(\pi'_t - \pi_t)(da)
\\
& = \left( \int_0^1 (D_x b_t)(X_t + \lambda (X^\varepsilon_t -  X_t), \pi_t)\,d\lambda  \right) V^\varepsilon_t 
+   I^{(1)}_t + I^{(2)}_t\,,  
\end{split}
\]
where 
\begin{align}
I^{(1,\eps)}_t&= \int_0^1 \left[(D_x b_t)(X_t + \lambda (X^\varepsilon_t -  X_t), \pi_t) - (D_x b_t)(X_t,\pi_t) \right] V_t\,d\lambda\,,
\\
I^{(2,\eps)}_t
&= 
\int_0^1 \int \left[ \frac{\delta b_t}{\delta m}(X^\varepsilon_t, \pi_t +\lambda (\pi^\varepsilon_t -  \pi_t),a) - \frac{\delta b_t}{\delta m}(X^\varepsilon_t, \pi_t,a) \right](\pi'_t-\pi_t)(da)\,d\lambda \,.
\end{align}
Similarly, we have  
\[
\begin{split}
& 
\frac{\sigma_t(X^\eps_t,\pi^\eps_t)-\sigma_t(X_t,\pi_t)}{\varepsilon}
-(D_x \sigma_t)(X_t,\pi_t)V_t - \int \frac{\delta \sigma_t}{\delta m}(X_t,\pi_t,a)(\pi'_t - \pi_t)(da)
\\
& = \left( \int_0^1 (D_x \sigma_t)(X_t + \lambda (X^\varepsilon_t -  X_t), \pi_t)\,d\lambda  \right) V^\varepsilon_t 
+   J^{(1)}_t + J^{(2)}_t\,,  
\end{split}
\]
where 
\begin{align}
J^{(1,\eps)}_t&= \int_0^1 \left[(D_x \sigma_t)(X_t + \lambda (X^\varepsilon_t -  X_t), \pi_t) - (D_x \sigma_t)(X_t,\pi_t) \right] V_t\,d\lambda\,,
\\
J^{(2,\eps)}_t
&= 
\int_0^1 \int \left[ \frac{\delta \sigma_t}{\delta m}(X^\varepsilon_t, \pi_t +\lambda (\pi^\varepsilon_t -  \pi_t),a) - \frac{\delta \sigma_t}{\delta m}(X^\varepsilon_t, \pi_t,a) \right](\pi'_t-\pi_t)(da)\,d\lambda \,.
\end{align}
Then by \cite[Theorem 3.4.3]{zhang2017backward} and the uniform boundedness of $D_x b$ and $D_x \sigma$, 
\begin{align}
\label{eq:V_eps_estimate}
\begin{split}
 \mathbb E \left[\sup_{t\in[0,T]} \left| V^\eps_t \right|^2\right] 
& \le C\left[\left(\int_0^T (|I^{(1,\eps)}_t|+|I^{(2,\eps)}_t|)dt\right)^2
+ \int_0^T ( |J^{(1,\eps)}_t|^2+|J^{(2,\eps)}_t|^2)dt 
\right]
\\
& \le C\left[\int_0^T \left( |I^{(1,\eps)}_t|^2+|I^{(2,\eps)}_t| ^2
+|J^{(1,\eps)}_t|^2+|J^{(2,\eps)}_t|^2\right)dt
\right]\,.
\end{split}
\end{align}
Above and hereafter,  we denote by $C$ a generic constant independent of $\varepsilon$.

It remains to prove the terms on the right-hand side of \eqref{eq:V_eps_estimate} converges to zero as $\eps\to 0$.
Observe that by Lemma \ref{lem standard sde estimate} and the convexity of $m\mapsto D_h(m|\pi_s)$,
\begin{equation*}
\E\left[ \sup_{0\le t\le T}|X^\eps_t -X_t |^2\right]\le C\E\int_0^T\,D_h(\pi^\eps_s|\pi_s)\,ds
\le \eps  C\E\int_0^T\,D_h(\pi'_s|\pi_s)\,ds \,,
\end{equation*}
which along with $\E\int_0^T\,D_h(\pi'_s|\pi_s)\,ds<\infty$ implies that 
$\lim_{\eps\searrow 0}\E\left[ \sup_{t\in [0,T] }|X^\eps_t -X_t |^2\right]=0$.
Thus by subtracting a subsequence if necessary, we can assume without loss of generality that 
$\lim_{\eps\searrow 0} X^\eps =X$ for a.e.~$(\omega ,t)\in \Omega\times [0,T]$.
This along with the continuity of $D_x b$   in $x$ implies that for a.e.~$(\omega ,t)\in \Omega\times [0,T]$,
$$
\lim_{\eps\searrow 0}|(D_x b_t)(X_t + \lambda (X^\varepsilon_t -  X_t), \pi_t) - (D_x b_t)(X_t,\pi_t)|=0\,.
$$
As $D_x b$ are bounded by $K$, 
$ 
\sup_{\eps\in [0,1]}|I^{(1,\eps)}_t|^2\le 2K |V_t|^2$, 
and hence by the square integrability of $V$,  $\E\int_0^T \sup_{\eps\in [0,1]}|I^{(1,\eps)}_t|^2 dt <\infty$.
By Lebesgue's dominated convergence theorem,
$\lim_{\eps\searrow 0} \int_0^T  |I^{(1,\eps)}_t|^2 
dt
=0$. Similar argument shows that 
$\lim_{\eps\searrow 0} \int_0^T  |J^{(1,\eps)}_t|^2 
dt
=0$. 
For the term 
$I^{(2,\eps)}_t$, by Assumption \ref{assum:regularity_flat_derivative}
and Fubini's theorem, 
\begin{align*}
& |I^{(2,\eps)}_t|^2
= 
\left|\int_0^1 \int  \left( 
\int_0^1
\int  \frac{\delta^2 b_t}{\delta^2 m}(X^\varepsilon_t, \pi_t +\lambda \varepsilon' (\pi^\varepsilon_t -  \pi_t),a,a' ) \lambda (\pi^\eps_t-\pi_t)(da') d \eps' \right)(\pi'_t-\pi_t)(da)\,d\lambda \right|^2 
\\
&\le \frac{1}{2} \eps K D_h (\pi'_t|\pi_t)\,,
\end{align*}
and hence $\lim_{\eps\searrow 0} \int_0^T 
|I^{(2,\eps)}_t|^2dt =0$. 
Finally  
the condition that $\frac{\delta^2 \sigma}{\delta ^2 m}=0 $ implies that $J^{(2,\eps)}_t=0$ for all $\eps\in [0,1]
$. 
This along with 
\eqref{eq:V_eps_estimate} yields the desired conclusion. 
\end{proof}

The next lemma expresses the directional derivative of $J^0$ using the sensitivity process $V$.

\begin{lemma} 
\label{lemma:dir_der_J0_1}
Suppose  Assumptions~\ref{assumption controlled SDE for modified MSA}, \ref{assum:regularity_f_g},
\ref{assumption:spatial_derivative}
and 
\ref{assum:regularity_flat_derivative}. 
Then  
\[
\begin{split}
&\lim_{\eps\searrow 0}\frac{J^0(\pi^\eps)-J^0(\pi)}{\eps} 
 = \mathbb E \bigg[\int_0^T \bigg(\int \frac{\delta f_t}{\delta m}(X_t, \pi_t,a)(\pi'_t-\pi_t)(da) + (D_x f)(X_t,\pi_t) V_t\bigg)\,dt + (D_x g)(X_T)V_T \bigg]	\,.
\end{split}
\]
\end{lemma}

\begin{proof}
It suffices to prove 
$I_\varepsilon \coloneqq   I^{(1)}_\varepsilon + I^{(2)}_\varepsilon \to 0$ as $\varepsilon\to 0$, where
\begin{align*}
I^{(1)}_\varepsilon 
&\coloneqq 
\mathbb E  \int_0^T \bigg|\frac{ f_t(X^\varepsilon_t,\pi^\varepsilon_t) - f_t(X_t,\pi_t)}{\varepsilon} 
- \int \frac{\delta f_t}{\delta m}(X_t,\pi_t,a)\,(\pi'_t-\pi_t)(da) - (D_x f_t)(X_t,\pi_t) V_t  \bigg|  \,dt \, ,   
\end{align*}
and
\[
I^{(2)}_\varepsilon \coloneqq  \mathbb E \Big| \frac{g(X^\varepsilon_T)-g(X_T)}{\varepsilon}  - (D_x g)(X_T)V_T \Big|\,. 
\]
To estimate $I^{(1)}_\varepsilon$,  first note that by Assumption \ref{assum:regularity_f_g},
\[
\begin{split}
& 
\frac{f_t(X^\varepsilon_t,\pi^\varepsilon_t) - f_t(X_t,\pi_t)}{\varepsilon} 
=
\frac{f_t(X^{\varepsilon}_t, \pi^\varepsilon_t) - f_t(X^{\varepsilon}_t, \pi_t) + f_t(X^{\varepsilon}_t, \pi_t) - f_t(X_t, \pi_t)}{\varepsilon} 
\\
& =  \int_0^1 \int \frac{\delta f_t}{\delta m}(X^\varepsilon_t, \pi_t +\lambda (\pi^\eps _t -\pi_t), a )(\pi'_t - \pi_t)(da)	\,d\lambda
+ \int_0^1(D_x f_t)(X_t + \lambda (X^\eps_t-X_t),\pi_t)(V^{\varepsilon}_t - V_t)\,d\lambda\,,
\end{split}
\]
where $V^\eps =\frac{X^\eps -X}{\eps}-V $.
Hence we can  write $I^{(1)}_\varepsilon \leq  I^{(1,1)}_\varepsilon + I^{(1,2)}_\varepsilon + I^{(1,3)}_\varepsilon$, where
\begin{align*}
I^{(1,1)}_\varepsilon & \coloneqq \mathbb E \int_0^T \bigg|  \int_0^1 \int \bigg[\frac{\delta f_t}{\delta m}(X^\varepsilon_t,  \pi_t +\lambda (\pi^\eps _t -\pi_t), a )
-\frac{\delta f_t}{\delta m}(X^\varepsilon_t,\pi_t,a)\bigg]
(\pi'_t - \pi_t)(da)d\lambda \bigg|dt\,,
\\
I^{(1,2)}_\varepsilon & \coloneqq \mathbb E \int_0^T \bigg|  \int_0^1 \int \bigg[\frac{\delta f_t}{\delta m}(X^\varepsilon_t,\pi_t,a)-\frac{\delta f_t}{\delta m}(X_t,\pi_t,a)\bigg]
(\pi'_t - \pi_t)(da)d\lambda \bigg|dt\,,
\\
I^{(1,3)}_\varepsilon &\coloneqq  \mathbb E \int_0^T \bigg|\int_0^1(D_x f_t)(X_t + \lambda (X^\eps_t-X_t),\pi_t)(V^{\varepsilon}_t - V_t) \,d\lambda - (D_x f_t)(X_t,\pi_t) V_t \bigg|\,dt \,.
\end{align*}
Recall that Lemma \ref{lemma:dir_der_control} shows that 
$$
\lim_{\eps\searrow 0}
\left(\E\left[ \sup_{t\in [0,T] }|X^\eps_t -X_t |^2\right]
+\E\left[ \sup_{t\in [0,T] }|V^\eps_t |^2\right]\right)=0\,.
$$
Hence using Assumptions \ref{assum:regularity_f_g} and \ref{assum:regularity_flat_derivative}
and Lebesgue's dominated convergence theorem,
$\lim_{\eps\searrow 0} I^{(1)}=0$.
The convergence of $I^{(2)}$ can be established by similar arguments. 
This finishes the proof. 
\end{proof}

\begin{proof}[Proof of Lemma~\ref{lem lemma form siska szpruch}]

The proof follows exactly the same lines as those  of ~\cite[Lemma 3.5]{vsivska2020gradient},
by first applying It\^{o}'s formula to $Y_t^\top V_t$
and using    Lemma \ref{lemma:dir_der_J0_1} and the definition of $H^0$ in \eqref{eq: hamiltonian}. The details are omitted.  
\end{proof}

\section{Proof of Lemma~\ref{lem standard sde estimate}}
\label{sec proof of sde stability estimate}

\begin{proof}[Proof of Lemma~\ref{lem standard sde estimate}]
Fix $\pi,\pi'\in \mathcal{A}_{\mathfrak C} $. Observe that  for all $t\in [0,T]$, 
\begin{equation*}
\begin{split}
&X_t(\pi)-X_t(\pi')
=\int_0^t\big(b_s(X_s(\pi),\pi_s)-b(X_s(\pi'),\pi'_s)\big)\,ds+\int_0^t\big(\sigma_s(X_s(\pi),\pi_s)-\sigma_s(X_s(\pi'),\pi'_s)\big)\,dW_s\,.
\end{split}
\end{equation*}
Using  the inequality $(a+b)^2\le 2a^2+2b^2$ and   H\"older's inequality, 
\begin{equation*}
\begin{split}
&|X_t(\pi)-X_t(\pi')|^2\\
&=\left|\int_0^t\left(b_s(X_s(\pi),\pi_s)-b(X_s(\pi'),\pi'_s)\right)\,ds+\int_0^t\left(\sigma_s(X_s(\pi),\pi_s)-\sigma_s(X_s(\pi'),\pi'_s)\right)\,dW_s \right|^2\\
&\le 2\left|\int_0^t\left(b_s(X_s(\pi),\pi_s)-b(X_s(\pi'),\pi'_s)\right)\,ds\right|^2+2\left|\int_0^t\left(\sigma_s(X_s(\pi),\pi_s)-\sigma_s(X_s(\pi'),\pi'_s)\right)\,dW_s\right|^2\\
&\le 2t\int_0^t\left|b_s(X_s(\pi),\pi_s)-b(X_s(\pi'),\pi'_s)\right|^2\,ds+2\left|\int_0^t\left(\sigma_s(X_s(\pi),\pi_s)-\sigma_s(X_s(\pi'),\pi'_s)\right)\,dW_s\right|^2\,.
\end{split}
\end{equation*}
Let $t'\in [0,T]$.
Taking supremum over $t\in[0,t']$ yields 
\begin{equation*}
\begin{split}
\sup_{0\le t\le t'} |X_t(\pi)-X_t(\pi')|^2  
&\le 2T\int_0^{t'}\left|b_s(X_s(\pi),\pi_s)-b(X_s(\pi'),\pi'_s)\right|^2\,ds\\
&\quad+2\sup_{0\le t\le t'}\left|\int_0^t\left(\sigma_s(X_s(\pi),\pi_s)-\sigma_s(X_s(\pi'),\pi'_s)\right)\,dW_s\right|^2\,.
\end{split}
\end{equation*}
By taking the expectation on both sides and  using   the Burkholder-Davis-Gundy inequality, 
\begin{equation*}
\begin{split}
\E\left[ \sup_{0\le t\le {t'}}|X_t(\pi)-X_t(\pi')|^2\right] 
& \le 2T\E \left[ \int_0^{t'}\left|b_s(X_s(\pi),\pi_s)-b(X_s(\pi'),\pi'_s)\right|^2\,ds\right]\\
& \quad+2C \E\left[ \int_0^{t'}\left|\sigma_s(X_s(\pi),\pi_s)-\sigma_s(X_s(\pi'),\pi'_s)\right|^2\,ds\right]
\\
& \le 2TK \E\int_0^{t'}\left(|X_s(\pi)-X_s(\pi')|^2+D_h(\pi_s|\pi'_s)\right)\,ds\\
&\quad +2KC \E\int_0^{t'}\left(|X_s(\pi)-X_s(\pi')|^2+D_h(\pi_s|\pi'_s)\right)\,ds\,,
\end{split}
\end{equation*}
where the last inequality used 
the  assumption of the lemma.  
We thus have, with some $C>0$ depending on $K$ and $T$, for any $t'\in[0,T]$, that
$
y(t') \leq C\int_0^{t'} y(s)\,ds + b(t') 
$
where $y(t):=\E\left[ \sup_{0\le s\le t}|X_{s}(\pi)-X_{s}(\pi')|^2\right]$ and $b(t):= \int_0^t D_h(\pi_s|\pi'_s)\,ds$.
Hence, by Gr\"onwall's inequality, we have $y(T) \leq b(T)e^{CT}$.
In other words
\begin{equation*}
\begin{split}
\E\left[ \sup_{0\le t\le T}|X_t(\pi)-X_t(\pi')|^2\right]\le C\E\int_0^T\,D_h(\pi_s|\pi'_s)\,ds\,.
\end{split}
\end{equation*}
This finishes the proof. 
\end{proof}

\bibliographystyle{siam} 
\bibliography{Bibliography.bib}

\end{document}